\documentclass[12pt]{amsart}

\usepackage[utf8]{inputenc}
\usepackage[T1,T2A]{fontenc}
\usepackage[english]{babel}
\usepackage{amsfonts}
\usepackage{amsbsy}
\usepackage{amssymb}
\usepackage{fixltx2e,graphicx}
\usepackage{mathrsfs}
\usepackage{hyperref}
\usepackage{longtable}
\usepackage{multirow}
\usepackage{enumitem}
\usepackage{comment}
\usepackage{textcomp}

\catcode`~=11 % make LaTeX treat tilde (~) like a normal character
\newcommand{\urltilde}{\kern -.15em\lower .7ex\hbox{~}\kern .04em}
\catcode`~=13 % revert back to treating tilde (~) as an active character

\renewcommand{\abovecaptionskip}{0pt}
\renewcommand{\belowcaptionskip}{6pt}
\makeatletter

\renewcommand{\@makecaption}[2]{
\vspace{\abovecaptionskip}%
\sbox{\@tempboxa}{#1. #2}%
\global\@minipagefalse \hbox to \hsize {{\scshape \hfil #1.
#2\hfil}} \vspace{\belowcaptionskip}}

\makeatother

\newcommand{\rk}{\operatorname{rk}}
\newcommand{\Lie}{\operatorname{Lie}}

\newcommand{\GL}{\operatorname{GL}}
\newcommand{\SL}{\operatorname{SL}}
\newcommand{\Sp}{\operatorname{Sp}}
\newcommand{\Spin}{\operatorname{Spin}}
\newcommand{\SO}{\operatorname{SO}}
\newcommand{\Aut}{\operatorname{Aut}}

\newcommand{\SpFl}{\operatorname{SpFl}}
\newcommand{\SpGr}{\operatorname{SpGr}}
\newcommand{\SOFl}{\operatorname{SOFl}}
\newcommand{\SOGr}{\operatorname{SOGr}}
\newcommand{\Fl}{\operatorname{Fl}}

\newcommand{\ZZ}{\mathbb Z}

\newcommand{\FF}{\mathbb F}
\newcommand{\PP}{\mathbb P}

\newtheorem{theorem}{Theorem}

\newtheorem{proposition}[theorem]{Proposition}
\newtheorem{lemma}[theorem]{Lemma}
\newtheorem{corollary}[theorem]{Corollary}

\newtheorem*{question*}{Question}

\theoremstyle{definition}

\theoremstyle{remark}

\newtheorem{remark}[theorem]{Remark}

\makeatletter

\@addtoreset{theorem}{section}

\makeatother

\numberwithin{equation}{section}

\newcommand{\bl}{\mbox{$[\hspace{-0.35ex}[$}}
\newcommand{\br}{\mbox{$]\hspace{-0.35ex}]$}}

\makeatletter

\@addtoreset{theorem}{section}

\makeatother

\numberwithin{equation}{section}

\newcounter{num}[table]
\newcounter{subnum}[num]

\newcommand{\newcase}{\refstepcounter{num}\arabic{num}}
\newcommand{\no}{\refstepcounter{subnum}\arabic{num}.\arabic{subnum}}

\newcounter{case}
\newcommand{\caseno}{\refstepcounter{case}\arabic{case}}
\newcounter{casea}
\newcommand{\casenoa}{\refstepcounter{casea}\arabic{casea}}
\newcounter{caseb}
\newcommand{\casenob}{\refstepcounter{caseb}\arabic{caseb}}
\newcounter{casec}
\newcommand{\casenoc}{\refstepcounter{casec}\arabic{casec}}
\newcounter{cased}
\newcommand{\casenod}{\refstepcounter{cased}\arabic{cased}}
\newcounter{casee}
\newcommand{\casenoe}{\refstepcounter{casee}\arabic{casee}}
\newcounter{casef}
\newcommand{\casenof}{\refstepcounter{casef}\arabic{casef}}
\newcounter{caseg}
\newcommand{\casenog}{\refstepcounter{caseg}\arabic{caseg}}
\newcounter{caseh}
\newcommand{\casenoh}{\refstepcounter{caseh}\arabic{caseh}}
\newcounter{casei}
\newcommand{\casenoi}{\refstepcounter{casei}\arabic{casei}}
\newcounter{casej}
\newcommand{\casenoj}{\refstepcounter{casej}\arabic{casej}}

\begin{document}

\renewcommand{\proofname}{Proof}
\renewcommand{\abstractname}{Abstract}
\renewcommand{\refname}{References}
\renewcommand{\figurename}{Figure}
\renewcommand{\tablename}{Table}

\title[Spherical actions on isotropic flag varieties]
{Spherical actions on isotropic flag varieties\\ and related branching rules}

\author{Roman Avdeev and Alexey Petukhov}

%\thanks{}

\address{%
{\bfseries Roman Avdeev} \newline\indent National Research University ``Higher School of Economics'', Moscow, Russia}

\email{suselr@yandex.ru}

\address{%
{\bfseries Alexey Petukhov}
\newline\indent Institute for Information Transmission Problems, Moscow, Russia}

\email{alex{-}{-}2@yandex.ru}

%\date{\today}



\begin{abstract}
Let $G$ be a symplectic or special orthogonal group, let $H$ be a connected reductive subgroup of~$G$, and let $X$ be a flag variety of~$G$.
We classify all triples $(G,H,X)$ such that the natural action of~$H$ on~$X$ is spherical.
For each of these triples, we determine the restrictions to~$H$ of all irreducible representations of~$G$ realized in spaces of sections of homogeneous line bundles on~$X$.
\end{abstract}

\maketitle

\section{Introduction}

Throughout this paper, we work over an algebraically closed field $\FF$ of characteristic zero.
The notation $\FF^\times$ stands for the multiplicative group of~$\FF$.

Let $H$ be a connected reductive algebraic group and let $X$ be an $H$-variety (that is, an algebraic variety equipped with a regular action of~$H$).
The action of $H$ on~$X$, as well as $X$ itself, is said to be \textit{spherical} (or \textit{$H$-spherical} if one needs to emphasize the acting group) if $X$ possesses a dense open orbit with respect to the induced action of a Borel subgroup of~$H$.
When $X$ is a flag variety of a connected semisimple group $G$ containing~$H$ as a subgroup, a result of Vinberg and Kimelfeld \cite{VK78} asserts that $X$ is $H$-spherical if and only if, for every irreducible representation $R$ of~$G$ realized in the space of sections of a homogeneous line bundle on~$X$, the restriction of~$R$ to~$H$ is multiplicity free.
In view of the importance of the latter representation-theoretic property, this result naturally raises the following two problems:
\begin{enumerate}[label=\textup{(P\arabic*)},ref=\textup{P\arabic*}]
\item \label{P1}
classify all triples $(G,H,X)$ such that $X$ is $H$-spherical;

\item \label{P2}
for each such triple $(G,H,X)$ determine the restrictions to $H$ of all irreducible representations of~$G$ realized in spaces of sections of homogeneous line bundles on~$X$.
\end{enumerate}

By now, problem~(\ref{P1}) has been solved in the following particular cases (in all of them $G$ is assumed to be simple):
\begin{enumerate}[label=\textup{(C\arabic*)},ref=\textup{C\arabic*}]
\item \label{case_C1}
$H$ is a Levi subgroup of~$G$ (with contributions of \cite{Lit, MWZ1, MWZ2, Stem}, see also~\cite{Pon13});

\item \label{case_C2}
$H$ is a symmetric subgroup of~$G$ (see~\cite{HNOO});

\item \label{case_C3}
$G = \SL_n$ (see~\cite{AvP1});

\item \label{case_C4}
$G$ is an exceptional simple group, $H$ is a maximal reductive subgroup of~$G$, and $X = G/P$ with $P$ a maximal parabolic subgroup of~$G$ (see the preprint~\cite{Nie}).
\end{enumerate}

\noindent In all these cases, problem~(\ref{P2}) has also been already solved.
Namely, in case~(\ref{case_C1}) a solution follows from results of the papers~\cite{Pon15}, \cite{Pon17} (see details in~\cite{AvP2}), cases~(\ref{case_C2}) and~(\ref{case_C3}) were completed in~\cite{AvP2}, and case~(\ref{case_C4}) was settled in~\cite{Nie}.

In this paper, which may be regarded as a continuation of~\cite{AvP1} and~\cite{AvP2}, we solve problems~(\ref{P1}) and~(\ref{P2}) in the cases $G = \Sp_{2n}$ and $G = \SO_n$; our results are stated in \S\S\,\ref{subsec_sp_results},\,\ref{subsec_so_results}.
In particular, we complete solutions of problems~(\ref{P1}) and~(\ref{P2}) for all the classical simple algebraic groups.
Below we outline the main ideas of the employed approaches and obtained results.

To solve problem~(\ref{P1}), we apply a general strategy developed in~\cite{AvP1} (see \S\,\ref{subsec_NE-relation} for details).
Let $\mathscr F(G)$ be the set of all nontrivial flag varieties for~$G$ and fix an arbitrary connected reductive subgroup $H \subset G$.
Given $X \in \mathscr F(G)$ such that $X = G/P$ for a parabolic subgroup $P \subset G$, let $\mathscr N(X)$ denote the Richardson nilpotent orbit in $\mathfrak g = \Lie G$ defined by~$P$.
We say that two varieties $X_1,X_2 \in \mathscr F(G)$ are \textit{nil-equivalent} (notation $X_1 \sim X_2$) if $\mathscr N(X_1) = \mathscr N(X_2)$.
Then it turns out that for a given flag variety the condition of being $H$-spherical depends only on its nil-equivalence class: if $X_1,X_2 \in \mathscr F(G)$ and $X_1 \sim X_2$ then $X_1$ is $H$-spherical if and only if $X_2$ is so.
Next, for every $X \in \mathscr F(G)$ let $\bl X \br$ denote the image of $X$ in the quotient set~$\mathscr F(G) / \sim$.
The inclusion relation on the set of nilpotent orbits in~$\mathfrak g$ determines a partial order $\preccurlyeq$ on $\mathscr F(G) / \sim$ as follows: we write $\bl X_1 \br \preccurlyeq \bl X_2 \br$ if and only if $\mathscr N(X_1)$ is contained in the closure of $\mathscr N(X_2)$.
This partial order has the following remarkable property: if $X_1,X_2 \in \mathscr F(G)$, $\bl X_1 \bl \preccurlyeq \bl X_2 \br$, and $X_2$ is $H$-spherical then $X_1$ is also $H$-spherical.
In other words, the property of being $H$-spherical ``spreads'' to smaller nil-equivalence classes in~$\mathscr F(G) / \sim$.
It follows that a natural starting point for solving problem~(\ref{P1}) for a given group~$G$ is to determine all the minimal nil-equivalence classes in $\mathscr F(G) / \sim$ and classify all spherical actions on the corresponding flag varieties.

If $G$ is one of the classical groups $\SL_n$, $\Sp_{2n}$, or $\SO_n$ then the nilpotent orbits in~$\mathfrak g$, the inclusion relation between their closures, and the map $X \mapsto \mathscr N(X)$ admit an effective combinatorial description in terms of partitions (see, for instance, \cite{CM}; a summary of these results is presented in~\cite[\S\,3]{AvP1} for $G = \SL_n$ and in~\S\,\ref{subsec_NO_in_sp_and_so} for $G = \Sp_{2n}$ and $\SO_n$).
Using this description, it is easy to determine the minimal elements of $\mathscr F(G)/ \sim$; see details in \S\S\,\ref{subsec_po_sp2n}--\ref{subsec_po_so2n}.
The results are presented in Table~\ref{table_min_elements} where $\SOGr_1(\FF^{n}) \subset \PP(\FF^n)$ is the variety of isotropic lines in~$\FF^n$ and $\SOGr_{\max}^\pm(\FF^{2n})$ are the two connected components of the variety of isotropic subspaces in~$\FF^{2n}$ of (maximal possible) dimension~$n$.
(Isotropic subspaces are taken with respect to the bilinear form defining the orthogonal group.)

\begin{table}[h]

\caption{}
\label{table_min_elements}

\begin{tabular}{|c|c|c|}
\hline

$G$ &
\renewcommand{\tabcolsep}{0pt}%
\begin{tabular}{c}Number of\\[-3pt]min. elements\end{tabular}
& Minimal elements \\

\hline

$\SL_n$ & 1 & $\bl \PP(\FF^n) \br$ \\

\hline

$\Sp_{2n}$ & 1 & $\bl \PP(\FF^{2n}) \br$ \\

\hline

$\SO_{2k+1}$ & 1 & $\bl \SOGr_1(\FF^{2k+1}) \br$ \\

\hline

$\SO_{4k+2}$ & 2 & $\bl \SOGr_1(\FF^{4k+2}) \br, \ \bl \SOGr_{\max}^+(\FF^{4k+2}) \br = \bl \SOGr_{\max}^-(\FF^{4k+2}) \br$ \\

\hline

$\SO_{4k}$ & 3 & $\bl \SOGr_1(\FF^{4k}) \br, \ \bl \SOGr_{\max}^+(\FF^{4k}) \br, \ \bl \SOGr_{\max}^-(\FF^{4k}) \br$ \\

\hline
\end{tabular}
\end{table}

For a finite-dimensional $H$-module~$U$, it is easy to see that $H$ acts spherically on the projective space $\PP(U)$ if and only if $H \times \FF^\times$ acts spherically on~$U$ (where the action of~$\FF^\times$ is by scalar transformations).
More generally, given a connected reductive group~$K$, a finite-dimensional $K$-module~$V$ is said to be \textit{spherical} if $V$ is spherical as a $K$-variety.
There is a complete classification of all spherical modules obtained in~\cite{Kac}, \cite{BR}, and~\cite{Lea} (see \S\,\ref{subsec_spherical_modules} for more details).
According to the above discussion, this classification provides a description of all connected reductive subgroups of $\SL_n$ that act spherically on $\PP(\FF^n)$, which was a starting point for solving problem~(\ref{P1}) for $G = \SL_n$ in~\cite{AvP1}.
Likewise, in the present paper we apply the classification of spherical modules to determine all connected reductive subgroups of $\Sp_{2n}$ acting spherically on~$\PP(\FF^{2n})$, which provides a starting point for solving problem~(\ref{P1}) in the case $G = \Sp_{2n}$.
We note that the list of such subgroups turns out to be very short (see Theorem~\ref{thm_sp_SA_P(V)}), and because of this our classification for $G = \Sp_{2n}$ turns out to be much easier than that for $G = \SO_n$.

In the case $G = \SO_n$ it is a much more complicated task to describe all spherical actions on flag varieties whose nil-equivalent classes are minimal elements of $\mathscr F(G) / \sim$.
For an action of $H$ on $X \in \mathscr F(G)$ to be spherical, a necessary condition is that $H$ has an open orbit in~$X$.
If $X$ is a variety of isotropic subspaces of $\FF^n$ of a fixed dimension (such varieties are often called isotropic Grassmannians or Grassmannians of isotropic subspaces), then the work of Kimelfeld~\cite{Kim} provides a classification of connected reductive subgroups $H \subset \SO_n$ having an open orbit in~$X$ and admitting no proper $H$-stable subspaces in~$\FF^n$ that are nondegenerate with respect to the bilinear form defining the group~$\SO_n$.
Although this classification deals with a rather particular situation, starting from it and using various reductions we ultimately manage to deduce classifications of all spherical actions on the varieties $\SOGr_1(\FF^n)$ and $\SOGr_{\max}^{(\pm)}(\FF^n)$ and then complete the whole classification in the case $G = \SO_n$.

Our results in solving problem~(\ref{P1}) for $G = \Sp_{2n}$ or $\SO_n$ show that the overwhelming majority of spherical actions on flag varieties occurs in the case where the variety acted on is a Grassmannian of isotropic subspaces of dimension~1, or~2, or maximal possible.
In particular, this means that the major part of our classification is concentrated in the analysis of flag varieties whose nil-equivalence classes are either minimal elements of $\mathscr F(G) / \sim$ or ``close'' to minimal.
It is also worth mentioning that, if $G = \Sp_{2n}$ or $G = \SO_{2k+1}$ and $X$ is not the Grassmannian of isotropic lines, then there are only a few cases of spherical actions on~$X$ where $H$ is neither intermediate between a Levi subgroup of $G$ and its derived subgroup nor a symmetric subgroup of~$G$, see Theorems~\ref{thm_result_sp} and~\ref{thm_result_so_odd}.

Although the general strategy for solving problem~(\ref{P1}) used in this paper is the same as in \cite{AvP1}, for checking $H$-sphericity of a given flag variety $X$ of~$G$ we use an effective criterion that is a consequence of a general result of Panyushev~\cite{Pan}.
Namely, from $H$ and $X$ one computes explicitly a Levi subgroup $M$ of $H$ together with a finite-dimensional $M$-module~$U$, and it turns out that $X$ is $H$-spherical if and only if $U$ is a spherical $M$-module (see Proposition~\ref{prop_sph_criterion_eff}).

For solving problem~(\ref{P2}) we use techniques described in~\cite[\S\,4]{AvP2}.
Given an $H$-spherical flag variety $X$ of~$G$, the restrictions to $H$ of all irreducible representations of~$G$ realized in spaces of sections of homogeneous line bundles on~$X$ are encoded in a free monoid of finite rank, which we call the \textit{restricted branching monoid}.
In the cases $G = \Sp_{2n}$, $X = \PP(\FF^{2n})$ and $G = \SO_n$, $X = \SOGr_1(\FF^n)$, the description of this monoid follows from well-known facts.
For the remaining cases, the restricted branching monoids are determined as follows.
First, the rank of the monoid is easily computed from the spherical $M$-module $U$ mentioned in the previous paragraph.
Second, to find all the indecomposable elements of the monoid, it suffices to explicitly compute the restrictions to~$H$ of several irreducible representations of~$G$ with ``small'' highest weights~$\lambda$, where $\lambda$ is usually a fundamental weight or the sum of two (not necessarily distinct) fundamental weights of~$G$.
Luckily, in the cases that appear in our paper, the computation of such restrictions is rather straightforward because it goes through a chain of successive restrictions to intermediate subgroups and each intermediate restriction is to a Levi subgroup, or to a symmetric subgroup, or a restriction from $\SO_7$ to~$\mathsf G_2$.
In the Levi subgroup and symmetric subgroup cases, the restrictions are computed using the tables in~\cite[\S\,5.5]{AvP2}.
The restrictions from $\SO_7$ to $\mathsf G_2$ can be computed via~\cite[Theorem~8, part~3]{AkP} or directly by using the program LiE~\cite{LiE1}.

As a final remark, we would like to mention that it would be interesting to characterize the spherical actions on flag varieties in terms of the existing combinatorial description of arbitrary spherical varieties (see, for instance, \cite[\S\S\,15.1,\,30.11]{Tim}) and/or compute the combinatorial data corresponding to all classified cases of such actions.

This paper is organized as follows.
In \S\,\ref{sect_notation&conventions} we set up some notation and conventions used throughout the paper.
In \S\,\ref{sect_preliminaries} we introduce basic notions and recall some facts needed to state our main results.
In turn, the main results of this paper are presented in~\S\,\ref{sect_main_results}.
In~\S\,\ref{sect_main_tools} we discuss the general strategy for classifying spherical actions on flag varieties of a given group~$G$ and analyze in more detail the cases of a symplectic and orthogonal group.
The classification itself is carried out in \S\,\ref{sect_sympl_case} for the symplectic case and in \S\,\ref{sect_orth_case} for the orthogonal case.
In \S\,\ref{sect_RBM} we explain how to compute the restricted branching monoids in all cases classified in our paper.
Finally, Appendix~\ref{sect_g2&spin7} contains explicit realizations of the algebra $\mathfrak g_2$ as a subalgebra of $\mathfrak{so}_7$ and the algebra $\mathfrak{spin}_7$ as a subalgebra of $\mathfrak{so}_8$ that are needed in some computations in~\S\,\ref{sect_orth_case}.

\subsection*{Acknowledgements}

A part of this work was done while the first author was visiting the Institut Fourier in Grenoble, France, in October 2018.
He thanks this institution for hospitality and excellent working conditions and also expresses his gratitude to Michel Brion for support and useful discussions.
Both authors are grateful to the referees for their valuable comments and suggestions on a previous version of this paper.

The results of \S\S\,\ref{subsec_sympl_case_end},\,%
\ref{subsec_SOGr_max_r=2}--\ref{subsec_orth_case_end} are obtained by the first author supported by the grant RSF--DFG 16-41-01013.
The results of \S\S\,\ref{subsec_P(V)}--\ref{subsec_SpGr_2},\,%
\ref{subsec_SOGr_1},\,\ref{subsec_SOGr_max_r=1} are obtained by the second author supported by  the RFBR grant no.~16-01-00818.

\section{Notation and conventions}
\label{sect_notation&conventions}

Throughout the paper, all topological terms refer to the Zariski topology.
All groups are assumed to be algebraic unless they explicitly appear as character groups.
All subgroups of algebraic groups are assumed to be closed.
The Lie algebras of groups denoted by capital Latin letters are denoted by the corresponding small Gothic letters.
Given a group~$K$, a \textit{$K$-variety} is an algebraic variety equipped with a regular action of~$K$.

Notation:

$\ZZ^+ = \lbrace z \in \ZZ \mid z \ge 0 \rbrace$;

$|X|$ is the cardinality of a finite set~$X$;

$\langle v_1,\ldots, v_k \rangle$ is the linear span of vectors $v_1,\ldots, v_k$ of a vector space~$V$;

$V^*$ is the vector space of linear functions on a vector space~$V$;

$\operatorname{S}^d V$ is the $d$th symmetric power of a vector space~$V$;

$\wedge^d V$ is the $d$th exterior power of a vector space~$V$;

$\mathfrak X(K)$ is the character group of a group~$K$ (in additive notation);

$K_x$ is the stabilizer of a point~$x$ of a $K$-variety $X$;

$\overline Y$ is the closure of a subset $Y$ of a variety~$X$;

$\FF[X]$ is the algebra of regular functions on an algebraic variety~$X$;

$\FF(X)$ is the field of rational functions on an irreducible algebraic variety~$X$;

$T_xX$ is the tangent space of an algebraic variety $X$ at a point $x \in X$;

$V^{(K)}_\chi$ is the space of semi-invariants of weight $\chi \in \mathfrak X(K)$ for an action of a group $K$ on a vector space~$V$.

The simple roots and fundamental weights of simple groups and their Lie algebras are numbered as in~\cite{Bou1}.

Given two groups $F \subset K$ and a $K$-module~$V$, the restriction of $V$ to~$F$ is denoted by $\left. V \right|_F$.

For every connected reductive group~$K$, we choose a Borel subgroup $B_K$ and a maximal torus~$T_K \subset B_K$.
Let $B_K^-$ be the Borel subgroup of~$K$ opposite to~$B_K$ with respect to~$T_K$, so that $B_K \cap B_K^- = T_K$.
The groups $\mathfrak X(B_K)$ and $\mathfrak X(B_K^-)$ are identified with $\mathfrak X(T_K)$ via restricting characters to~$T_K$.
Let $\Lambda^+(K) \subset \mathfrak X(T_K)$ be the set of dominant weights of $T_K$ with respect to~$B_K$.
For every $\lambda \in \Lambda^+(K)$, we denote by $R_K(\lambda)$ the simple $K$-module with highest weight~$\lambda$.
When the group $K$ is clear from the context, we write just~$R(\lambda)$.

Given a connected reductive group~$K$ and a finite-dimensional $K$-module~$V$, by abuse of language the pair $(K,V)$ itself is often referred to as a module.

Throughout the paper (except for the introduction), $G$ denotes a simply connected semisimple group.
Let $\pi_1, \ldots, \pi_s \in \Lambda^+(G)$ be all the fundamental weights of~$G$ and consider the index set $S = \lbrace 1, \ldots, s \rbrace$.

For every subset $I \subset S$, we consider the monoid $\Lambda^+_I(G) = \ZZ^+ \lbrace \pi_i \mid i \in I \rbrace \subset \Lambda^+(G)$.
Put $\lambda_I = \sum \limits_{i \in I} \pi_i$ and let $P_I^-$ be the stabilizer in $G$ of the line spanned by a lowest weight vector (with respect to~$B_G$ and~$T_G$) in~$R_G(\lambda_I)^*$.
Then $P_I^-$ is a parabolic subgroup of~$G$ containing~$B_G^-$.
Note that the character group $\mathfrak X(P_I^-)$ is canonically identified with $\ZZ \Lambda^+_I(G) = \ZZ \lbrace \pi_i \mid i \in I \rbrace$ via restricting characters from $P_I^-$ to~$T_G$.
At last, we let $X_I = G/P_I^-$ be the flag variety of~$G$ corresponding to~$I$.

A flag variety $X$ of the group $G$ is said to be \textit{trivial} if $X$ is a point and \textit{nontrivial} otherwise.
For a subset $I \subset S$, the flag variety $X_I$ is nontrivial if and only if $I \ne \varnothing$.

When explicitly describing modules for connected reductive groups, we always use the following conventions:

\begin{itemize}
\item
the groups $\GL_n$, $\SL_n$, $\Sp_n$ (for $n = 2m$), $\SO_n$ act on~$\FF^n$ via their tautological representations; the actions on $(\FF^n)^*$ as well as on symmetric and exterior powers of~$\FF^n$ are induced from this action on~$\FF^n$;

\item
for the group $\Sp_{2m}$ the notation $\wedge^2_0 \FF^{2m}$ stands for the module~$R(\pi_2)$ (which is realized as a codimension~1 submodule of $\wedge^2 \FF^{2m}$);

\item
the groups $\Spin_7$ and $\Spin_9$ act on $\FF^8$ and $\FF^{16}$, respectively, via the spinor representation;

\item
the group $\Spin_{10}$ acts on $\FF^{16}$ via a (either of two) half-spin representation;

\item
the group $\mathsf G_2$ acts on $\FF^7$ via a faithful representation of minimal dimension.

\end{itemize}

In this paper, we often fall into a situation where a connected reductive group~$K$ acts on a finite-dimensional module~$V$ written as a direct sum of several submodules each being a tensor product of several components acted on by different factors of~$K$ and we need to specify precisely the action of~$K$ on~$V$.
In such situations, our notation follow the following conventions:
\begin{enumerate}
\item
the group $K$ is written as
$K = K_1 \times \ldots \times K_p \times \underbrace{\FF^\times \times \ldots \times \FF^\times}_q$
where each factor $K_i$ is visually different from $\FF^\times$ (for example, some $K_i$'s may be written as $\GL_1$ or~$\SO_2$); for short, the product $\underbrace{\FF^\times \times \ldots \times \FF^\times}_q$ is denoted by~$T$ below;

\item
for every $i = 1, \ldots, p$ we write the number $i$ right below each component of $V$ on which $K_i$ acts nontrivially (exceptions: this notation is omitted when $p=1$ or $V$ is a simple $K$-module);

\item
for every $j = 1,\ldots, q$, we denote by $\chi_j$ a basis character of the $j$th copy of $\FF^\times$ in~$T$ (if $q = 1$ we write just $\chi$ instead of~$\chi_1$);

\item
if $T$ nontrivially acts on a simple summand $U$ of $V$ via a character~$\psi$, we write $[U]_\psi$ instead of~$U$ (exception: if $U = \FF^1$ then we write simply $\FF^1_\psi$).
\end{enumerate}
An example of a pair $(K,V)$ written using the above conventions is given by
\[
(\SL_2 \times \SL_2 \times \FF^\times \times \FF^\times , [\mathop{\vphantom|\FF^2} \limits_1{\!} \otimes \mathop{\vphantom|\FF^2} \limits_2{\!}]_{\chi_1+\chi_2} \oplus \mathop{[\vphantom|\FF^2} \limits_1{\!}]_{2\chi_1} \oplus \mathop{\vphantom|\mathrm{S}^2 \FF^2} \limits_2{\!} \oplus \FF^1_{2\chi_2}).
\]

For explicit calculations, we use the following realizations of the symplectic and orthogonal groups (where $A$ is the $(n \times n)$-matrix with ones on the antidiagonal and zeros elsewhere):
\begin{itemize}
\item
$\Sp_{2n}$ is the subgroup of $\GL_{2n}$ preserving the skew-symmetric bilinear form with matrix
$
\begin{pmatrix}
0 & A \\ -A & 0
\end{pmatrix};
$

\item
$\SO_n$ is the subgroup of $\SL_n$ preserving the symmetric bilinear form with matrix~$A$.
\end{itemize}
With these realizations, for $K = \Sp_{2n}$ and $K = \SO_n$ we choose $B_K$ (resp. $B_K^-$, $T_K$) to be the group of all upper-triangular (resp. lower-triangular, diagonal) matrices in~$K$.
Then the root system of $K$ is identified with a subset of $\mathfrak X(T_K)$, in which we choose the set of simple roots corresponding to~$B_K$.
In the case $K=\SO_{2m}$, due to the symmetry of the Dynkin diagram, we put by convention that the $(m-1)$th (resp.~$m$th) simple root takes the value $t_{m-1}t_m^{-1}$ (resp.~$t_{m-1}t_m$) on every diagonal matrix $t \in\nobreak T_K$ with diagonal entries $t_1,\ldots,t_m, t_m^{-1},\ldots, t_1^{-1}$.

The vectors of the standard basis of $\FF^n$ are denoted by $e_1,\ldots, e_n$.

\section{Preliminaries}
\label{sect_preliminaries}

\subsection{Homogeneous line bundles on flag varieties}

Let $K$ be a subgroup of~$G$ and consider the homogeneous space $G/K$.

Given $\chi \in \mathfrak X(K)$, consider the one-dimensional $K$-module $\FF^1_{\chi}$ on which $K$ acts via~$\chi$.
Let $K$ act on $G$ by right multiplication and let $L(\chi)$ be the quotient $(G \times \FF^1_{\chi})/K$ with respect to the diagonal action of~$K$.
Then the natural map $L(\chi) \to G/K$ turns $L(\chi)$ into a line bundle on $G/K$.
As this map is $G$-equivariant, $L(\chi)$ is called a \textit{homogeneous line bundle} on~$G/K$.
Note that the space of global sections $H^0(G/K, L(\chi))$ is a $G$-module.

When $K = P_I^-$ for some $I \subset S$, one has $G/K = X_I$.
In this case, a version of the Borel--Weil theorem states that for every $\lambda \in \mathfrak X(P_I^-)$ there is a $G$-module isomorphism
\begin{equation} \label{eqn_H^0_for_XI}
H^0(X_I, L(\lambda)) \simeq
\begin{cases} R_G(\lambda) & \text{if} \ \lambda \in \Lambda^+_I(G);\\
0 & \text{otherwise};
\end{cases}
\end{equation}
it follows, for instance, from general results discussed in~\cite[Part~I, \S\S\,5.12--5.16 and Part~II, \S\,2.2]{Jan}; see also~\cite[\S\,4.1]{AvP2} for an explanation in this particular situation.
Formula~(\ref{eqn_H^0_for_XI}) shows how irreducible representations $R_G(\lambda)$ with $\lambda \in \Lambda_I^+(G)$ are realized as spaces of sections of homogeneous line bundles on~$X_I$.

\subsection{Finite-dimensional modules with invariant bilinear forms}

Let $K$ be a connected reductive group and let $V$ be a finite-dimensional $K$-module.
Suppose that $K$ preserves a nondegenerate bilinear form $\omega$ on $V$ that is either symmetric or skew-symmetric.

A $K$-submodule $W \subset V$ is said to be \textit{nondegenerate} if the restriction of $\omega$ to $W$ is nondegenerate.
Clearly, in this situation there is a $K$-module decomposition $V = W \oplus W^\perp$ where $W^\perp$ is the orthogonal complement of $W$ in~$V$ with respect to the form~$\omega$.

Given a simple $K$-submodule $W \subset V$, the kernel of the restriction of $\omega$ to~$W$ is $K$-stable and hence equal to either $\lbrace 0 \rbrace$ or the whole~$W$.
It follows that $W$ is either nondegenerate or isotropic.

The following fact is well known; see, for instance, \cite[Theorem~4]{Mal} or~\cite[Theorem~2.2]{Kim} for a proof.

\begin{proposition} \label{prop_no_nondegen}
Suppose that $V$ contains no proper nondegenerate $K$-submodules.
Then one of the following two cases occurs:
\begin{enumerate}[label=\textup{(\arabic*)},ref=\textup{\arabic*}]
\item
$V$ is irreducible;

\item \label{prop_no_nondegen_2}
there are simple $K$-submodules $W_1,W_2 \subset V$ such that $V = W_1 \oplus W_2$, both $W_1,W_2$ are isotropic, and $W_2 \simeq W_1^*$ as $K$-modules.
\end{enumerate}
\end{proposition}

Following the terminology of Kimelfeld~\cite{Kim}, we say that $V$ is \textit{weakly reducible} (with respect to~$\omega$) if it falls into case~(\ref{prop_no_nondegen_2}) of Proposition~\ref{prop_no_nondegen}.

For every $K$-module~$W$, we introduce the notation $\Omega(W) = W \oplus W^*$.
In what follows, we shall regard $\Omega(W)$ as a $K$-module equipped with a $K$-invariant nondegenerate symmetric or skew-symmetric bilinear form such that $\Omega(W)$ is a direct sum of two isotropic subspaces isomorphic to $W$ and~$W^*$ as $K$-modules.

\subsection{Equivalence and BF-equivalence on finite-dimensional modules}
\label{subsec_eq&BF-eq}

Given two connected reductive groups $K_1,K_2$, for $i=1,2$ let $V_i$ be a finite-dimensional $K_i$-module and consider the corresponding representation $\rho_i \colon K_i \to \GL(V_i)$.
We say that the pairs $(K_1, V_1)$ and $(K_2, V_2)$ are \textit{equivalent} if there exists an isomorphism $V_1 \xrightarrow{\sim} V_2$ identifying the groups $\rho_1(K_1) \subset \GL(V_1)$ and $\rho_2(K_2) \subset \GL(V_2)$.
In other words, the pairs $(K_1, V_1)$ and $(K_2, V_2)$ are equivalent if and only if they define the same linear group.

As an important example, every pair $(K,V)$ is equivalent to the pair $(K,V^*)$.

In Table~\ref{table_equiv_modules} we list several equivalences for pairs $(\SO_n, \FF^n)$ with small values of~$n$, these equivalences are widely used throughout this paper.

\begin{table}[h]
\caption{} \label{table_equiv_modules}

\begin{tabular}{|c|c|c|c|c|}
\hline

$(\SO_2, \FF^2)$ & $(\SO_3, \FF^3)$ & $(\SO_4, \FF^4)$ & $(\SO_5, \FF^5)$ & $(\SO_6, \FF^6)$ \\

\hline

$(\GL_1, \Omega(\FF^1))$ & $(\SL_2, \mathrm{S}^2 \FF^2)$ & $(\SL_2 \times \SL_2, \FF^2 \otimes \FF^2)$ & $(\Sp_4, \wedge^2_0 \FF^4)$ & $(\SL_4, \wedge^2 \FF^4)$ \\

\hline

\end{tabular}

\end{table}

Now let $K_1, K_2, V_1, V_2, \rho_1, \rho_2$ be as above and suppose that for $i = 1,2$ the space $V_i$ carries a $K_i$-invariant bilinear form~$\omega_i$.
We say that the pairs $(K_1,V_1)$ and $(K_2, V_2)$ are \textit{BF-equivalent} if there exists an isomorphism $V_1 \xrightarrow{\sim} V_2$ identifying the group $\rho_1(K_1) \subset \GL(V_1)$ with $\rho_2(K_2) \subset \GL(V_2)$ and taking the form $\omega_1$ to~$\omega_2$.
In particular, if the pairs $(K_1,V_1)$ and $(K_2,V_2)$ are BF-equivalent then they are equivalent.

Given a connected reductive group $K$ and a finite-dimensional $K$-module~$V$, in this paper we shall often need to specify the pair $(K,V)$ up to BF-equivalence.
To this end, we always assume that the corresponding $K$-invariant bilinear form~$\omega$ on~$V$ is nondegenerate and either symmetric or skew-symmetric (the choice between symmetric and skew-symmetric will always be clear from the context).
Further, when $V$ is explicitly written as
\begin{equation} \label{eqn_decomposition}
V_1 \oplus \ldots \oplus V_p \oplus \Omega(W_1) \oplus \ldots \oplus \Omega(W_q)
\end{equation}
we also assume the following properties:
\begin{enumerate}
\item \label{conv_E_1}
all direct summands in~(\ref{eqn_decomposition}) are pairwise orthogonal with respect to the form~$\omega$;

\item \label{conv_E_2}
for each $j = 1,\ldots,q$ the summand $\Omega(W_j)$ is a direct sum of two $K$-stable isotropic subspaces isomorphic to $W_j$ and $W_j^*$ as $K$-modules.
\end{enumerate}
Throughout this paper, the above conventions will always be enough to uniquely determine the BF-equivalence class of the pair $(K,V)$.

\subsection{Spherical varieties}

Let $K$ be a connected reductive group.
We recall from the introduction that a $K$-variety $X$ is said to be \textit{spherical} (or \textit{$K$-spherical}) if the Borel subgroup $B_K$ has a dense open orbit in~$X$.

Given a spherical $K$-variety~$X$, we put
\[
\Lambda_X = \lbrace \lambda \in \mathfrak X(T_K) \mid \FF(X)^{(B)}_\lambda \ne \lbrace 0 \rbrace \rbrace.
\]
It is easy to see that $\Lambda_X$ is a sublattice of~$\mathfrak X(T_K)$, it is called the \textit{weight lattice} of~$X$.
The rank of this lattice is said to be the \textit{rank} of~$X$; we denote it by $\rk_K X$.

In this paper, we shall need the following general fact on spherical varieties, which follows, for instance, from~\cite[Theorem~1]{Vin86}.

\begin{theorem} \label{thm_subvar_sph}
Suppose that $X$ is a spherical $K$-variety.
Then any $K$-stable irreducible subvariety of $X$ is also spherical.
\end{theorem}

\subsection{Spherical modules}
\label{subsec_spherical_modules}

Let $K$ be a connected reductive group and let $V$ be a finite-dimensional $K$-module.

As was already mentioned in the introduction, $V$ is said to be a \textit{spherical $K$-module} if $V$ is spherical as a $K$-variety.
In this case, it is easy to see that every $K$-submodule of~$V$ is also spherical. (This also follows from Theorem~\ref{thm_subvar_sph}.)
According to \cite[Theorem~2]{VK78}, the condition of $V$ being a spherical $K$-module is equivalent to the fact that the $K$-module $\FF[V]$ is multiplicity free.
Given a spherical $K$-module $V$, the highest weights of all simple $K$-modules that occur in~$\FF[V]$ form a submonoid $\mathrm E_K(V)$ of~$\Lambda^+(K)$, called the \textit{weight monoid} of~$V$.
It is well known that $\mathrm E_K(V)$ is free (see, for instance,~\cite[Theorem~3.2]{Kn}) and $\rk \mathrm E_K(V) = \rk_K V$ (see, for instance,~\cite[Proposition~5.14]{Tim}).

The terminology introduced below follows Knop, see~\cite[\S\,5]{Kn}.

Let $\rho \colon K \to \GL(V)$ be the representation defining the $K$-module structure on~$V$.
We say that $V$ is \textit{saturated} if the dimension of the center of $\rho(K)$ equals the number of irreducible summands of~$V$.

We say that $V$ is \textit{decomposable} if for $i = 1,2$ there exist a connected reductive group $K_i$ and a finite-dimensional $K_i$-module $V_i$ such that the pair $(K,V)$ is equivalent to $(K_1 \times K_2, V_1 \oplus V_2)$.
Evidently, in this situation $(K,V)$ is a spherical module if and only if so are both $(K_1,V_1)$ and $(K_2, V_2)$, in which case $\rk_K V = \rk_{K_1} V_1 + \rk_{K_2} V_2$.
We say that $V$ is \textit{indecomposable} if $V$ is not decomposable.

There is a complete classification of spherical modules.
In the case where $V$ is simple the classification was obtained in~\cite{Kac}.
The case of arbitrary $V$ was settled in the two independent papers~\cite{BR} and~\cite{Lea}, it reduces essentially to classifying all indecomposable saturated spherical modules. (In fact, all such modules for which the derived subgroup of the acting group is simple were classified earlier in~\cite{Bri}.)
The weight monoids of all spherical modules are also known thanks to the works~\cite{HU} (the case of simple~$V$) and~\cite{Lea} (the general case).
A complete list (up to equivalence) of all indecomposable saturated spherical modules can be found in~\cite[\S\,5]{Kn} along with various additional data, including the rank and indecomposable elements of the weight monoids.
In this paper, for checking sphericity of a given module we find it convenient to use~\cite[Theorem~5.3]{AvP1}, which is a reformulation of~\cite[Theorem~7]{BR} and~\cite[Theorem~2.6]{Lea}.

Now suppose that $V$ is a spherical $K$-module and fix a decomposition $V = V_1 \oplus \ldots \oplus V_k$ where each direct summand is a simple $K$-module.
Let $K'$ be the derived subgroup of~$K$ and let $Z$ be the subgroup of $\GL(V)$ consisting of all elements that act by scalar transformations on each~$V_i$, $i = 1,\ldots, k$.
Then $V$ is a saturated spherical ($K' \times Z$)-module, hence the pair $(K,V)$ is equivalent to $(K_1 \times \ldots \times K_m, W_1 \oplus \ldots \oplus W_m)$ where $K_i$ is a connected reductive group and $W_i$ is an indecomposable saturated spherical $K_i$-module for each $i = 1,\ldots, m$.
In this situation, it is easy to see that $\rk_K V = \rk_{K' \times Z} V = \rk_{K_1} W_1 + \ldots + \rk_{K_m} W_m$.
The latter observation will be always used in this paper for computing the ranks of spherical modules.

\subsection{Spherical modules with invariant bilinear forms}

Retain the notation of \S\,\ref{subsec_spherical_modules}.
Combining the well-known description of invariant bilinear forms on spaces of irreducible representations of semisimple groups (see~\cite[\S\,2]{Mal} or~\cite[Ch.~VIII, \S\,7.5, Proposition~12]{Bou2}) with the classification of spherical modules (see the references in~\S\,\ref{subsec_spherical_modules}) and Proposition~\ref{prop_no_nondegen} one obtains the following results.

\begin{proposition} \label{prop_spherical+sympl}
Suppose that $K$ preserves a nondegenerate skew-symmetric bilinear form $\omega$ on~$V$ and $V$ contains no proper $K$-submodules that are nondegenerate with respect to~$\omega$.
Then $V$ is a spherical $(K \times \FF^\times)$-module \textup(with $\FF^\times$ acting by scalar transformations\textup) if and only if the pair $(K,V)$ is equivalent to a pair in Table~\textup{\ref{table_sympl_sph_mod}}.
\end{proposition}

\begin{table}[h]
\caption{} \label{table_sympl_sph_mod}

\begin{tabular}{|c|c|c|c|}
\hline
No. & $K$ & $V$ &  Note \\

\hline

\multicolumn{4}{|c|}{$V$ irreducible} \\

\hline

\newcase & $\Sp_{2n}$ & $\FF^{2n}$ & $n \ge 1$ \\

\hline

\multicolumn{4}{|c|}{$V$ weakly reducible} \\

\hline

\newcase &  $\GL_n$ & $\Omega(\FF^n)$ & $n \ge 1$ \\

\hline

\newcase & $\SL_n$ & $\Omega(\FF^n)$ & $n \ge 3$ \\

\hline

\newcase & $\Sp_{2n} \times \FF^\times$ & $\Omega([\FF^{2n}]_\chi)$ & $n \ge 2$ \\

\hline

\end{tabular}

\end{table}

\begin{proposition} \label{prop_spherical+orth}
Suppose that $K$ preserves a nondegenerate symmetric bilinear form $\omega$ on~$V$ and $V$ contains no proper $K$-submodules that are nondegenerate with respect to~$\omega$.
Then $V$ is a spherical $(K \times \FF^\times)$-module \textup(with $\FF^\times$ acting by scalar transformations\textup) if and only if the pair $(K,V)$ is equivalent to a pair in Table~\textup{\ref{table_orth_sph_mod}}.
\end{proposition}

\begin{table}[h]
\caption{} \label{table_orth_sph_mod}

\begin{tabular}{|c|c|c|c|}
\hline
No. & $K$ & $V$ &  Note \\

\hline

\multicolumn{4}{|c|}{$V$ irreducible} \\

\hline

\newcase & $\SO_{n}$ & $\FF^{n}$ & $n \ge 1$, $n \ne 2$ \\

\hline

\newcase & $\Sp_{2n} \times \SL_2$ & $\FF^{2n} \otimes \FF^2$ & $n \ge 2$ \\

\hline

\newcase & $\Spin_7$ & $\FF^8$ & \\

\hline

\newcase & $\Spin_9$ & $\FF^{16}$ & \\

\hline

\newcase & $\mathsf G_2$ & $\FF^7$ & \\

\hline

\multicolumn{4}{|c|}{$V$ weakly reducible} \\

\hline

\newcase &  $\GL_n$ & $\Omega(\FF^n)$ & $n \ge 1$ \\

\hline

\newcase & $\SL_n$ & $\Omega(\FF^n)$ & $n \ge 3$ \\

\hline

\newcase & $\Sp_{2n} \times \FF^\times$ & $\Omega([\FF^{2n}]_\chi)$ & $n \ge 2$ \\

\hline

\end{tabular}

\end{table}

\subsection{Branching monoids and restricted branching monoids}

A more detailed discussion of the notions introduced in this subsection can be found in \cite[\S\,3.4]{AvP2}.

Let $H \subset G$ be a connected reductive subgroup.

We put
\[
\Gamma(G, H) = \lbrace (\lambda; \mu) \in \Lambda^+(G) \times \Lambda^+(H) \mid R_H(\mu) \ \text{is a submodule of} \left. R_G(\lambda) \right|_H \rbrace.
\]
Then $\Gamma(G, H)$ is a submonoid of $\Lambda^+(G) \times \Lambda^+(H)$, it is called the \textit{branching monoid} for the pair $(G,H)$.

Given any subset $I \subset S$, the monoid
\[
\Gamma_I(G,H) = \lbrace (\lambda; \mu) \in \Gamma(G,H) \mid \lambda \in \Lambda^+_I(G) \rbrace
\]
is called the \textit{restricted branching monoid} corresponding to the subset~$I$.

\subsection{Spherical actions on flag varieties and the corresponding restricted branching monoids}

Let $H \subset G$ be a connected reductive subgroup and let $I \subset S$ be an arbitrary subset.

The next theorem is a particular case of~\cite[Corollary~1]{VK78}.

\begin{theorem} \label{thm_criterion_spherical}
The following conditions are equivalent:
\begin{enumerate}[label=\textup{(\arabic*)},ref=\textup{\arabic*}]
\item
For every $\lambda \in \Lambda^+_I(G)$, the $H$-module $\left. R_G(\lambda) \right|_H$ is multiplicity free.

\item
The flag variety $X_I$ is $H$-spherical.
\end{enumerate}
\end{theorem}

Under the conditions of Theorem~\ref{thm_criterion_spherical}, the restriction to~$H$ of any simple $G$-module $R_G(\lambda)$ with $\lambda \in \Lambda^+_I(G)$ is uniquely determined by the monoid $\Gamma_I(G,H)$ as follows:
\begin{equation} \label{eqn_restriction}
\left. R_G(\lambda) \right|_H \simeq \bigoplus \limits_{\mu \in \Lambda^+(H) \, : \, (\lambda; \mu) \in \Gamma_I(G,H)} R_H(\mu).
\end{equation}

The following result is implied by \cite[Theorem~4.2 and Proposition~4.4]{AvP2}.

\begin{theorem} \label{thm_Gamma_is_free}
Under the conditions of Theorem~\textup{\ref{thm_criterion_spherical}}, the following assertions hold:
\begin{enumerate}[label=\textup{(\alph*)},ref=\textup{\alph*}]
\item
the monoid $\Gamma_I(G,H)$ is free;

\item \label{thm_Gamma_is_free_b}
$\rk \Gamma_I(G,H) = |I| + \rk_H X_I$.
\end{enumerate}
\end{theorem}

\section{Statement of the main results}
\label{sect_main_results}

\subsection{Reductions}
\label{subsec_reductions}

In this subsection, we describe several reductions that simplify the statement of main theorems in this section.
Fix a connected reductive subgroup $H \subset G$ along with a subset $I \subset S$ and suppose that the variety $X_I$ is $H$-spherical.

\textit{Reduction}~1.
Let $I' \subset I$ be an arbitrary subset.
Then the variety $X_{I'}$ is automatically $H$-spherical and the indecomposable elements of $\Gamma_{I'}(G,H)$ are those of $\Gamma_I(G,H)$ for which the first component belongs to $\Lambda^+_{I'}(G)$.
Therefore, for a given pair $(G,H)$, it is enough to consider subsets $I \subset S$ that are maximal with the property that $X_I$ is $H$-spherical.

\textit{Reduction}~2.
Let $\sigma$ be an automorphism of~$G$.
Then $(\tau\sigma)(B_G) = B_G$ and $(\tau\sigma)(T_G) = T_G$ for an appropriate inner automorphism $\tau$ of~$G$.
By abuse of notation, we shall use the same letter $\sigma$ to denote the following objects:
\begin{itemize}
\item
the bijection of $S$ corresponding to the automorphism of the Dynkin diagram of~$G$ induced by~$\tau\sigma$;

\item
the bijection $\Lambda^+(G) \to \Lambda^+(G)$ induced by $\tau\sigma$;

\item
the induced bijection $\Lambda^+(H) \to \Lambda^+(\sigma(H))$.
\end{itemize}
Now, given $\lambda \in \Lambda^+_I(G)$, after changing the action of $G$ on $R_G(\lambda)$ to $(g,v) \mapsto \sigma^{-1}(g)(v)$ formula~(\ref{eqn_restriction}) takes the form
\[
\left. R_G(\sigma(\lambda)) \right|_{\sigma(H)} \simeq \bigoplus \limits_{\mu \in \Lambda^+(H) \, : \, (\lambda;\mu) \in \Gamma_I(G,H)} R_{\sigma(H)}(\sigma(\mu)).
\]
Then $X_{\sigma(I)}$ is $\sigma(H)$-spherical.
Moreover, $(\lambda; \mu) \in \nobreak \Gamma_I(G,H)$ if and only if $(\sigma(\lambda); \sigma(\mu)) \in \Gamma_{\sigma(I)}(G,\sigma(H))$.
In particular, $(\lambda; \mu)$ is an indecomposable element of $\Gamma_I(G,H)$ if and only if $(\sigma(\lambda); \sigma(\mu))$ is an indecomposable element of $\Gamma_{\sigma(I)}(G,\sigma(H))$.
In this situation, we say that the triple $(G,\sigma(H),\sigma(I))$ is obtained from $(G,H,I)$ by the automorphism~$\sigma$.

\textit{Reduction}~3.
Let $H'$ denote the derived subgroup of~$H$ and suppose $K \subset G$ is a connected reductive subgroup such that $H' \subset K \subset H$.
Then it follows from Theorems~\ref{thm_criterion_spherical} and~\ref{thm_Gamma_is_free} that $K$ acts spherically on $X_I$ if and only if the restrictions to $\Lambda_I^+(G) \oplus \Lambda^+(K)$ of all the indecomposable elements of $\Gamma_I(G,H)$ are linearly independent (in which case these restrictions are all the indecomposable elements of $\Gamma_I(G,K)$).
Therefore it suffices to classify spherical actions on flag varieties only for groups $H$ that are not intermediate between a bigger connected reductive subgroup of $G$ and its derived subgroup.

We remark that, for completeness of the results obtained in this paper, we use Reduction~3 only to exclude subgroups $H$ that are intermediate between a Levi subgroup of $G$ and its derived subgroup.

\subsection{Levi subgroups and symmetric subgroups in \texorpdfstring{$\Sp_{2n}$}{Sp\_2n} and \texorpdfstring{$\SO_n$}{$SO\_n$}}
\label{subsec_Levi&symmetric}

Let $V$ be a finite-dimensional vector space equipped with a nondegenerate bilinear form $\omega$ that is either symmetric or skew-symmetric.
Put $G = \Sp(V)$ if $\omega$ is skew-symmetric and $G = \Spin(V)$ if $\omega$ is symmetric.
Let $H \subset G$ be a connected reductive subgroup.

In the statements of our main theorems in \S\S\,\ref{subsec_sp_results}--\ref{subsec_so_results}, we exclude the cases where $H$ is either intermediate between a Levi subgroup of $G$ and its derived subgroup or a symmetric subgroup of~$G$.
For convenience of the reader, in this subsection we specify explicitly all Levi subgroups and all symmetric subgroups in~$G$.

We recall that a subgroup $K$ of $G$ is said to be \textit{symmetric} if $K$ is the subgroup of fixed points of a nontrivial involutive automorphism of~$G$.
As $G$ is simply connected, in this case $K$ is reductive and connected by~\cite[Theorem~8.1]{Stei}.

\begin{proposition}
Suppose that $\omega$ is skew-symmetric. Then
\begin{enumerate}[label=\textup{(\alph*)},ref=\textup{\alph*}]
\item
$H$ is a Levi subgroup of $G$ if and only if the pair $(H,V)$ is BF-equivalent to
\[
(\GL_{n_1}\times \ldots \times \GL_{n_k} \times \Sp_{2m}, \Omega(\mathop{\FF^{n_1}} \limits_1{\!}) \oplus \Omega(\mathop{\FF^{n_2}} \limits_2{\!}) \oplus \Omega(\mathop{\FF^{n_k}} \limits_k{\!}) \oplus \mathop{\FF^{2m}} \limits_{k+1})
\]
for some $k \ge 0$, $m \ge 0$, and $n_i \ge 1$;

\item
$H$ is a symmetric subgroup of $G$ if and only if the pair $(H,V)$ is BF-equivalent to either $(\Sp_{2n_1} \times \Sp_{2n_2}, \FF^{2n_1} \oplus \FF^{2n_2})$ for some $n_1,n_2 \ge 1$ or $(\GL_n, \Omega(\FF^n))$ for some $n \ge 1$.
\end{enumerate}
\end{proposition}

\begin{proposition}
Suppose that $\omega$ is symmetric.
Then
\begin{enumerate}[label=\textup{(\alph*)},ref=\textup{\alph*}]
\item
$H$ is a Levi subgroup of $G$ if and only if the pair $(H,V)$ is BF-equivalent to
\[
(\GL_{n_1}\times \ldots \times \GL_{n_k} \times \SO_{m}, \Omega(\mathop{\FF^{n_1}} \limits_1{\!}) \oplus \Omega(\mathop{\FF^{n_2}} \limits_2{\!}) \oplus \Omega(\mathop{\FF^{n_k}} \limits_k{\!}) \oplus \mathop{\FF^{m}} \limits_{k+1})
\]
for some $k \ge 0$, $m \ge 0$, and $n_i \ge 1$;

\item
$H$ is a symmetric subgroup of $G$ if and only if the pair $(H,V)$ is BF-equivalent to either $(\SO_{n_1} \times \SO_{n_2}, \FF^{n_1} \oplus \FF^{n_2})$ for some $n_1,n_2 \ge 1$ or $(\GL_n, \Omega(\FF^n))$ for some $n \ge 2$.
\end{enumerate}
\end{proposition}

\subsection{The symplectic case}
\label{subsec_sp_results}

Let $V$ be a vector space of dimension $2n$ ($n \ge 2$) equipped with a nondegenerate skew-symmetric bilinear form~$\omega$.
Let $G = \Sp(V) \simeq \Sp_{2n}$ be the subgroup of $\GL(V)$ preserving~$\omega$.

First we consider separately the case $I = \lbrace 1 \rbrace$.
Note that $X_{\lbrace 1 \rbrace} \simeq \PP(V)$ (see \S\,\ref{subsec_FV_via_compositions}).

\begin{theorem} \label{thm_result_sp_PV}
The variety $X_{\lbrace 1 \rbrace}$ is $H$-spherical if and only if $V$ is a spherical $(H \times \FF^\times)$-module where $\FF^\times$ acts by scalar transformations.
Moreover, the above conditions hold if and only if the pair $(H,V)$ is BF-equivalent to one of the pairs in Table~\textup{\ref{table_sph_PV}}.
\end{theorem}

\begin{table}[h]
\caption{The symplectic case for $I = \lbrace 1 \rbrace$} \label{table_sph_PV}

\renewcommand{\tabcolsep}{3pt}
\begin{tabular}{|c|c|c|c|}
\hline
No. & $H$ & $V$ & Note \\

\hline

\newcase \label{sp_case_1}
& $\Sp_{2n_1} \times \ldots \times \Sp_{2n_k}$ & $\mathop{\FF^{2n_1}}\limits_1{\!} \oplus \ldots \oplus \mathop{\FF^{2n_k}}\limits_k{\!}$ & $k {\ge} 1$, $n_i {\ge} 1$ \\

\hline

\newcase \label{sp_case_2}
& $\Sp_{2n_1} \times \ldots \times \Sp_{2n_k} \times \GL_{m}$ & $\mathop{\FF^{2n_1}}\limits_1{\!} \oplus \ldots \oplus \mathop{\FF^{2n_k}}\limits_k{\!} \oplus \mathrm \Omega(\mathop{\FF^{m}}\limits_{k+1}{\!})$ & $k {\ge} 0$, $n_i {\ge} 1$, $m {\ge} 1$ \\

\hline

\newcase \label{sp_case_3}
& $\Sp_{2n_1} \times \ldots \times \Sp_{2n_k} \times \SL_{m}$ & $\mathop{\FF^{2n_1}}\limits_1{\!} \oplus \ldots \oplus \mathop{\FF^{2n_k}}\limits_k{\!} \oplus \Omega(\mathop{\FF^{m}}\limits_{k+1}{\!})$ & $k {\ge} 0$, $n_i {\ge} 1$, $m {\ge} 3$ \\

\hline

\newcase \label{sp_case_4}
& $\Sp_{2n_1} \times \ldots \times \Sp_{2n_k} \times \Sp_{2m} \times \FF^\times$ & $\mathop{\FF^{2n_1}}\limits_1{\!} \oplus \ldots \oplus \mathop{\FF^{2n_k}}\limits_k{\!} \oplus \Omega([\mathop{\FF^{2m}}\limits_{k+1}{\!}]_\chi)$ & $k {\ge} 0$, $n_i {\ge} 1$, $m {\ge} 2$ \\

\hline
\end{tabular}
\end{table}

\begin{remark}
The first equivalence in Theorem~\ref{thm_result_sp_PV} is trivial.
The second one is proved in Theorem~\ref{thm_sp_SA_P(V)}.
\end{remark}

Recall from \S\,\ref{subsec_spherical_modules} the notion of weight monoid $\mathrm E_K(U)$ of a spherical $K$-module~$U$.

\begin{theorem} \label{thm_RBM_sp_PV}
In the situation of Theorem~\textup{\ref{thm_result_sp_PV}}, let $\delta$ denote the character via which $\FF^\times$ acts on~$V$ and let $\mathrm E^0_{H \times \FF^\times}(V^*)$ be the set of indecomposable elements of $\mathrm E_{H \times \FF^\times}(V^*)$.
Identify $\Lambda^+(H \times \FF^\times)$ with $\Lambda^+(H) \oplus \ZZ \delta$ and consider the map $\Lambda^+(H \times \FF^\times) \to \Lambda^+_{\lbrace 1 \rbrace}(G) \times \Lambda^+(H)$ given by $\lambda = \mu + k\delta \mapsto \overline \lambda = (k\pi_1; \mu)$.
Then the set of indecomposable elements of $\Gamma_{\lbrace 1 \rbrace}(G, H)$ is
$
\lbrace \overline \lambda \mid \lambda \in \mathrm E^0_{H \times \FF^\times}(V^*) \rbrace.
$
\end{theorem}

The proof of Theorem~\ref{thm_RBM_sp_PV} is given in~\S\,\ref{subsec_RBM_I={1}}.

We now turn to the case $I \ne \lbrace 1 \rbrace$.

\begin{theorem} \label{thm_result_sp}
Suppose that $I \subset S$ is a nonempty subset distinct from $\lbrace 1 \rbrace$ and $H \subset G$ is a connected reductive subgroup such that
\begin{itemize}
\item
$H$ is not intermediate between a Levi subgroup of~$G$ and its derived subgroup;

\item
$H$ is not a symmetric subgroup of~$G$.
\end{itemize}
Then the following conditions are equivalent:
\begin{enumerate}[label=\textup{(\arabic*)},ref=\textup{\arabic*}]
\item \label{thm_result_sp_1}
The variety $X_I$ is $H$-spherical and $I$ is maximal with this property.

\item \label{thm_result_sp_2}
the pair $(H,V)$, considered up to BF-equivalence, and the set $I$ fall into the only case in Table~\textup{\ref{table_result_sympl}}.
\end{enumerate}
Moreover, Table~\textup{\ref{table_result_sympl}} lists also the rank and indecomposable elements of the monoid $\Gamma_I(G,H)$ for that only case.
\end{theorem}

\begin{table}[h]

\caption{The symplectic case for $I \ne \lbrace 1 \rbrace$}
\label{table_result_sympl}

\begin{tabular}{|c|l|l|l|}
\hline
No. & Conditions & Rank & Indecomposable elements of~$\Gamma_I(G,H)$  \\

\hline
\newcase
& \multicolumn{3}{|c|}{$(\Sp_{2m} \times \SL_2 \times \SL_2, \mathop{\FF^{2m}} \limits_1{\!} \oplus \mathop{\FF^{2}}\limits_2{\!} \oplus \mathop{\FF^{2}} \limits_3{\!})$, $m \ge 1$}\\

\hline
\no \label{sp_case_1.2}
&
$I {=} \lbrace n \rbrace$
&
$5 {-} \delta_{m}^1$
&
\renewcommand{\tabcolsep}{0pt}%
\begin{tabular}{l}
$(\pi_n; \pi_{m} {+} \pi'_1 {+} \pi''_1)$, $(\pi_n; \pi_{m})$, %\\
$(\pi_n; \pi_{m-1} {+} \pi'_1)$,
$(\pi_n; \pi_{m-1} {+} \pi''_1)$, \\[-2pt]
$(\pi_n; \pi_{m-2})$ ($m {\ge} 2$)
\end{tabular}
\\

\hline

\end{tabular}
\end{table}

For notation used in Table~\ref{table_result_sympl}, see \S\,\ref{subsec_notation_in_tables}.

The equivalence of conditions~(\ref{thm_result_sp_1}) and~(\ref{thm_result_sp_2}) in Theorem~\ref{thm_result_sp} follows from results in \S\S\,\ref{subsec_SpGr_max}--\ref{subsec_sympl_case_end}.
The method for computing the monoid $\Gamma_I(G,H)$ for the case in Table~\ref{table_result_sympl} is described in~\ref{subsec_RBM_Ine{1}}.

\subsection{The orthogonal case}
\label{subsec_so_results}

Let $V$ be a vector space of dimension $d \ge 5$ equipped with a nondegenerate symmetric bilinear form~$\omega$.
Let $G = \Spin(V) \simeq \Spin_d$ be the spinor group determined by~$V$ and~$\omega$.
Let $H \subset G$ be a connected reductive subgroup.

We fix a decomposition $V = V_1 \oplus \ldots \oplus V_r$ into a direct sum of $H$-submodules such that the summands are pairwise orthogonal with respect to the form~$\omega$ and each summand is either irreducible or weakly reducible.
We denote by $H_0$ the image of $H$ in $\SO(V)$.
For each $i = 1,\ldots, r$ we let $H_i$ be the image of $H$ in the group $\SO(V_i)$.

First we consider separately the case $I = \lbrace 1 \rbrace$.
Note that $X_{\lbrace 1 \rbrace}$ is isomorphic to the variety $\SOGr_1(V)$ mentioned in the introduction (see \S\,\ref{subsec_FV_via_compositions}).

\begin{theorem} \label{thm_result_so_Q1}
The following assertions hold.
\begin{enumerate}[label=\textup{(\alph*)},ref=\textup{\alph*}]
\item \label{thm_result_so_Q1_a}
Suppose that $r = 1$.
Then $X_{\lbrace 1 \rbrace}$ is $H$-spherical if and only if $V$ is a spherical $(H \times \FF^\times)$-module where $\FF^\times$ acts by scalar transformations.
Moreover, the above conditions hold if and only if the pair $(H,V)$ is BF-equivalent to one of the pairs in Table~\textup{\ref{table_orth_sph_mod}}.

\item \label{thm_result_so_Q1_b}
Suppose that $r = 2$.
Then $X_{\lbrace 1 \rbrace}$ is $H$-spherical if and only if $V$ is a spherical $(H \times \FF^\times \times \FF^\times)$-module where for $i = 1,2$ the $i$th factor $\FF^\times$ acts on~$V_i$ by scalar transformations.
Moreover, the above conditions hold if and only if one of the following cases occurs:
\begin{enumerate}[label=\textup{(\arabic*)},ref=\textup{\arabic*}]
\item
$H_0 = H_1 \times H_2$ and for each $i =1,2$ the pair $(H_i,V_i)$ is BF-equivalent to one of the pairs in Table~\textup{\ref{table_orth_sph_mod}};

\item
the pair $(H,V)$ is BF-equivalent to a pair in Table~\textup{\ref{table_so_Q1}}.
\end{enumerate}

\item
If $r \ge 3$ then $X_{\lbrace 1 \rbrace}$ is not $H$-spherical.
\end{enumerate}
\end{theorem}

\begin{table}[h]

\caption{}
\label{table_so_Q1}

\begin{tabular}{|c|c|c|c|}
\hline
No. & $H$ & $V$ & Note \\

\hline

\newcase & $\Sp_{2l} \times \SL_2 \times \Sp_{2m}$ & $\mathop{\FF^{2l}} \limits_1{\!} \otimes \mathop{\FF^2} \limits_2{\!} \oplus \mathop{\FF^{2m}} \limits_3{\!} \otimes \mathop{\FF^2} \limits_2{\!}$ & $l,m\ge 1$ \\

\hline

\newcase \label{spin_8+-}
& $\Spin_8$ & $\FF^8_+ \oplus \FF^8_-$ & \\

\hline

\newcase & $\SL_l \times \SL_m \times \FF^\times$ & $\Omega([\mathop{\vphantom|\FF^l} \limits_1{\!}]_{a\chi}) \oplus \Omega([\mathop{\vphantom|\FF^m} \limits_2{\!}]_{b\chi})$, $a,b \in \ZZ \setminus \lbrace 0 \rbrace$ & $l,m \ge 3$ \\

\hline

\newcase & $\SL_l \times \Sp_{2m} \times \FF^\times$ & $\Omega([\mathop{\vphantom|\FF^l} \limits_1{\!}]_{a\chi}) \oplus \Omega([\mathop{\vphantom|\FF^{2m}} \limits_2{\!}]_{b\chi})$, $a,b \in \ZZ \setminus \lbrace 0 \rbrace$ & $l \ge 3$, $m \ge 2$ \\

\hline

\end{tabular}
\end{table}

In Case~\ref{spin_8+-} of Table~\ref{table_so_Q1} the symbols $\FF^8_{\pm}$ stand for the spaces of the two half-spin representations of~$\Spin_8$.

Theorem~\ref{thm_result_so_Q1} follows from results in~\S\,\ref{subsec_SOGr_1}.

\begin{remark}
In Theorem~\ref{thm_result_so_Q1}(\ref{thm_result_so_Q1_a}), the fact that $X_{\lbrace 1 \rbrace}$ is $H$-spherical if and only if $V$ is a spherical $(H \times \FF^\times)$-module is obtained a posteriori as a result of classification.
On the contrary, in Theorem~\ref{thm_result_so_Q1}(\ref{thm_result_so_Q1_b}) the fact that $X_{\lbrace 1 \rbrace}$ is $H$-spherical if and only if $V$ is a spherical $(H \times \FF^\times \times \FF^\times)$-module is proved by a general argument, see Proposition~\ref{prop_so_Q1_V_sr_aux}.
\end{remark}

Recall from \S\,\ref{subsec_spherical_modules} the notion of weight monoid $\mathrm E_K(U)$ of a spherical $K$-module~$U$.

\begin{theorem} \label{thm_RBM_so_Q1}
The following assertions hold.
\begin{enumerate}[label=\textup{(\alph*)},ref=\textup{\alph*}]
\item \label{thm_RBM_so_Q1_a}
In the situation of Theorem~\textup{\ref{thm_result_so_Q1}(\ref{thm_result_so_Q1_a})}, let $\delta$ denote the character via which $\FF^\times$ acts on~$V$ and let $\mathrm E^0_{H \times \FF^\times}(V^*)$ be the set of indecomposable elements of $\mathrm E_{H \times \FF^\times}(V^*)$.
Identify $\Lambda^+(H \times \FF^\times)$ with $\Lambda^+(H) \oplus \ZZ \delta$ and consider the map
\[
\Lambda^+(H \times \FF^\times) \to \Lambda^+_{\lbrace 1 \rbrace}(G) \times \Lambda^+(H), \quad \lambda = \mu + k\delta \mapsto \overline \lambda = (k\pi_1; \mu).
\]
Then the set of indecomposable elements of $\Gamma_{\lbrace 1 \rbrace}(G, H)$ is
\[
\lbrace \overline \lambda \mid \lambda \in \mathrm E^0_{H \times \FF^\times}(V^*) \setminus \lbrace 2\delta \rbrace \rbrace.
\]

\item \label{thm_RBM_so_Q1_b}
In the situation of Theorem~\textup{\ref{thm_result_so_Q1}(\ref{thm_result_so_Q1_b})}, for $i=1,2$ let $\delta_i$ denote the character via which the $i$th factor $\FF^\times$ acts on $V_i$ and let $\mathrm E^0_{H \times \FF^\times \times \FF^\times}(V^*)$ be the set of indecomposable elements of $\mathrm E_{H \times \FF^\times \times \FF^\times}(V^*)$.
Identify $\Lambda^+(H \times \FF^\times \times \FF^\times)$ with $\Lambda^+(H) \oplus \ZZ \delta_1 \oplus \ZZ \delta_2$ and consider the map
\[
\Lambda^+(H \times \FF^\times \times \FF^\times) \to \Lambda^+_{\lbrace 1 \rbrace} (G) \times \Lambda^+(H), \quad \lambda = \mu + k_1\delta_1 + k_2 \delta_2 \mapsto \overline \lambda = ((k_1+k_2)\pi_1; \mu).
\]
Then the set of indecomposable elements of $\Gamma_{\lbrace 1 \rbrace}(G, H)$ is
\[
\lbrace \overline \lambda \mid \lambda \in \mathrm E^0_{H \times \FF^\times \times \FF^\times}(V^*) \rbrace.
\]
\textup(Note that $2\delta_1,2\delta_2 \in \mathrm E^0_{H \times \FF^\times \times \FF^\times}(V^*)$ and these elements give rise to the same indecomposable element $(2\pi_1;0)$ of $\Gamma_{\lbrace 1 \rbrace}(G,H)$.\textup)
\end{enumerate}
\end{theorem}

The proof of Theorem~\ref{thm_RBM_so_Q1} is given in~\S\,\ref{subsec_RBM_I={1}}.

We now turn to the case $I \ne \lbrace 1 \rbrace$.

\begin{theorem} \label{thm_result_so_odd}
Suppose that $d = 2n+1$ with $n \ge 3$, $I \subset S$ is a nonempty subset distinct from $\lbrace 1 \rbrace$, and the following properties hold:
\begin{itemize}
\item
$H$ is not intermediate between a Levi subgroup of~$G$ and its derived subgroup;

\item
$H$ is not a symmetric subgroup of~$G$.
\end{itemize}
Then the following conditions are equivalent:
\begin{enumerate}[label=\textup{(\arabic*)},ref=\textup{\arabic*}]
\item
The variety $X_I$ is $H$-spherical and $I$ is maximal with this property.

\item
the pair $(H,V)$, considered up to BF-equivalence, and the set $I$ fall into one of the cases in Table~\textup{\ref{table_result_so_odd}}.
\end{enumerate}
Moreover, Table~\textup{\ref{table_result_so_odd}} lists also the rank and indecomposable elements of the monoid $\Gamma_I(G,H)$ for each of the cases.
\end{theorem}

\begin{table}[h]

\caption{The odd orthogonal case for $I \ne \lbrace 1 \rbrace$}
\label{table_result_so_odd}

\begin{tabular}{|c|l|l|l|}
\hline
No. & Conditions & Rank & Indecomposable elements of~$\Gamma_I(G,H)$  \\

\hline

\newcase
& \multicolumn{3}{|c|}{$(\mathsf G_2, \FF^7)$} \\
\hline
\no & $I = \lbrace 1,2 \rbrace$ & $4$ & $(\pi_1;\pi_1)$, $(\pi_2;\pi_1)$, $(\pi_2; \pi_2)$, $(\pi_1 {+} \pi_2; \pi_2)$ \\
\hline
\no & $I = \lbrace 1,3 \rbrace$ & $4$ & $(\pi_1;\pi_1)$, $(\pi_3; \pi_1)$, $(\pi_3; 0)$, $(\pi_1 {+} \pi_3; \pi_2)$ \\
\hline
\no & $I = \lbrace 2,3 \rbrace$ & $4$ & $(\pi_2; \pi_1)$, $(\pi_2; \pi_2)$, $(\pi_3; \pi_1)$, $(\pi_3; 0)$ \\

\hline

\newcase
& \multicolumn{3}{|c|}{$(\Spin_7, \FF^8 \oplus \FF^1)$} \\
\hline
\no & $I = \lbrace 1,2 \rbrace$ & 5 &
$(\pi_1; \pi_3)$, $(\pi_1; 0)$, $(\pi_2; \pi_1)$, $(\pi_2; \pi_2)$, $(\pi_2; \pi_3)$
\\
\hline
\no & $I = \lbrace 3 \rbrace$ & 4 & $(\pi_3;\pi_1)$, $(\pi_3; \pi_2)$, $(\pi_3;\pi_3)$, $(\pi_3; \pi_1 {+} \pi_3)$ \\
\hline
\no & $I = \lbrace 4 \rbrace$ & 3 & $(\pi_4;\pi_1)$, $(\pi_4; 0)$, $(\pi_4; \pi_3)$ \\

\hline

\end{tabular}

\end{table}

\begin{theorem} \label{thm_result_so_even}
Suppose that $d = 2n$ with $n \ge 4$, $I \subset S$ is a nonempty subset distinct from $\lbrace 1 \rbrace$, and the following properties hold:
\begin{itemize}
\item
$H$ is not intermediate between a Levi subgroup of~$G$ and its derived subgroup;

\item
$H$ is not a symmetric subgroup of~$G$.
\end{itemize}
Then the following conditions are equivalent:
\begin{enumerate}[label=\textup{(\arabic*)},ref=\textup{\arabic*}]
\item
The variety $X_I$ is $H$-spherical.

\item
Up to an automorphism of~$G$, the pair $(H,V)$, considered up to BF-equivalence, and the set $I$ fall into one of the cases in Table~\textup{\ref{table_result_so_even}}.
\end{enumerate}
Moreover, Table~\textup{\ref{table_result_so_even}} lists also the rank and indecomposable elements of the monoid $\Gamma_I(G,H)$ for each of the cases.
\end{theorem}

\begin{table}[h]

\caption{The even orthogonal case for $I \ne \lbrace 1 \rbrace$}
\label{table_result_so_even}

\begin{tabular}{|c|l|l|l|}
\hline
No. & Conditions & Rank & Indecomposable elements of~$\Gamma_I(G,H)$  \\

\hline

\newcase
& \multicolumn{3}{|c|}{$(\SO_{2l} \times \SO_{2m+1}, \mathop{\FF^{2l}} \limits_1{\!} \oplus \mathop{\FF^{2m+1}} \limits_2{\!} \oplus \FF^1)$, $l \ge 2$, $m \ge 1$, $2l+2m+2 = 2n$} \\
\hline
\no & $I = \lbrace n \rbrace$ & $\min(2l,2m{+}1) {+} 1$
&
\renewcommand{\tabcolsep}{0pt}%
\begin{tabular}{l}
$(\pi_n; \pi_{l-1} {+} \pi'_{m})$,
$(\pi_n; \pi_{l} {+} \pi'_{m})$,\\[-2pt]
$(2\pi_n; \pi_{l-1} {+} \pi_{l} {+} \pi'_{m-1})$, \\[-2pt]
$(2\pi_n; \pi_{l -k} {+} \pi'_{m-k})$ for $2 {\le} k {\le} [\frac{\min(2l,2m+1)}2]$,\\[1pt]
$(2\pi_n;\pi_{l - k} {+} \pi'_{m+1-k})$ for $2 {\le} k {\le} [\frac{\min(2l,2m+1)+1}2]$\\[1pt]
\end{tabular}\\

\hline

\newcase
& \multicolumn{3}{|c|}{$(\SO_{2m+1} \times \FF^\times, \FF^{2m+1} \oplus \Omega(\FF^1_\chi) \oplus \FF^1)$, $m \ge 2$, $2m+4 = 2n$} \\
\hline
\no \label{chi/2_no1}
& $I = \lbrace n \rbrace$ & 3 & $(\pi_{n}; \pi_m {+} \frac{\chi}2)$, $(\pi_{n}; \pi_m {-} \frac{\chi}2)$, $(2\pi_{n}; \pi_{m-1})$ \\

\hline

\newcase
& \multicolumn{3}{|c|}{$(\Spin_7 \times \SO_{2l}, \mathop{\FF^8} \limits_1{\!} \oplus \mathop{\FF^{2l}} \limits_2{\!})$, $l \ge 2$, $2l+8 = 2n$} \\
\hline
\no & $I = \lbrace n \rbrace$ & $\min(7,2l){+}1$ &
\renewcommand{\tabcolsep}{0pt}%
\begin{tabular}{l}
$(\pi_n; \pi_3 {+} \pi'_l)$, $(\pi_n; \pi_1 {+} \pi'_{l-1})$, $(\pi_n; \pi'_{l-1})$, \\[-2pt]
$(2\pi_n; \pi_2 {+} \pi'_{l-2})$, $(2\pi_n; \pi_1 {+} \pi'_{l-2})$,\\[-2pt]
$(2\pi_n; \pi_1 {+} \pi'_{l-3})$, ($l {\ge} 3$), $(2\pi_n; \pi'_{l-4})$ ($l {\ge} 4$), \\[-2pt]
$(3\pi_n; \pi_2 {+} \pi'_{l} {+} \pi'_{l-3})$ ($l {\ge} 3$)
\end{tabular}
\\

\hline

\newcase
& \multicolumn{3}{|c|}{$(\Spin_7 \times \FF^\times, \FF^8 \oplus \Omega(\FF^1_{\chi}))$} \\
\hline
\no & $I = \lbrace 2 \rbrace$ & 5 & $(\pi_2; \pi_1)$, $(\pi_2; \pi_2)$, $(\pi_2; 0)$, $(\pi_2; \pi_3 {+} \chi)$, $(\pi_2; \pi_3 {-} \chi)$ \\
\hline
\no \label{chi/2_no2}
& $I = \lbrace 5 \rbrace$ & 3 & $(\pi_5; \pi_1 {+} \frac{\chi}2)$, $(\pi_5,\frac{\chi}2)$, $(\pi_5; \pi_3 {-} \frac{\chi}2)$ \\

\hline

\newcase
& \multicolumn{3}{|c|}{$(\Spin_7, \FF^8 \oplus \FF^1 \oplus \FF^1)$} \\
\hline

\no & $I = \lbrace 5 \rbrace$ & 3 & $(\pi_5;\pi_1)$, $(\pi_5; 0)$, $(\pi_5; \pi_3)$ \\

\hline

\newcase \label{Spin7xGL2}
& \multicolumn{3}{|c|}{$(\Spin_7 \times \SL_2 \times \FF^\times, \mathop{\vphantom|\FF^8} \limits_1{\!} \oplus \Omega([\mathop{\vphantom|\FF^2} \limits_2{\!}]_\chi))$} \\

\hline

\no & $I = \lbrace 6 \rbrace$ & 6 &
\renewcommand{\tabcolsep}{0pt}%
\begin{tabular}{l}
$(\pi_6; \pi_1 {+} \pi'_1 {+} \chi)$, $(\pi_6; \pi'_1 {+} \chi)$, $(\pi_6; \pi_3 {+} \chi)$,\\[-2pt]
$(\pi_6; \pi_3 {-} \chi)$, $(2\pi_6; \pi_1 {+} \chi)$, $(2\pi_6; 2\chi)$
\end{tabular}
\\

\hline

\newcase
& \multicolumn{3}{|c|}{$(\mathsf G_2 \times \SO_3, \mathop{\FF^7} \limits_1{\!} \oplus \mathop{\FF^3} \limits_2{\!})$} \\
\hline
\no & $I = \lbrace 5 \rbrace$ & 4 & $(\pi_5; \pi_1 {+} \pi'_1)$, $(\pi_5; \pi'_1)$, $(2\pi_5; \pi_1)$, $(2\pi_5; \pi_2)$\\

\hline

\newcase
& \multicolumn{3}{|c|}{$(\mathsf G_2, \FF^7 \oplus \FF^1)$} \\
\hline
\no & $I = \lbrace 4 \rbrace$ & 2 & $(\pi_4; \pi_1)$, $(\pi_4; 0)$ \\

\hline

\end{tabular}
\end{table}

For notation used in Tables~\ref{table_result_so_odd} and~\ref{table_result_so_even}, we refer to \S\,\ref{subsec_notation_in_tables}.

In Theorems~\ref{thm_result_so_odd} and~\ref{thm_result_so_even}, the equivalence of conditions~(1) and~(2) follows from results in~\S\S\,\ref{subsec_SOGr_max_r=1}--\ref{subsec_orth_case_end}.
The method for computing the monoid $\Gamma_I(G,H)$ for each of the cases in Tables~\ref{table_result_so_odd} and~\ref{table_result_so_even} is described in~\S\,\ref{subsec_RBM_Ine{1}}.

\subsection{Notation and conventions used in Tables~\texorpdfstring{\ref{table_result_sympl}}{5},\,%
\texorpdfstring{\ref{table_result_so_odd}}{7}, and~\texorpdfstring{\ref{table_result_so_even}}{8}}
\label{subsec_notation_in_tables}

The symbol $\delta_i^j$ denotes the Kronecker delta, that is, $\delta_i^j = 1$ for $i = j$ and $\delta_i^j = 0$ otherwise.

Whenever an element $(\lambda; \mu)$ in the last column is followed by a parenthesis containing an inequality on parameters, this means that $(\lambda; \mu)$ is an indecomposable element of $\Gamma_I(G,H)$ if and only if the inequality is satisfied.

In all the tables under consideration, the group $H$ contains at most three factors different from $\FF^\times$; we write $\pi_i$ (resp.~$\pi'_i$, $\pi''_i$) for the $i$th fundamental weight of the first (resp. second, third) factor (for $\SO_4$, the fundamental weights have numbers~$1$ and~$2$).
For convenience in certain formulas, we put $\pi_0 = \pi'_0 = 0$.
We point out that the usage of the symbols $\pi_i$ to denote fundamentals weights of both $G$ and $H$ simultaneously does not cause any ambiguity.

In Cases~\ref{chi/2_no1} and~\ref{chi/2_no2} of Table~\ref{table_result_so_even}, the symbol $\chi/2$ stands for the character of the preimage of $\FF^\times$ in $\Spin(V)$ such that $2 \cdot (\chi/2) = \chi$.

In Case~\ref{Spin7xGL2} of Table~\ref{table_result_so_even}, there are two conjugacy classes in $G$ of subgroups $H$ having the indicated type, and the variety $X_I$ is $H$-spherical only for one of them.
See \S\,\ref{subsec_so_prelim_remarks} for the convention choosing the right conjugacy class.

\section{Main tools}
\label{sect_main_tools}

\subsection{Nil-equivalence relation on \texorpdfstring{$\mathscr F(G)$}{F(G)} and its properties}
\label{subsec_NE-relation}

Let $\mathscr F(G)$ denote the set of nontrivial flag varieties of the group~$G$.

Given a parabolic subgroup $P \subset G$, consider a Levi decomposition $P = L P^u$ where $L$ is a Levi subgroup and $P^u$ is the unipotent radical of~$P$.
A well-known result of Richardson \cite[Proposition~6(c)]{Rich74} asserts that $\mathfrak p^u$ has an open orbit for the adjoint action of~$P$. Let $O_P$ denote this open orbit and put
\[
\mathscr N(G/P) = GO_P \subset \mathfrak g.
\]
Then $\mathscr N(G/P)$ is a nilpotent orbit in~$\mathfrak g$.

We say that two varieties $X_1, X_2 \in \mathscr F(G)$ are \textit{nil-equivalent} (notation $X_1 \sim X_2$) if $\mathscr N(X_1) = \mathscr N(X_2)$.

\begin{remark}
If $P,Q \subset G$ are two \textit{associated} parabolic subgroups (that is, their Levi subgroups are conjugate in~$G$) then $\mathscr N(G/P) = \mathscr N(G/Q)$ by~\cite[Theorem~2.7]{JR}.
Hence $G/P$ and $G/Q$ are automatically nil-equivalent in this case.
\end{remark}

Now let $K \subset G$ be an arbitrary connected reductive subgroup.

\begin{theorem}[{\cite[Theorem~1.3]{AvP1}}] \label{thm_nil-equiv}
Suppose that $X_1, X_2 \in \mathscr F(G)$ and $X_1 \sim X_2$.
Then the following conditions are equivalent:
\begin{enumerate}[label=\textup{(\arabic*)},ref=\textup{\arabic*}]
\item
$X_1$ is $K$-spherical.

\item
$X_2$ is $K$-spherical.
\end{enumerate}
\end{theorem}

For every $X \in \mathscr F(G)$, let $\bl X \br$ denote the nil-equivalence class of~$X$.

The set $\mathscr F(G) /\! \sim$ of all nil-equivalence classes is naturally equipped with a partial order $\preccurlyeq$ defined as follows: $\bl X_1 \br \preccurlyeq \bl X_2 \br$ (or $\bl X_2 \br \succcurlyeq \bl X_1 \br$) if and only if $\mathscr N(X_1) \subset \overline{\mathscr N(X_2)}$ for all $X_1, X_2 \in \mathscr F(G)$. We shall also write $\bl X_1 \br \prec \bl X_2 \br$ (or $\bl X_2 \br \succ \bl X_1 \br$) when $\bl X_1 \br \preccurlyeq \bl X_2 \br$ but $\bl X_1 \br \ne \bl X_2 \br$.

The following theorem, which traces back to~\cite[Theorem~5.8]{Pet}, is based on a result of Losev~\cite{Los}.

\begin{theorem}[{\cite[Theorem~1.4]{AvP1}}] \label{thm_sph_descent}
Suppose that $X_1, X_2 \in \mathscr F(G)$, $\bl X_1 \br \prec \bl X_2 \br$, and $X_2$ is $K$-spherical.
Then $X_1$ is also $K$-spherical.
\end{theorem}

\begin{remark}
In~\cite{AvP1}, Theorem~\ref{thm_nil-equiv} was proved by showing that, for a given $X \in \mathscr F(G)$, the $K$-sphericity of $X$ is equivalent to the action of $K$ on $\mathscr N(X)$ being coisotropic (see the definition in loc.~cit.), and the core of the proof of Theorem~\ref{thm_sph_descent} was the fact that the property of being coisotropic is inherited by all adjoint orbits in~$\mathfrak g$ lying in the closure of a given one.
The latter fact has been recently generalized in~\cite{PaYa} to actions of higher corank on adjoint orbits in~$\mathfrak g$, which provides a more general framework for the results discussed in this subsection.
\end{remark}

\subsection{Compositions and partitions}

Let $d$ be a positive integer.

A tuple $(a_1, \ldots, a_p)$ of positive integers satisfying $a_1 +\nobreak \ldots + a_p = d$ is called a \textit{composition} of~$d$.
Given a composition $\mathbf{a} = (a_1, \ldots, a_p)$ of~$d$, each number $a_i$ is said to be a \textit{part} of~$\mathbf{a}$. For every part $x$ of $\mathbf{a}$, its \textit{multiplicity} is the cardinality of the set $\lbrace i \mid a_i = x \rbrace$.

We say that a composition $(a_1,\ldots, a_p)$ is \textit{trivial} if $p = 1$ and \textit{nontrivial} if $p \ge 2$.

A composition $(a_1, \ldots, a_p)$ of~$d$ is said to be \textit{symmetric} if $a_i = a_{p+1-i}$ for all $i = 1, \ldots, p$.

A composition $(a_1, \ldots, a_p)$ of~$d$ is said to be a \textit{partition} if $a_1 \ge \ldots \ge a_p$.
We let $\mathcal P(d)$ denote the set of all partitions of~$d$.

If $b_1 > b_2 > \ldots > b_l$ are all parts of a partition $\mathbf a \in \mathcal P(d)$ and $k_1,k_2,\ldots, k_l$ are their multiplicities, then $\mathbf a$ will be also written as $[b_1^{k_1},b_2^{k_2},\ldots,b_l^{k_l}]$.

Given a partition $(a_1,\ldots, a_p)$ of~$d$, it is often convenient to assume that $a_i = 0$ for $i > p$.

The set $\mathcal P(d)$ carries a natural partial order defined as follows.
For two partitions $\mathbf{a} = (a_1, \ldots, a_p)$ and $\mathbf{b} = (b_1, \ldots, b_q)$, we write $\mathbf a \preccurlyeq \mathbf b$ (or $\mathbf b \succcurlyeq \mathbf a$) if
\[
a_1 + \ldots + a_i \le b_1 + \ldots + b_i \quad \text{for all} \quad i = 1, \ldots, d.
\]
We shall also write $\mathbf a \prec \mathbf b$ (or $\mathbf b \succ \mathbf a$) if $\mathbf a \preccurlyeq \mathbf b$ and $\mathbf a \ne \mathbf b$.

For each $\varepsilon \in \lbrace \pm 1 \rbrace$, we define the subset $\mathcal P_\varepsilon(d) \subset \mathcal P(d)$ by the formula
\[
\mathcal P_\varepsilon(d) =
\begin{cases}
\lbrace \mathbf{a} \in \mathcal P(d) \mid \text{each odd part of $\mathbf{a}$ has even multiplicity} \rbrace & \text{if} \ \varepsilon = -1;\\
\lbrace \mathbf{a} \in \mathcal P(d) \mid \text{each even part of $\mathbf{a}$ has even multiplicity} \rbrace & \text{if} \ \varepsilon = 1.
\end{cases}
\]

A partition $\mathbf a \in \mathcal P_1(d)$ is said to be \textit{very even} if all parts of $\mathbf a$ are even.
Clearly, very even partitions occur only when $d = 4k$ for an integer~$k$.

\subsection{Flag varieties for the symplectic and orthogonal group}
\label{subsec_FV_via_compositions}

In this subsection, we discuss the description of flag varieties of the symplectic and orthogonal group as varieties of flags of isotropic subspaces in the space of the tautological representation.
This description will be widely used in the remaining part of the paper.

Let $V$ be a vector space of dimension $d > 0$ equipped with a nondegenerate bilinear form $\omega$ such that $\omega(x,y) = \varepsilon \omega(y,x)$ for all $x,y \in V$, where $\varepsilon \in \lbrace \pm 1 \rbrace$.
In other words, $\omega$ is symmetric for $\varepsilon = 1$ and skew-symmetric for $\varepsilon = -1$.
Let $\widetilde G_\varepsilon = \Aut(V,\omega)$ be the subgroup of $\GL(V)$ consisting of all elements preserving~$\omega$, so that $\widetilde G_\varepsilon = \mathrm{O}(V)$ for $\varepsilon = 1$ and $\widetilde G_\varepsilon = \Sp(V)$ for $\varepsilon = -1$.
We also put $G_\varepsilon = (\widetilde G_\varepsilon)^0$, so that  $G_\varepsilon = \SO(V)$ for $\varepsilon = 1$ and $G_\varepsilon = \widetilde G_\varepsilon = \Sp(V)$ for $\varepsilon = -1$.

For every symmetric composition $\mathbf a = (a_1, \ldots, a_p)$ of~$d$, let $\Fl^{(\varepsilon)}_{\mathbf a}(V)$ be the set of all tuples $(V_0, V_1, \ldots, V_{[p/2]})$ where $V_0 = \lbrace 0 \rbrace$ and $V_1, \ldots, V_{[p/2]}$ are isotropic subspaces of $V$ such that $V_1 \subset \ldots \subset V_{[p/2]}$ and $\dim V_i = a_1 + \ldots + a_i$ for all $i = 1, \ldots, [p/2]$.
It is well known that $\Fl^{(\varepsilon)}_{\mathbf a}(V)$ is a projective homogeneous $\widetilde G_\varepsilon$-variety.
If $[p/2] = 1$ then the points of $\Fl^{(\varepsilon)}_{\mathbf a}(V)$ are naturally identified with the $a_1$-dimensional isotropic subspaces of $V$; in this situation $\Fl^{(\varepsilon)}_{\mathbf a}(V)$ is also called an \textit{isotropic Grassmannian}.

If $G_{\varepsilon} \simeq \Sp_{2n}$, $n \ge 1$, then for every symmetric composition $\mathbf a = (a_1, \ldots, a_p)$ of $2n$ the variety $\Fl^{(-1)}_{\mathbf a}(V)$ is a single $G_{\varepsilon}$-orbit isomorphic to $X_I$ with
\[
I = \lbrace a_1, a_1 + a_2, \ldots, a_1 + \ldots +  a_{[p/2]} \rbrace.
\]
In this case, we shall use the notation $\SpFl_{\mathbf a}(V) = \Fl^{(-1)}_{\mathbf a}(V)$.

If $G_{\varepsilon} \simeq \SO_{2n+1}$, $n \ge 1$, then for every symmetric composition $\mathbf a = (a_1, \ldots, a_p)$ of $2n+1$ the variety $\Fl^{(1)}_{\mathbf a}(V)$ is a single $G_{\varepsilon}$-orbit isomorphic to $X_I$ with
\[
I = \lbrace a_1, a_1 + a_2, \ldots, a_1 + \ldots + a_{[p/2]} \rbrace.
\]
In this case, we shall use the notation $\SOFl_{\mathbf a}(V) = \Fl^{(1)}_{\mathbf a}(V)$.

If $G_\varepsilon \simeq \SO_{2n}$, $n \ge 2$, then, given a symmetric composition $\mathbf a = (a_1, \ldots, a_p)$ of~$2n$, there are the following possibilities:
\begin{itemize}
\item
if $p$ is odd then $\Fl^{(1)}_{\mathbf a}(V)$ is a single $G_{\varepsilon}$-orbit isomorphic to $X_I$ with
\[
I = \begin{cases}
\lbrace a_1, a_1 + a_2, \ldots, a_1 + \ldots + a_{[p/2]}\rbrace & \text{if} \ a_{[p/2]+1} \ge4; \\
\lbrace a_1, a_1 + a_2, \ldots, a_1 + \ldots + a_{[p/2]-1}, n-1, n \rbrace & \text{if} \ a_{[p/2]+1} = 2;
\end{cases}
\]
in this case we shall use the notation $\SOFl_{\mathbf a}(V) = \Fl^{(1)}_{\mathbf a}(V)$;

\item
if $p = 2q$ is even and $a_q = 1$ then $\Fl^{(1)}_{\mathbf a}(V)$ is a single $\SO_{2n}$-orbit isomorphic to $X_I$ with
\[
I = \lbrace a_1, a_1+a_2, \ldots a_1 + \ldots + a_{q-1}, n-1,n \rbrace;
\]
in this case we shall use the notation $\SOFl_{\mathbf a}(V) = \Fl^{(1)}_{\mathbf a}(V)$;

\item
if $p = 2q$ is even and $a_q \ge 2$ then $\Fl^{(1)}_{\mathbf a}(V)$ is a disjoint union of two $\SO_{2n}$-orbits $\SOFl^+_{\mathbf a}(V)$ and $\SOFl^{-}_{\mathbf a}(V)$ such that, by convention in the choice of signs, $\SOFl^+_{\mathbf a}(V) \simeq X_I$ with
\[
I = \lbrace a_1, a_1 + a_2, \ldots, a_1 + \ldots + a_{q-1}, n \rbrace
\]
and $\SOFl^{-}_{\mathbf a}(V) \simeq X_I$ with
\[
I = \lbrace a_1, a_1 + a_2, \ldots, a_1 + \ldots + a_{q-1}, n-1 \rbrace.
\]
\end{itemize}

For isotropic Grassmannians, we shall use special notation.
Namely, for the variety of $k$-dimensional isotropic subspaces of~$V$ we shall write $\SpGr_k(V)$ if $\varepsilon = -1$ and $\SOGr_k(V)$ if $\varepsilon = 1$.
If $k = [d/2]$ is maximal possible, we shall also write $\SpGr_{\max}(V)$ and $\SOGr_{\max}(V)$ instead of $\SpGr_k(V)$ and $\SOGr_k(V)$, respectively.
Finally, if $\varepsilon = 1$ and $d = 2k$ is even then we put $\SOGr_{\max}^+(V) = \SOGr_k^+(V) = \SOFl_{(k,k)}^+(V)$ and $\SOGr_{\max}^-(V) = \SOGr_k^-(V) = \SOFl_{(k,k)}^-(V)$.

Observe that a flag variety $\SpFl_{\mathbf a}(V)$ (or $\SOFl_{\mathbf a}(V)$) is nontrivial if and only if so is the corresponding symmetric composition~$\mathbf a$.

\subsection{Nilpotent orbits in the symplectic and orthogonal Lie algebra and their relation to flag varieties}
\label{subsec_NO_in_sp_and_so}

We retain the notation introduced in \S\,\ref{subsec_FV_via_compositions}.

For every partition $\mathbf a = (a_1, \ldots, a_p) \in \mathcal P_\varepsilon(d)$, let $O_{\mathbf a}$ be the set of all matrices in $\mathfrak g_\varepsilon$ whose Jordan normal form has zeros on the diagonal and the block sizes are $a_1, \ldots, a_p$ up to permutation.

In the next theorem, which provides a parametrization of nilpotent orbits in the symplectic and orthogonal Lie algebra, parts~(\ref{thm_partitions_vs_orbits_a}) and~(\ref{thm_partitions_vs_orbits_b}) were obtained in~\cite[Ch.~II, \S\,1]{Ger} and~\cite[Ch.~IV, 2.27(ii)]{SpSt}, respectively; see also \cite[\S\,5.1]{CM}.

\begin{theorem} \label{thm_partitions_vs_orbits}
The following assertions hold.
\begin{enumerate}[label=\textup{(\alph*)},ref=\textup{\alph*}]
\item \label{thm_partitions_vs_orbits_a}
The map $\mathbf a \mapsto O_{\mathbf a}$ is a bijection between the set $\mathcal P_{\varepsilon}(d)$ and the nilpotent $\widetilde G_\varepsilon$-orbits in $\mathfrak g_{\varepsilon}$.

\item \label{thm_partitions_vs_orbits_b}
For every $\varepsilon \in \lbrace \pm 1 \rbrace$ and $\mathbf a \in \mathcal P_\varepsilon(d)$, the set $O_{\mathbf a}$ is a single $G_\varepsilon$-orbit unless $\varepsilon = 1$ and $\mathbf a$ is very even, in which case $O_{\mathbf a}$ is a union of two $G_\varepsilon$-orbits $O_{\mathbf a}^+$ and $O_{\mathbf a}^{-}$.
\end{enumerate}
\end{theorem}

The following theorem was obtained in~\cite[Ch.~III, \S\,3]{Ger} and \cite[Theorem~3.10]{Hes}; see also~\cite[Theorem~6.2.5]{CM}.

\begin{theorem} \label{thm_nilp_orb_inclusions}
Suppose that $\mathbf a, \mathbf b \in \mathcal P_\varepsilon(d)$ and $N_{\mathbf a}$ \textup(resp. $N_{\mathbf b}$\textup) is a $G_\varepsilon$-orbit in $O_{\mathbf a}$ \textup(resp.~$O_{\mathbf b}$\textup).
Then the following conditions are equivalent:
\begin{enumerate}[label=\textup{(\arabic*)},ref=\textup{\arabic*}]
\item
$N_{\mathbf a} \subsetneq \overline N_{\mathbf b}$.

\item
$\mathbf a \prec \mathbf b$.
\end{enumerate}
\end{theorem}

\begin{proposition}[{\cite[\S\,2.2, Main Lemma]{Kem}}] \label{prop_collapse}
For every $\mathbf a \in \mathcal P(d)$ and $\varepsilon \in \lbrace \pm 1 \rbrace$, there exists a unique $\mathbf a^\sharp_\varepsilon \in \mathcal P_\varepsilon(d)$ such that $\mathbf a^\sharp_\varepsilon \preccurlyeq \mathbf a$ and $\mathbf b \preccurlyeq \mathbf a^\sharp_\varepsilon$ for all $\mathbf b \in \mathcal P_\varepsilon(d)$ with $\mathbf b \preccurlyeq \mathbf a$.
\end{proposition}

We refer to the partition $\mathbf a^\sharp_\varepsilon$ as the $\varepsilon$-\textit{collapse} of~$\mathbf a$.

The proof of the above proposition in~\cite{Kem} contains an explicit algorithm for constructing the partition $\mathbf a^\sharp_\varepsilon$ from the initial partition $\mathbf a = (a_1,\ldots, a_p) \in \mathcal P(d)$.
We reproduce its steps below.
\begin{enumerate}[label=\textup{\arabic*)},ref=\textup{\arabic*}]
\item
If $\mathbf a \in \mathcal P_{\varepsilon}(d)$ then put $\mathbf a^\sharp_\varepsilon = \mathbf a$ and exit.

\item
If $\mathbf a \notin \mathcal P_{\varepsilon}(d)$ then define
\[
m = \max (\lbrace 0 \rbrace \cup \lbrace i \ge 1 \mid (a_1,\ldots, a_i) \in \mathcal P_{\varepsilon}(a_1 + \ldots + a_i) \rbrace).
\]
One automatically has $\varepsilon(-1)^{a_{m+1}} = 1$ and $a_{m+1} > a_{m+2}$.

\item
Define
\[
l = \min\lbrace j \ge m+2 \mid \varepsilon(-1)^{a_j} = 1 \rbrace,
\]
the minimum always exists.

\item
Define a new partition $\mathbf a' = (a'_1, a'_2, \ldots)$ as follows: $a'_{m+1} = a_{m+1}-1$, $a'_l = a_l + 1$, $a'_i = a_i$ for $i \ne m+1, l$.

\item
Repeat the algorithm for~$\mathbf a'$.
\end{enumerate}

For each composition $\mathbf a = (a_1, \ldots, a_p)$ of~$d$ one defines the \textit{dual} partition $\mathbf a^\top = (\widehat a_1, \ldots, \widehat a_q)$ of~$d$ by the following rule:
\[
\widehat a_i = |\lbrace j \mid a_j \ge i \rbrace|, \qquad i = 1, \ldots, q.
\]
The operation $\mathbf a \mapsto \mathbf a^\top$ is an involution on the set~$\mathcal P(d)$.

\begin{proposition}[{\cite[\S\,3.4, Corollary~2]{Kem}}] \label{prop_induction_of_orbits}
Suppose that $\mathbf a$ is a symmetric composition of~$d$ and $X$ is a $G_\varepsilon$-orbit in $\Fl^\varepsilon_{\mathbf a}(V)$.
Then $\mathscr N(X)$ is a $G_\varepsilon$-orbit in $O_{(\mathbf a^\top)^\sharp_\varepsilon}$.
\end{proposition}

\begin{corollary} \label{crl_partial_order}
Suppose that $\mathbf a, \mathbf b$ are two symmetric compositions of~$d$, $X$ is a $G_\varepsilon$-orbit in $\Fl^{\varepsilon}_{\mathbf a}(V)$, and $Y$ is a $G_\varepsilon$-orbit in $\Fl^{\varepsilon}_{\mathbf b}(V)$.
Then the following conditions are equivalent:
\begin{enumerate}[label=\textup{(\arabic*)},ref=\textup{\arabic*}]
\item
$\bl X \br \prec \bl Y \br$.

\item
$(\mathbf a^\top)^\sharp_\varepsilon \prec (\mathbf b^\top)^\sharp_\varepsilon$.
\end{enumerate}
\end{corollary}

\begin{proof}
This follows from Theorem~\ref{thm_nilp_orb_inclusions} and Proposition~\ref{prop_induction_of_orbits}.
\end{proof}

\subsection{An analysis of the partial order on \texorpdfstring{$\mathscr F(G) / \sim$}{F(G)/\textasciitilde} for \texorpdfstring{$G \simeq \Sp_{2n}$}{G = Sp\_2n}}
\label{subsec_po_sp2n}

In this subsection, we assume that $G = \Sp(V)$ with $\dim V = 2n$ and $n \ge 2$.
Throughout this subsection, $\varepsilon = -1$.

\begin{proposition} \label{prop_P(V)_is_minimal}
Suppose that $X \in \mathscr F(G)$.
Then $\bl X \br \succcurlyeq \bl \PP(V) \br$.
Moreover, $\PP(V)$ is the unique element in $\bl \PP(V) \br$ for $n \ge 3$ and $\bl \PP(V) \br = \lbrace \PP(V), \SpGr_{\max}(V) \rbrace$ for $n = 2$.
\end{proposition}

\begin{proof}
We apply Corollary~\ref{crl_partial_order}.
Put $\mathbf a = (1,2n-2,1)$.
Then $\SpFl_{\mathbf a}(V) = \SpGr_1(V) \simeq \PP(V)$, $\mathbf a^\top = [3,1^{2n-3}]$, and $(\mathbf a^\top)^\sharp_\varepsilon = [2^2,1^{2n-4}]$.
Let $\mathbf b$ be the symmetric composition of $2n$ such that $X = \SpFl_{\mathbf b}(V)$ and assume $\mathbf b \ne \mathbf a$.
If $|\mathbf b| = 2$ then $\mathbf b = (n,n)$ and $\mathbf b^\top = (\mathbf b^\top)^\sharp_\varepsilon = [2^n]$, so that $(\mathbf b^\top)^\sharp_\varepsilon = (\mathbf a^\top)^\sharp_\varepsilon$ for $n = 2$ and $(\mathbf b^\top)^\sharp_\varepsilon \succ (\mathbf a^\top)^\sharp_\varepsilon$ for $n \ge 3$.
If $|\mathbf b | = 3$ then $\mathbf b^\top$ is of the form $(3,3,\ldots)$, hence $(\mathbf b^\top)^\sharp_\varepsilon$ is also of the form $(3,3,\ldots)$, hence $(\mathbf b^\top)^\sharp_\varepsilon \succ (\mathbf a^\top)^\sharp_\varepsilon$.
If $|\mathbf b | = l \ge 4$, then $\mathbf b^\top$ is of the form $(l,\ldots)$, hence $(\mathbf b^\top)^\sharp_\varepsilon \succcurlyeq [4,1^{2n-4}] \succ (\mathbf a^\top)^\sharp_\varepsilon$.
\end{proof}

\begin{proposition} \label{prop_SpGr2_partial_order}
Suppose that $X \in \mathscr F(G)$ and $X \ne \PP(V)$.
Then $X = \SpGr_{\max}(V)$, or $X = \SpGr_2(V)$, or $\bl X \br \succ \bl \SpGr_2(V) \br$.
\end{proposition}

\begin{proof}
If $n = 2$ then $\bl \PP(V) \br = \bl \SpGr_2(V) \br$ and the claim is implied by Proposition~\ref{prop_P(V)_is_minimal}, hence in what follows we assume $n \ge 3$.
Put $\mathbf a = (2, 2n-4, 2)$.
Then $\SpFl_{\mathbf a}(V) = \SpGr_2(V)$ and $\mathbf a^\top = (\mathbf a^\top)^\sharp_\varepsilon = [3^2,1^{2n-6}]$.
Let $\mathbf b$ be the symmetric composition of $2n$ such that $X = \SpFl_{\mathbf b}(V)$.
If $|\mathbf b| = 2$ then $X = \SpGr_{\max}(V)$.
Now consider the case $|\mathbf b| = 3$, so that $\mathbf b = (b_1,b_2,b_1)$.
As $X \ne \PP(V)$, we have $b_1 \ne 1$.
If $b_1 = 2$ then $X = \SpGr_2(V)$.
If $b_1 \ge 3$ then $\mathbf b^\top$ is of the form $(3,3,3,\ldots)$ or $(3,3,2,\ldots)$, hence $(\mathbf b^\top)^\sharp_\varepsilon$ is also of the form $(3,3,3,\ldots)$ or $(3,3,2,\ldots)$, which implies $(\mathbf b^\top)^\sharp_\varepsilon \succ (\mathbf a^\top)^\sharp_\varepsilon$.
If $|\mathbf b| \ge 4$ then $(\mathbf b^\top)^\sharp_\varepsilon \succcurlyeq [4,2,1^{2n-6}] \succ (\mathbf a^\top)^\sharp_\varepsilon$.
\end{proof}

\subsection{An analysis of the partial order on \texorpdfstring{$\mathscr F(G) / \sim$}{F(G)/\textasciitilde} for \texorpdfstring{$G \simeq \Spin_{2n+1}$}{G = Spin\_{2n+1}}}

In this subsection, we assume that $G = \Spin(V)$ with $\dim V = 2n+1$ and $n \ge 1$.
Throughout this subsection, $\varepsilon = 1$.

\begin{proposition} \label{prop_so_odd_Q1_is_minimal}
Suppose that $X \in \mathscr F(G)$.
Then $\bl X \br \succcurlyeq \bl \SOGr_1(V) \br$.
Moreover, $\SOGr_1(V)$ is the unique element in $\bl \SOGr_1(V) \br$ for $n \ne 2$ and
\[
\bl \SOGr_1(V) \br =\lbrace \SOGr_1(V), \SOGr_{\max}(V) \rbrace
\]
for $n = 2$.
\end{proposition}

\begin{proof}
If $n = 1$ then $\mathscr F(G) = \lbrace \SOGr_1(V) \rbrace$ and the assertion holds trivially, hence in what follows we assume $n \ge 2$.
Put $\mathbf a = (1, 2n-1, 1)$. Then $\SOFl_{\mathbf a}(V) = \SOGr_1(V)$ and $\mathbf a^\top = (\mathbf a^\top)^\sharp_\varepsilon = [3,1^{2n-2}]$.
Let $\mathbf b$ be the symmetric composition of $2n+1$ such that $X = \SOFl_{\mathbf b}(V)$ and assume $\mathbf b \ne \mathbf a$.
First consider the case $|\mathbf b | = 3$, so that $\mathbf b = (b_1,b_2,b_1)$ with $b_1 \ge 2$.
If $b_2 = 1$ then $\mathbf b^\top = [3,2^{n-1}]$, so that $(\mathbf b^\top)^\sharp_\varepsilon = (\mathbf a^\top)^\sharp_\varepsilon$ for $n = 2$ and $(\mathbf b^\top)^\sharp_\varepsilon \succ (\mathbf a^\top)^\sharp_\varepsilon$ for $n \ge 3$.
If $b_2 > 1$ then $(\mathbf b^\top)^\sharp_\varepsilon$ is of the form $(3,3,\ldots)$, hence $(\mathbf b^\top)^\sharp_\varepsilon \succ (\mathbf a^\top)^\sharp_\varepsilon$.
In the case $|\mathbf b| = l \ge 5$ the partition $\mathbf b^\top$ has the form $(l, \ldots)$, hence $(\mathbf b^\top)^\sharp_\varepsilon \succcurlyeq [5, 1^{2n-4}] \succ (\mathbf a^\top)^\sharp_\varepsilon$.
\end{proof}

\begin{proposition} \label{prop_so_odd_SOGr_2}
Suppose that $n \ge 2$, $X \in \mathscr F(G)$, and $X \ne \SOGr_1(V)$.
Then $X = \SOGr_{\max}(V)$, or $X = \SOGr_2(V)$, or $\bl X \br \succ \bl \SOGr_2(V) \br$.
\end{proposition}

\begin{proof}
If $n = 2$ then $\bl \SOGr_1(V) \br = \bl \SOGr_{\max}(V) \br$ and the claim is implied by Proposition~\ref{prop_so_odd_Q1_is_minimal}, hence in what follows we assume $n \ge 3$.
Put $\mathbf a = (2, 2n-3, 2)$.
Then $\SOFl_{\mathbf a}(V) = \SOGr_2(V)$ and $\mathbf a^\top = (\mathbf a^\top)^\sharp_\varepsilon = [3^2,1^{2n-5}]$.
Let $\mathbf b$ be the symmetric composition of $2n+1$ such that $X = \SOFl_{\mathbf b}(V)$.
First consider the case $|\mathbf b| = 3$, so that $\mathbf b = (b_1,b_2,b_1)$. As $X \ne \SOGr_1(V)$, we have $b_1 \ne 1$.
If $b_1 = 2$ then $X = \SOGr_2(V)$.
If $b_2 = 1$ then $X = \SOGr_{\max}(V)$.
If $b_1 \ge 3$ and $b_2 \ge 2$ then $b_2 \ge 3$ and $(\mathbf b^\top)^\sharp_\varepsilon$ has the form $(3,3,3,\ldots)$, which implies $(\mathbf b^\top)^\sharp_\varepsilon \succ (\mathbf a^\top)^\sharp_\varepsilon$.
In the case $|\mathbf b| = l \ge 5$ we have $(\mathbf b^\top)^\sharp_\varepsilon \succcurlyeq [5, 1^{2n-4}] \succ (\mathbf a^\top)^\sharp_\varepsilon$.
\end{proof}

\subsection{An analysis of the partial order on \texorpdfstring{$\mathscr F(G) / \sim$}{F(G)/\textasciitilde} for \texorpdfstring{$G \simeq \Spin_{2n}$}{G = Spin\_2n}}
\label{subsec_po_so2n}

In this subsection, we assume that $G = \Spin(V)$ with $\dim V = 2n$ and $n \ge 2$.
Throughout this subsection, $\varepsilon = 1$.

\begin{proposition} \label{prop_so_even_minimal}
Suppose that $X \in \mathscr F(G)$.
Then $X = \SOGr_1(V)$, or $X = \SOGr^+_{\max}(V)$, or $X = \SOGr^{-}_{\max}(V)$, or $\bl X \br \succ \bl \SOGr_1(V) \br$.
\end{proposition}

\begin{proof}
If $n = 2$ then $\mathscr F(G) = \lbrace \SOGr_1(V), \SOGr^+_{\max}(V), \SOGr^{-}_{\max}(V) \rbrace$ and the assertion holds trivially, so in what follows we assume $n \ge 3$. Put $\mathbf a = (1, 2n-2, 1)$. Then $\SOFl_{\mathbf a}(V) = \SOGr_1(V)$ and $\mathbf a^\top = (\mathbf a^\top)^\sharp_\varepsilon = [3,1^{2n-3}]$. Let $\mathbf b$ be the symmetric composition of $2n$ corresponding to~$X$.

If $|\mathbf b | = 2$ then $\mathbf b = (n,n)$ and $X$ is one of $\SOGr^+_{\max}(V)$ or $\SOGr^{-}_{\max}(V)$.

Suppose that $|\mathbf b | = 3$, so that $\mathbf b = (b_1,b_2,b_1)$.
If $b_1 = 1$ then $X = \SOGr_1(V)$.
If $b_1 \ge 2$ then $(\mathbf b^\top)^\sharp_\varepsilon$ has the form $(3,3,\ldots)$, hence $(\mathbf b^\top)^\sharp_\varepsilon \succ (\mathbf a^\top)^\sharp_\varepsilon$.

Suppose that $|\mathbf b| = 4$.
Then $\mathbf b^\top$ is of the form $(4,4,\ldots)$ or $(4,2,\ldots)$, hence $(\mathbf b^\top)^\sharp_\varepsilon$ is of the form $(4,4,\ldots)$ or $(3,3,\ldots)$, which implies $(\mathbf b^\top)^\sharp_\varepsilon \succ (\mathbf a^\top)^\sharp_\varepsilon$.

Finally, in the case $|\mathbf b| \ge 5$ we have $(\mathbf b^\top)^\sharp_\varepsilon \succcurlyeq [5,1^{2n-5}] \succ (\mathbf a^\top)^\sharp_\varepsilon$.
\end{proof}

\begin{proposition} \label{prop_so_even_SOGr_2}
Suppose that $n \ge 3$, $X \in \mathscr F(G)$, $X \ne \SOGr_1(V)$, $X \ne \SOGr^+_{\max}(V)$, and $X \ne \SOGr^{-}_{\max}(V) \rbrace$.
Then $\bl X \br \succcurlyeq \bl \SOGr_2(V) \br$.
Moreover, $\SOGr_2(V)$ is the unique element in $\bl \SOGr_2(V) \br$ for $n \ge 5$,
\[
\bl \SOGr_2(V) \br = \lbrace \SOGr_2(V), \SOGr_3(V), \SOFl^+_{(1,3,3,1)}(V), \SOFl^{-}_{(1,3,3,1)}(V) \rbrace
\]
for $n =4$, and
\[
\bl \SOGr_2(V) \br = \lbrace \SOGr_2(V), \SOFl^+_{(1,2,2,1)}(V), \SOFl^{-}_{(1,2,2,1)}(V) \rbrace
\]
for $n = 3$.
\end{proposition}

\begin{proof}
Put $\mathbf a = (2,2n-4,2)$. Then $\SOFl_{\mathbf a}(V) = \SOGr_2(V)$ and $\mathbf a^\top = (\mathbf a^\top)^\sharp_\varepsilon = [3^2,1^{2n-6}]$.
Let $\mathbf b$ be the symmetric composition of $2n$ corresponding to~$X$.
As $X \ne \SOGr^+_{\max}(V)$ and $X \ne \SOGr^-_{\max}(V)$, we have $|\mathbf b| \ge 3$.

Suppose that $|\mathbf b | = 3$, so that $\mathbf b = (b_1, b_2, b_1)$.
As $X \ne \SOGr_1(V)$, we have $b_1 \ne 1$.
If $b_1 = 2$ then $X = \SOGr_2(V)$.
If $b_1 = 3$ and $b_2 = 2$ then $(\mathbf b^\top)^\sharp_\varepsilon = (\mathbf a^\top)^\sharp_\varepsilon$, hence $X = \SOGr_3(V) \sim \SOGr_2(V)$.
If either $b_1 = 3$, $b_2 \ge 3$ or $b_1 \ge 4$ then $(\mathbf b^\top)^\sharp_\varepsilon$ has the form $(3,3,3,\ldots)$ or $(3,3,2,2,\ldots)$, which implies $(\mathbf b^\top)^\sharp_\varepsilon \succ (\mathbf a^\top)^\sharp_\varepsilon$.

Suppose that $|\mathbf b| = 4$, so that $\mathbf b = (b_1,b_2,b_2,b_1)$.
If $b_1 = 1$ and $b_2 = 2$ then $(\mathbf b^\top)^\sharp_\varepsilon = (\mathbf a^\top)^\sharp_\varepsilon$ and we get $X \in \lbrace \SOFl^\pm_{(1,2,2,1)}(V) \rbrace \subset \bl \SOGr_2(V) \br$.
If $b_1 = 1$ and $b_2 = 3$ then $(\mathbf b^\top)^\sharp_\varepsilon = (\mathbf a^\top)^\sharp_\varepsilon$ and we get $X \in \lbrace \SOFl^\pm_{(1,3,3,1)}(V) \rbrace \subset \bl \SOGr_2(V) \br$.
If $b_2 = 1$ and $b_1 = 2$ then $X = \SOGr_2(V)$.
If $b_2 = 1$ and $b_1 = 3$ then $X = \SOGr_3(V)$.
If either $b_1 = 1$, $b_2 \ge 4$ or $b_2 = 1$, $b_1 \ge 4$ then $(\mathbf b^\top)^\sharp_\varepsilon$ has the form $(3,3,2,2,\ldots)$, hence $(\mathbf b^\top)^\sharp_\varepsilon \succ (\mathbf a^\top)^\sharp_\varepsilon$.
If $b_1,b_2 \ge 2$ then $(\mathbf b^\top)^\sharp_\varepsilon$ has the form $(4,4,\ldots)$, hence $(\mathbf b^\top)^\sharp_\varepsilon \succ (\mathbf a^\top)^\sharp_\varepsilon$.

In the case $|\mathbf b| \ge 5$ we have $(\mathbf b^\top)^\sharp_\varepsilon \succcurlyeq [5,1^{2n-5}] \succ (\mathbf a^\top)^\sharp_\varepsilon$.
\end{proof}

\subsection{Checking \texorpdfstring{$H$}{H}-sphericity for a given flag variety}

Let $H \subset G$ be a connected reductive subgroup of~$G$ and let $I \subset S$ be an arbitrary subset.
In this subsection, we present a criterion which enables one to check $H$-sphericity of~$X_I$ effectively.
This criterion is based on results of Panyushev~\cite{Pan}; see details in~\cite[\S\,4.3]{AvP2}.

\begin{proposition}[{\cite[Proposition~4.6]{AvP2}}] \label{prop_sphericity_criterion}
Let $K$ be a connected reductive group.
Suppose that $X$ is a smooth complete irreducible $K$-variety, $Y \subset X$ is a closed $K$-orbit, $y \in Y$, and $M$ is a Levi subgroup of~$K_y$.
Then the following conditions are equivalent:
\begin{enumerate}[label=\textup{(\arabic*)},ref=\textup{\arabic*}]
\item \label{X_spherical_1}
$X$ is a $K$-spherical variety.

\item \label{X_spherical_2}
$T_y X / T_y Y$ is a spherical $M$-module.
\end{enumerate}
Moreover, under the above two conditions one has $\rk_K X = \rk_M (T_y X / T_y Y)$.
\end{proposition}

It is well known that, under an appropriate choice of $H$ within its conjugacy class in~$G$, one can achieve the inclusion $B_H^- \subset B_G^-$.
In this situation, Proposition~\ref{prop_sphericity_criterion} combined with Theorem~\ref{thm_Gamma_is_free}(\ref{thm_Gamma_is_free_b}) yield the following result, which is widely used throughout this paper in explicit calculations.

\begin{proposition}[{see~\cite[Corollary~4.8]{AvP2}}]
\label{prop_sph_criterion_eff}
Suppose that $B_H^- \subset B_G^-$ and $M$ is a Levi subgroup of $P_I^- \cap H$.
Then the following conditions are equivalent:
\begin{enumerate}[label=\textup{(\arabic*)},ref=\textup{\arabic*}]
\item
$X_I$ is an $H$-spherical variety.

\item
$\mathfrak g/ (\mathfrak p_I^- + \mathfrak h)$ is a spherical $M$-module.
\end{enumerate}
Moreover, under the above conditions one has
\begin{equation} \label{eqn_rank_of_RBM}
\rk \Gamma_I(G,H) = |I| + \rk_M(\mathfrak g / (\mathfrak p_I^- + \mathfrak h)).
\end{equation}
\end{proposition}

\section{Classification in the symplectic case}
\label{sect_sympl_case}

\subsection{Preliminary remarks}

Throughout this section, we assume that $V$ is a vector space of even dimension $d \ge 4$ equipped with a nondegenerate skew-symmetric bilinear form~$\omega$ and $H$ is a connected reductive subgroup of $G = \Sp(V)$.

When referring to the classification of spherical modules, we always use the list in \cite[\S\,5]{Kn} (see also \cite[Theorems~5.1--5.2]{AvP1}) and the general criterion provided by \cite[Theorem~5.3]{AvP1}.
One useful consequence of this classification is the following lemma, for which we provide a proof in order to demonstrate application of the above-cited sources in an example.

\begin{lemma} \label{lemma_VixVj_spherical}
Consider the group $K = \GL_{n_1} \times \ldots \times \GL_{n_k}$ with $n_1 \ge \ldots \ge n_k \ge 1$ and the $K$-module
\[
W = \bigoplus \limits_{1 \le i < j \le k} \mathop{\vphantom|\FF^{n_i}}\limits_i{\!} \otimes \mathop{\vphantom|\FF^{n_j}}\limits_j{\!}.
\]
Then $W$ is $K$-spherical if and only if one of the following two conditions holds:
\begin{enumerate}[label=\textup{(\arabic*)},ref=\textup{\arabic*}]
\item
$k = 2$;

\item
$k = 3$ and $n_2 = n_3 = 1$.
\end{enumerate}
\end{lemma}

\begin{proof}
If $k = 2$ then $W = \FF^{n_1} \otimes \FF^{n_2}$, the latter module appearing in the list of \cite[\S\,5]{Kn} hence being $K$-spherical.
In what follows we assume $k \ge 3$.
If $n_1 = \ldots = n_k = 1$ then the pair $(K,W)$ is equivalent to
\[
(\underbrace{\FF^\times \times \ldots \times \FF^\times}_k, \bigoplus \limits_{1 \le i < j \le k} \FF^1_{\chi_i + \chi_j}).
\]
By \cite[Theorem~5.3]{AvP1}, the latter module is spherical if and only if all characters in the multiset $\lbrace \chi_i + \chi_j \mid 1 \le i < j \le k \rbrace$ are linearly independent, which holds only for $k = 3$.
Now we assume $n_1 \ge 2$.
If $k \ge 4$ then $W$ contains the $K$-submodule $\bigoplus \limits_{2 \le j \le k} \mathop{\vphantom|\FF^{n_1}}\limits_1{\!} \otimes \mathop{\vphantom|\FF^{n_j}}\limits_j{\!}$, which is saturated and indecomposable but not present in the list of \cite[\S\,5]{Kn}, hence not spherical.
It remains to consider the case $k = 3$.
If $n_2 \ge 2$ then $W$ is saturated, indecomposable but not present in the list of~\cite[\S,5]{Kn}, hence not spherical.
Finally, for $n_2 = n_3 = 1$ the pair $(K,V)$ is equivalent to
\[
(\SL_{n_1} \times \FF^\times \times \FF^\times \times \FF^\times, [\FF^{n_1}]_{\chi_1+\chi_2} \oplus [\FF^{n_1}]_{\chi_1+ \chi_3} \oplus \FF^1_{\chi_2 + \chi_3}).
\]
By \cite[Theorem~5.3]{AvP1}, the sphericity conditions for the latter module are as follows:
\begin{itemize}
\item
$\chi_1 + \chi_2, \ \chi_1 + \chi_3, \ \chi_2 + \chi_3$ are linearly independent if $n_1 = 2$;
\item
$\chi_2 - \chi_3, \ \chi_2 + \chi_3$ are linearly independent if $n_1 \ge 3$.
\end{itemize}
As both conditions hold, the proof is completed.
\end{proof}

For explicit calculations involving the subgroup~$H$, we use the following conventions.
First, we identify $V$ with $\FF^d$.
Second, suppose that $V$ is written as $V = V_1 \oplus \ldots \oplus V_m$ where all direct summands are pairwise orthogonal with respect to~$\omega$ and each $V_i$ is an $H$-module that is either simple or weakly reducible.
If $\dim V_1 = 2k$ then $V_1$ is embedded in $V$ as the linear span of the vectors $e_1,\ldots, e_k, e_{d-k+1},\ldots, e_d$ and $V_2 \oplus \ldots \oplus V_m$ is embedded as the linear span of the vectors $e_{k+1}, \ldots, e_{d-k}$.
Moreover, if $V_1$ is weakly reducible of the form $\Omega(W_1)$ then we assume in addition that the $H$-submodule $W_1 \subset V$ is the linear span of the vectors $e_1,\ldots, e_k$ and the $H$-submodule $W_1^* \subset V$ is the linear span of the vectors $e_{d-k+1}, \ldots, e_d$.
The embeddings of $V_2$, \ldots, $V_m$ in $V$ are determined by iterating the above procedure.

For checking sphericity of a given $H$-variety we always use Proposition~\ref{prop_sph_criterion_eff} with $B_H^- = B_G^- \cap H$.
The above conventions on~$H$ always guarantee that $B_G^- \cap H$ is a Borel subgroup of~$H$.
It is well known that in the case $H = G$ every flag variety of~$G$ is $H$-spherical, this fact will be used without extra explanation.

In \S\S\,\ref{subsec_P(V)}--\ref{subsec_SpGr_2}, our classification of spherical actions on flag varieties of $G$ involves all possible subgroups~$H$ including symmetric subgroups and Levi subgroups.
On the contrary, in~\S\,\ref{subsec_sympl_case_end} we exclude symmetric subgroups and Levi subgroups referring to~\cite[\S\,5]{AvP2}.

\subsection{Spherical actions on \texorpdfstring{$\PP(V)$}{P(V)}}
\label{subsec_P(V)}

The starting point of our classification in the symplectic case is the following result.

\begin{theorem} \label{thm_sp_SA_P(V)}
The variety $\PP(V)$ is $H$-spherical if and only if the pair $(H,V)$ is BF-equivalent to a pair in Table~\textup{\ref{table_sph_PV}}.
\end{theorem}

\begin{proof}
The variety $\PP(V)$ is $H$-spherical if and only if $V$ is a spherical $(H \times \FF^\times)$-module, where $\FF^\times$ acts on $V$ by scalar transformations.

We may assume that the pair $(H,V)$ is BF-equivalent to a pair
\[
(H, V_1 \oplus \ldots \oplus V_p \oplus \Omega(W_1) \oplus \ldots \oplus \Omega(W_q))
\]
where $p,q \ge 0$ and each $V_i$ and $W_j$ are simple $H$-modules.

If $V$ is a spherical $(H \times \FF^\times)$-module then each $V_i$ and each $W_i$ is also a spherical $(H \times \FF^\times)$-module.
Then by Proposition~\ref{prop_spherical+sympl} the image of $H$ in $\GL(V_i)$ coincides with $\Sp(V_i)$ for all $i = 1,\ldots,p$.
Similarly, for each $i = 1,\ldots,q$ the pair $(H,W_i)$ is equivalent to one of $(\GL_n, \FF^n)$ ($n \ge 1$), $(\SL_n, \FF^n)$ ($n \ge 3$), or $(\Sp_{2n} \times \FF^\times, [\FF^{2n}]_{\chi})$ ($n \ge 2$).

Next we show that $q \le 1$. Indeed, otherwise $\Omega(W_1) \oplus \Omega(W_2)$ would be a spherical $(H \times \FF^\times)$-module, hence $W_1 \oplus W_1^* \oplus W_2 \oplus W_2^*$ would be a spherical $(\GL(W_1) \times \GL(W_2) \times \FF^\times)$-module (where $\FF^\times$ acts by scalar transformations), which is not the case.

Now suppose there is a simple factor of $H$ that acts nontrivially on some $V_i$ and some other summand.
According to the classification, this other summand cannot be $W_1$; neither can it be $V_j$ with $j \ne i$ since otherwise $V_i \simeq V_j$ as $(H \times \FF^\times)$-modules, in which case $V_i \oplus V_j$ cannot be a spherical $(H \times \FF^\times)$-module.

It follows from the above arguments that $V$ can be a spherical $(H \times \FF^\times)$-module only if the pair $(H,V)$ is BF-equivalent to a pair in Table~\ref{table_sph_PV}.
On the other hand, for each of these pairs the $(H \times \FF^\times)$-module $V$ is spherical.
\end{proof}

\subsection{Spherical actions on \texorpdfstring{$\SpGr_{\max}(V)$}{SpGr\_max(V)}}
\label{subsec_SpGr_max}

Recall from \S\,\ref{subsec_FV_via_compositions} that $\SpGr_{\max}(V) \simeq X_I$ with $I = \lbrace d/2 \rbrace$.
If $\SpGr_{\max}(V)$ is an $H$-spherical variety then $\PP(V)$ should be also $H$-spherical in view of Proposition~\ref{prop_P(V)_is_minimal} and Theorem~\ref{thm_sph_descent}.
Consequently, by Theorem~\ref{thm_sp_SA_P(V)} it suffices to consider only the cases listed in Table~\ref{table_sph_PV}.

\begin{proposition} \label{prop_sp...sp_Grmax}
Suppose that the pair $(H,V)$ is BF-equivalent to that in Case~\textup{\ref{sp_case_1}} of Table~\textup{\ref{table_sph_PV}}
with $n_1 \ge n_2 \ge \ldots \ge n_k \ge 1$.
Then $\SpGr_{\max}(V)$ is $H$-spherical if and only if one of the following conditions holds:
\begin{enumerate}[label=\textup{(\arabic*)},ref=\textup{\arabic*}]
\item \label{prop_sp...sp_Grmax_1}
$k \le 2$;

\item \label{prop_sp...sp_Grmax_2}
$k = 3$ and $n_2 = n_3 = 1$.
\end{enumerate}
\end{proposition}

\begin{proof}
If $k = 1$ then $\SpGr_{\max}(V)$ is $H$-spherical.
If $k \ge 2$ then it is easy to see that the pair $(M, \mathfrak g / (\mathfrak p_I^- + \mathfrak h))$ is equivalent to
\[
(\GL_{n_1} \times \ldots \times \GL_{n_k}, \bigoplus \limits_{1 \le i < j \le k} \mathop{\vphantom|\FF^{n_i}}\limits_i{\!} \otimes \mathop{\vphantom|\FF^{n_j}}\limits_j{\!}).
\]
By Lemma~\ref{lemma_VixVj_spherical}, the latter module is spherical if and only if either $k = 2$ or condition~(\ref{prop_sp...sp_Grmax_2}) holds.
\end{proof}

\begin{proposition} \label{prop_sp...spgl_Grmax}
Suppose that the pair $(H,V)$ is BF-equivalent to that in Case~\textup{\ref{sp_case_2}} of Table~\textup{\ref{table_sph_PV}}.
Then $\SpGr_{\max}(V)$ is $H$-spherical if and only if one of the following two conditions holds:
\begin{enumerate}[label=\textup{(\arabic*)},ref=\textup{\arabic*}]
\item
$k = 0$;

\item
$k = m = 1$.
\end{enumerate}
\end{proposition}

\begin{proof}
Without loss of generality we may assume $n_1 \ge n_2 \ge \ldots \ge n_k \ge 1$.
If $H$ acts spherically on $X = \SpGr_{\max}(V)$ then the group $\Sp_{2n_1} \times \ldots \times \Sp_{2n_k} \times \Sp_{2m}$ also acts spherically on~$X$.
Then Proposition~\ref{prop_sp...sp_Grmax} implies that $k \le 2$ and the following cases may occur.

\textit{Case}~1: $k = 0$.
The pair $(M, \mathfrak g/(\mathfrak p_I^- + \mathfrak h))$ is equivalent to $(\GL_m, \mathrm{S}^2 \FF^m)$, the latter module being spherical.

\textit{Case}~2: $k = 1$.
The pair $(M, \mathfrak g/(\mathfrak p_I^- + \mathfrak h))$ is equivalent to
\[
(\GL_{n_1} \times \GL_{m}, \mathop{\vphantom|\FF^{n_1}}\limits_1{\!} \otimes \mathop{\vphantom|\FF^m}\limits_2{\!} \oplus \mathop{\vphantom| \mathrm{S}^2\FF^m} \limits_2{\!}),
\]
the latter module being spherical if and only if $m = 1$.

\textit{Case}~3: $k = 2$, $n_1 = n_2 = 1$.
If $X$ is $H$-spherical then $X$ is $(\Sp_4 \times \GL_m)$-spherical, which implies $m = 1$ by the previous case.
Then $\dim B_H = 5 < 6 = \dim X$, hence $X$ is not $H$-spherical.

\textit{Case}~4: $k = 2$, $n_2 = m = 1$. In this case, the pair $(M, \mathfrak g / (\mathfrak p_I^- + \mathfrak h))$ is equivalent to
\[
(\GL_{n_1} \times \FF^\times \times \FF^\times, \mathop{[\FF^{n_1}]}_1{\!}_{\chi_1} \oplus \mathop{[\FF^{n_1}]}_1{\!}_{\chi_2} \oplus \FF^1_{\chi_1+\chi_2} \oplus \FF^1_{2\chi_1}),
\]
the latter module being not spherical.
\end{proof}

\begin{proposition}
Suppose that the pair $(H,V)$ is BF-equivalent to that in Case~\textup{\ref{sp_case_3}} of Table~\textup{\ref{table_sph_PV}}.
Then $\SpGr_{\max}(V)$ is not $H$-spherical.
\end{proposition}

\begin{proof}
If $H$ acts spherically on $X = \SpGr_{\max}(V)$ then the group $\Sp_{2n_1} \times \ldots \times \Sp_{2n_k} \times \GL_{m}$ also acts spherically on~$X$. Then Proposition~\ref{prop_sp...spgl_Grmax} leaves us with the following two cases.

\textit{Case}~1: $k = 0$.
The pair $(M, \mathfrak g/(\mathfrak p_I^- + \mathfrak h))$ is equivalent to $(\SL_m, \mathrm{S}^2 \FF^m)$, the latter module being not spherical.

\textit{Case}~2: $k = m = 1$.
The pair $(M, \mathfrak g/(\mathfrak p_I^- + \mathfrak h))$ is equivalent to $(\GL_{n_1}, \mathop{\FF^{n_1}}\limits_1{\!} \oplus \FF^1)$, the latter module being not spherical.
\end{proof}

\begin{proposition} \label{prop_sp...spFx_Grmax}
Suppose that the pair $(H,V)$ is BF-equivalent to that in Case~\textup{\ref{sp_case_4}} of Table~\textup{\ref{table_sph_PV}}.
Then $\SpGr_{\max}(V)$ is not $H$-spherical.
\end{proposition}

\begin{proof}
If $H$ acts spherically on $X = \SpGr_{\max}(V)$ then the group $\Sp_{2n_1} \times \ldots \times \Sp_{2n_k} \times \GL_{2m}$ also acts spherically on~$X$.
By Proposition~\ref{prop_sp...spgl_Grmax}, the latter is possible only if $k = 0$.
In this case, the pair $(M, \mathfrak g/(\mathfrak p_I^- + \mathfrak h))$ is equivalent to $(\Sp_{2m}\times \FF^\times, [\mathrm{S}^2\FF^{2m}]_{2\chi})$, the latter module being not spherical.
\end{proof}

Summarizing the results obtained in Propositions~\ref{prop_sp...sp_Grmax}--\ref{prop_sp...spFx_Grmax} and comparing them with the statements in~\S\,\ref{subsec_Levi&symmetric}, we arrive at

\begin{corollary} \label{crl_SpGr_max}
Suppose that $H$ is neither a symmetric subgroup nor a Levi subgroup of~$\Sp(V)$.
Then $\SpGr_{\max}(V)$ is $H$-spherical if and only if the pair $(H,V)$ is BF-equivalent to $(\Sp_{2m} \times \SL_2 \times \SL_2, \mathop{\FF^{2m}} \limits_1{\!} \oplus \mathop{\FF^{2}}\limits_2{\!} \oplus \mathop{\FF^{2}} \limits_3{\!})$ with $m \ge 1$.
\end{corollary}

\subsection{Spherical actions on \texorpdfstring{$\SpGr_2(V)$}{SpGr\_2(V)}}
\label{subsec_SpGr_2}

Recall from \S\,\ref{subsec_FV_via_compositions} that $\SpGr_2(V) \simeq X_I$ with $I = \lbrace 2 \rbrace$.
If $\SpGr_2(V)$ is an $H$-spherical variety then $\PP(V)$ should be also $H$-spherical in view of Proposition~\ref{prop_P(V)_is_minimal} and Theorem~\ref{thm_sph_descent}.
Consequently, by Theorem~\ref{thm_sp_SA_P(V)} it suffices to consider only the cases listed in Table~\ref{table_sph_PV}.

\begin{proposition} \label{prop_sp...sp_Gr2}
Suppose that the pair $(H,V)$ is BF-equivalent to that in Case~\textup{\ref{sp_case_1}} of Table~\textup{\ref{table_sph_PV}}.
Then $\SpGr_2(V)$ is $H$-spherical if and only if $k \le 2$.
\end{proposition}

\begin{proof}
If $k = 1$ then $\SpGr_2(V)$ is $H$-spherical.
In what follows we assume $k \ge 2$ and $n_1 \ge n_2 \ge \ldots \ge n_k \ge 1$ without loss of generality.

If $n_1 = \ldots = n_k = 1$ then the pair $(M, \mathfrak g/(\mathfrak p_I^- + \mathfrak h))$ is equivalent to
\[
(\underbrace{\SL_2 \times \ldots \times \SL_2}_{k-2} \times \FF^\times \times \FF^\times, \FF^1_{\chi_1 + \chi_2} \oplus \bigoplus \limits_{1 \le i \le k-2} (\mathop{[\FF^{2}]}\limits_i{\!}_{\chi_1} \oplus \mathop{[\FF^{2}]}\limits_i{\!}_{\chi_2})),
\]
the latter module being spherical if and only if $k = 2$.

If $n_1 \ge 2$ then the pair $(M, \mathfrak g/(\mathfrak p_I^- + \mathfrak h))$ is equivalent to
\[
(\GL_2 \times \Sp_{2n_2} \times \ldots \times \Sp_{2n_k}, \bigoplus \limits_{2 \le i \le k}
(\mathop{\vphantom|\FF^2}\limits_1{\!} \otimes \mathop{\vphantom|\FF^{2n_i}}\limits_i{\!})),
\]
the latter module being spherical if and only if $k = 2$.
\end{proof}

\begin{proposition} \label{prop_sp...spgl_Gr2}
Suppose that the pair $(H,V)$ is BF-equivalent to that in Case~\textup{\ref{sp_case_2}} of Table~\textup{\ref{table_sph_PV}}.
Then $\SpGr_2(V)$ is $H$-spherical if and only if one of the following two conditions holds:
\begin{enumerate}[label=\textup{(\arabic*)},ref=\textup{\arabic*}]
\item
$k = 0$, $m = 2$.

\item
$k=m=1$.
\end{enumerate}
\end{proposition}

\begin{proof}
If $H$ acts spherically on $X = \SpGr_2(V)$ then the group $\Sp_{2n_1} \times \ldots \times \Sp_{2n_k} \times \Sp_{2m}$ also acts spherically on~$X$.
Then Proposition~\ref{prop_sp...sp_Gr2} leaves us with the following two cases.

\textit{Case}~1: $k = 0$.
The pair $(M, \mathfrak g/(\mathfrak p_I^- + \mathfrak h))$ is equivalent to
\[
(\GL_2 \times \GL_{m-2}, \mathop{\vphantom|\FF^2}\limits_1{\!} \otimes \mathop{\vphantom|\FF^{m-2}}\limits_2{\!} \oplus \mathop{\vphantom|\mathrm{S}^2\FF^2}\limits_1{\!}),
\]
the latter module being spherical if and only if $m = 2$.

\textit{Case}~2: $k = 1$.
If $n_1 = 1$ then the pair $(M, \mathfrak g/(\mathfrak p_I^- + \mathfrak h))$ is equivalent to
\[
(\GL_{m-1} \times \FF^\times \times \FF^\times,
\mathop{[\FF^{m-1}]}\limits_1{\!}_{\chi_1} \oplus \mathop{[\FF^{m-1}]}\limits_1{\!}_{\chi_2} \oplus \mathop{[(\FF^{m-1})^*]}\limits_1{\!}_{\chi_1} \oplus \FF^1_{\chi_1 + \chi_2} \oplus \FF^1_{2\chi_2}),
\]
the latter module being spherical if and only if $m = 1$.
If $n_1 \ge 2$ then the pair $(M, \mathfrak g/(\mathfrak p_I^- +\nobreak \mathfrak h))$ is equivalent to
\[
(\GL_2 \times \GL_m, \mathop{\vphantom|\FF^2}\limits_1{\!} \otimes \mathop{\vphantom|\FF^m}\limits_2{\!} \oplus \mathop{\vphantom|\FF^2}\limits_1{\!} \otimes \mathop{(\FF^m)^*}\limits_2{\!}),
\]
the latter module being spherical if and only if $m = 1$.
\begin{comment}
If $m = 1$ then the pair $(M, \mathfrak g/(\mathfrak p_I^- + \mathfrak h))$ is equivalent to
\[
(\Sp_{2n_1-2} \times \FF^\times \times \FF^\times, \FF^{2n_1-2} \otimes \FF_{\chi_1} \oplus \FF_{\chi_1 + \chi_2} \oplus \FF_{2\chi_1}),
\]
the latter module being spherical.
If $m \ge 2$ then the pair $(M, \mathfrak g/(\mathfrak p_I^- + \mathfrak h))$ is equivalent to $(\GL_2 \times \GL_{m-2} \times \Sp_{2n_1}, \mathrm{S}^2 \FF^2 \oplus \FF^2 \otimes (\FF^{m-2} \oplus \FF^{2n_1}))$, the latter module being not spherical.
\end{comment}
\end{proof}

\begin{proposition}
Suppose that the pair $(H,V)$ is BF-equivalent to that in Case~\textup{\ref{sp_case_3}} of Table~\textup{\ref{table_sph_PV}}.
Then $\SpGr_2(V)$ is not $H$-spherical.
\end{proposition}

\begin{proof}
If $H$ acts spherically on $X = \SpGr_2(V)$ then the group $\Sp_{2n_1} \times \ldots \times \Sp_{2n_k} \times \GL_{m}$ also acts spherically on~$X$. As $m \ge 3$, the latter is impossible by Proposition~\ref{prop_sp...spgl_Gr2}.
\end{proof}

\begin{proposition} \label{prop_SpGr2_last}
Suppose that the pair $(H,V)$ is BF-equivalent to that in Case~\textup{\ref{sp_case_4}} of Table~\textup{\ref{table_sph_PV}}.
Then $\SpGr_2(V)$ is not $H$-spherical.
\end{proposition}

\begin{proof}
If $H$ acts spherically on $X = \SpGr_2(V)$ then the group $\Sp_{2n_1} \times \ldots \times \Sp_{2n_k} \times \GL_{2m}$ also acts spherically on~$X$. As $m \ge 2$, the latter is impossible by Proposition~\ref{prop_sp...spgl_Gr2}.
\end{proof}

Summarizing the results obtained in Propositions~\ref{prop_sp...sp_Gr2}--\ref{prop_SpGr2_last} and comparing them with the statements in~\S\,\ref{subsec_Levi&symmetric}, we arrive at

\begin{corollary} \label{crl_SpGr_2}
If $\SpGr_2(V)$ is $H$-spherical then $H$ is either a symmetric subgroup or a Levi subgroup of~$\Sp(V)$.
\end{corollary}

\subsection{Completion of the classification}
\label{subsec_sympl_case_end}

We conclude our classification in the symplectic case by

\begin{proposition}
Let $X$ be a nontrivial flag variety of $\Sp(V)$ different from $\PP(V)$, $\SpGr_2(V)$, and $\SpGr_{\max}(V)$.
If $X$ is $H$-spherical then $H$ is either a symmetric subgroup or a Levi subgroup of $\Sp(V)$.
\end{proposition}

\begin{proof}
By Proposition~\ref{prop_SpGr2_partial_order} and Theorem~\ref{thm_sph_descent}, $X$ being $H$-spherical implies $\SpGr_2(V)$ being $H$-spherical.
Then the assertion follows from Corollary~\ref{crl_SpGr_2}.
\end{proof}

\section{Classification in the orthogonal case}
\label{sect_orth_case}

\subsection{Preliminary remarks}
\label{subsec_so_prelim_remarks}

Throughout this section, we assume that $V$ is a vector space of dimension $d \ge 3$ equipped with a nondegenerate symmetric bilinear form~$\omega$.
We assume that $H$ is a connected reductive subgroup of $G = \Spin(V)$.

We fix a decomposition $V = V_1 \oplus \ldots \oplus V_r$ where all direct summands are pairwise orthogonal with respect to~$\omega$ and each $V_i$ is an $H$-module that is either simple or weakly reducible.
For each $i = 1,\ldots, r$ we put $d_i = \dim V_i$.

The image of $H$ in $\SO(V)$ is denoted by~$H_0$.
Note that either $H \simeq H_0$ or $H$ is a two-fold covering of~$H_0$.
For every $i = 1, \ldots, r$, the projection of $H_0$ to $\SO(V_i)$ is denoted by~$H_i$.

As in \S\,\ref{sect_sympl_case}, when referring to the classification of spherical modules, we always use the list in \cite[\S\,5]{Kn} (see also \cite[Theorems~5.1--5.2]{AvP1}) and the general criterion provided by~\cite[Theorem~5.3]{AvP1}.

Suppose that $d$ is even and there is a decomposition $V = \widetilde V_1 \oplus \widetilde V_2$ where the direct summands are pairwise orthogonal and $\dim \widetilde V_2 = 1$.
In this situation, it is well known that the group $\SO(\widetilde V_1)$ acts transitively on both varieties $\SOGr_{\max}^\pm(V)$ and each of these is isomorphic to $\SOGr_{\max}(\widetilde V_1)$ as an $\SO(\widetilde V_1)$-variety.
This leads to the following result, which will be used several times in this section.

\begin{proposition} \label{prop_so_2n_to_2n-1}
Under the above notation, suppose that $H_0$ is a subgroup of $\SO(\widetilde V_1)$.
Then the following conditions are equivalent:
\begin{enumerate}[label=\textup{(\arabic*)},ref=\textup{\arabic*}]
\item
$\SOGr_{\max}^+(V)$ is $H$-spherical;

\item
$\SOGr_{\max}^-(V)$ is $H$-spherical;

\item
$\SOGr_{\max}(\widetilde V_1)$ is $H$-spherical.
\end{enumerate}
\end{proposition}

For explicit calculations involving the subgroup~$H$, we use the following conventions.
First, we identify $V$ with~$\FF^d$.
Second, suppose that $V$ is written as $V = \widetilde V_1 \oplus \ldots \oplus \widetilde V_m$ where all direct summands are pairwise orthogonal with respect to~$\omega$ and each $\widetilde V_i$ is an $H$-module that is either simple or weakly reducible.
Unless otherwise specified, we use the following embeddings.
\begin{itemize}
\item
If $\dim \widetilde V_1 = 2k$ then $\widetilde V_1$ is embedded in $V$ as the linear span of the vectors $e_1,\ldots, e_k, e_{d-k+1},\ldots, e_d$ and $\widetilde V_2 \oplus \ldots \oplus \widetilde V_m$ is embedded as the linear span of the vectors $e_{k+1}, \ldots, e_{d-k}$.
Moreover, if $\widetilde V_1$ is weakly reducible of the form $\Omega(W_1)$ then we assume in addition that the $H$-submodule $W_1 \subset V$ is the linear span of the vectors $e_1,\ldots, e_k$ and the $H$-submodule $W_1^* \subset V$ is the linear span of the vectors $e_{d-k+1}, \ldots, e_d$.

\item
If $\dim \widetilde V_1 = 2k+1$ and $d=2l+1$ then $\widetilde V_1$ is embedded in $V$ as the linear span of the vectors $e_{l-k+1},\ldots, e_{l+k+1}$ and $\widetilde V_2 \oplus \ldots \oplus \widetilde V_m$ is embedded as the linear span of the vectors $e_1, \ldots, e_{l-k}, e_{l+k+2},\ldots, e_d$.

\item
If $\dim \widetilde V_1 = 2k+1$ and $d = 2l$ then $\widetilde V_1$ is embedded in $V$ as the linear span of the vectors $e_1,\ldots, e_k, e_l + e_{l+1}, e_{d-k+1},\ldots, e_d$ and $\widetilde V_2 \oplus \ldots \oplus \widetilde V_m$ is embedded as the linear span of the vectors $e_{k+1},\ldots, e_{l-1}, e_l - e_{l+1}, e_{l+2},\ldots, e_{d - k}$.
\end{itemize}
The embeddings of $\widetilde V_2$, \ldots, $\widetilde V_m$ in $V$ are determined by iterating the above procedure. (In fact, the situation where $\dim \widetilde V_1$ is odd occurs only for $m = 2$.)

The group $\mathsf G_2$ is always embedded in $\SO_7$ according to the embedding of the corresponding Lie algebras described in Appendix~\ref{sect_g2&spin7}.

The group $\Spin_7$ is always embedded in $\SO_8$ according to the embedding of the corresponding Lie algebras described in Appendix~\ref{sect_g2&spin7}.
Since there are two conjugacy classes of $\Spin_7$ in $\SO_8$, to distinguish between them we write $\Spin^+_7$ and $\Spin^-_7$ where $\Spin^+_7$ refers to the above-described subgroup of~$\SO_8$.

For the module $(\Sp_{2n} \times \SL_2, \FF^{2n} \otimes \FF^2)$, $n \ge 2$, the image of the group $\Sp_{2n} \times \SL_2$ in $\SO_{4n}$ is described as follows.
Let $e_1,\ldots,e_{2n}$ (resp.~$f_1,f_2$) be the standard basis of $\FF^{2n}$ (resp.~$\FF^2$).
Then the standard basis of $\FF^{2n} \otimes \FF^2$ providing the required embedding is $e_1 \otimes f_1$, $e_2 \otimes f_1$, $\ldots$, $e_{2n} \otimes f_1$, $e_1 \otimes f_2$, $e_2 \otimes f_2$, $\ldots$, $e_{2n} \otimes f_2$.
For $n = 2$, there are two conjugacy classes in $\SO_8$ of (the image of) $\Sp_4 \times \SL_2$, and the above-described realization will be referred to as the default one.

For checking sphericity of a given $H$-variety we always use Proposition~\ref{prop_sph_criterion_eff} with $B_H^- = B_G^- \cap H$.
The above conventions on embeddings of~$H$, $\mathsf G_2$, $\Spin_7$, and $\Sp_{2n} \times \SL_2$ always guarantee that $B_G^- \cap H$ is a Borel subgroup of~$H$.
It is well known that any flag variety of $G$ is $H$-spherical if the pair $(H,V)$ is equivalent to either $(\SO_m, \FF^m)$ or $(\Spin_7, \FF^8)$ (in the latter case, the subgroup $H \simeq \Spin_7$ is symmetric in~$G \simeq \Spin_8$, and a suitable outer automorphism of~$G$ takes $H$ to $\widetilde H$ such that the pair $(\widetilde H, V)$ is equivalent to $(\SO_7, \FF^7 \oplus \FF^1)$), this will be used without extra explanation.

In \S\S\,\ref{subsec_SOGr_1}--\ref{subsec_SOGr_2}, our classification of spherical actions on flag varieties of $G$ involves all possible subgroups~$H$ including symmetric subgroups and Levi subgroups.
On the contrary, in~\S\,\ref{subsec_orth_case_end} we exclude the cases where $H$ is either a symmetric subgroup or intermediate between a Levi subgroup of $G$ and its derived subgroup; for these cases we refer to~\cite[\S\,5]{AvP2}.

We finish this subsection with the following lemma, which allows us to shorten computations in many cases.

\begin{lemma} \label{lemma_outer_aut}
Suppose that $\SO(V_i) \subset H_0$ for some $i \in \lbrace 1,\ldots, d \rbrace$.
Then the following assertions hold.
\begin{enumerate}[label=\textup{(\alph*)},ref=\textup{\alph*}]
\item \label{lemma_outer_aut_a}
If $d$ is even then $\SOGr_{\max}^+(V)$ is $H$-spherical if and only if so is $\SOGr_{\max}^-(V)$.

\item \label{lemma_outer_aut_b}
If $H_j \subset H_0$ for some $j \ne i$ and the pair $(H_j,V_j)$ is BF-equivalent to one of $(\Spin_7, \FF^8)$ or $(\Sp_4 \times \SL_2$, $\FF^4 \otimes\nobreak \FF^2)$ then for every subset $I \subset S$ the condition of $X_I$ being $H$-spherical holds or does not hold simultaneously for both choices of the conjugacy class of $H_j$ in $\SO(V_j)$.
\end{enumerate}
\end{lemma}

\begin{proof}
For every element $g \in \mathrm{O}(V)$, let $\sigma \colon x \mapsto gxg^{-1}$ be the corresponding automorphism of the group $\SO(V)$.
Then, as described in Reduction~2 (see~\S\,\ref{subsec_reductions}), $X_I$ is $H$-spherical if and only if $X_{\sigma(I)}$ is $\sigma(H)$-spherical.

(\ref{lemma_outer_aut_a})
If $d = 2n$ and $g \in \mathrm{O}(V_i) \setminus \SO(V_i)$ then $\sigma(H) = H$, $\SOGr_{\max}^+(V) = X_{\lbrace n \rbrace}$, and $\SOGr_{\max}^-(V) = X_{\lbrace n-1 \rbrace} = X_{\sigma(\lbrace n \rbrace)}$.

(\ref{lemma_outer_aut_b})
If $g = g_1g_2$ with $g_1 \in \mathrm{O}(V_i) \setminus \SO(V_i)$ and $g_2 \in \mathrm{O}(V_j) \setminus \SO(V_j)$ then $g \in \SO(V)$ and hence $\sigma(H)$ is conjugate to~$H$ and $\sigma(I) = I$ for every subset $I \subset S$.
On the other hand, $\sigma(H)$ differs from $H$ by changing the conjugacy class of $H_j$ in~$\SO(V_j)$.
\end{proof}

\subsection{Spherical actions on~\texorpdfstring{$\SOGr_1(V)$}{SOGr\_1(V)}}
\label{subsec_SOGr_1}

Recall from \S\,\ref{subsec_FV_via_compositions} that $\SOGr_1(V) \simeq X_I$ with $I = \lbrace 1 \rbrace$.
We start with the following auxiliary lemma in which $V$ is not necessarily a simple $H$-module.

\begin{lemma} \label{lemma_sogr1_sph}
Suppose that $v \in V$ is a highest weight vector with respect to the $H$-module structure, $v$ is isotropic in~$V$, $Q$ is the stabilizer in $H$ of the point $\langle v \rangle \in \PP(V)$, and $M$ is a Levi subgroup of~$Q$.
Suppose also that $\SOGr_1(V)$ is $H$-spherical.
Then the $M$-module $T_{\langle v \rangle} \PP(V) / T_{\langle v \rangle} (H\langle v \rangle)$ contains a spherical submodule of codimension~$1$.
Moreover, there are $M$-module isomorphisms
\begin{gather} \label{eqn_iso's1}
T_{\langle v \rangle} \PP(V) \simeq (V/\langle v \rangle)\otimes \langle v \rangle^*,\\
\label{eqn_iso's2}
T_{\langle v \rangle} (Hv) \simeq \mathfrak h/\mathfrak q \simeq (\mathfrak q^u)^*
\end{gather}
where $\mathfrak q^u$ is the nilpotent radical of~$\mathfrak q$.
\end{lemma}

\begin{proof}
The hypotheses imply that $\langle v \rangle \in \SOGr_1(V)$ and $Q$ is a parabolic subgroup of~$H$, hence $H\langle v \rangle \simeq H/Q$ is a closed $H$-orbit in~$\SOGr_1(V)$.
Then it follows from Proposition~\ref{prop_sphericity_criterion} that $T_{\langle v \rangle} \SOGr_1(V) / T_{\langle v \rangle} (H\langle v \rangle)$ is a spherical $M$-module.
Clearly, the latter module has codimension~$1$ in $T_{\langle v \rangle} \PP(V) / T_{\langle v \rangle} (H\langle v \rangle)$.
The isomorphism in~(\ref{eqn_iso's1}) follows from the $P$-module isomorphisms $T_{\langle v \rangle} \PP(V) \simeq \mathfrak{gl}(V)/\mathfrak p \simeq (V/\langle v \rangle) \otimes \langle v \rangle^*$ where $P$ is the stabilizer of $\langle v \rangle$ in $\GL(V)$.
The isomorphism in~(\ref{eqn_iso's2}) is obvious.
\end{proof}

\begin{remark} \label{rem_sogr1_sph}
If $V$ is a nontrivial simple $H$-module then every highest weight vector in $V$ is automatically isotropic.
\end{remark}

\begin{proposition} \label{prop_so_Q1_V_irr}
Suppose that $V$ is a simple $H$-module.
Then $\SOGr_1(V)$ is $H$-spherical if and only if the pair $(H,V)$ is BF-equivalent to one of $(\SO_n, \FF^n)$ $(n \ge 3)$, $(\Sp_{2n} \times \SL_2$, $\FF^{2n} \otimes \FF^2)$ $(n \ge 2)$, $(\mathsf G_2, \FF^7)$, $(\Spin_7, \FF^8)$, $(\Spin_9, \FF^{16})$.
\end{proposition}

\begin{proof}
If $H$ acts spherically on~$\SOGr_1(V)$ then $H$ has an open orbit in $\SOGr_1(V)$. According to~\cite[Theorem~2.1]{Kim}, all pairs $(H,V)$ (up to BF-equivalence) for which $H$ has an open orbit in $\SOGr_1(V)$ are listed in Table~\ref{table_open_Q1_irr}.

\begin{table}[h]
\caption{} \label{table_open_Q1_irr}

\begin{tabular}{|c|c|c|c|c|c|}
\hline
No. & $H$ & $V$ & $\dim \SOGr_1(V)$ & $\dim B_H$ & Note \\

\hline

1 & $\SO_n$ & $\FF^n$ & $n-2$ & $[\frac{n}2][\frac{n+1}2]$ & $n \ge 3$ \\

\hline

2 & $\Sp_{2n} \times \SL_2$ & $\FF^{2n} \otimes \FF^2$ & $4n-2$ & $n^2 + n + 2$ & $n \ge 2$ \\

\hline

3 & $\mathsf G_2$ & $\FF^7$ & $5$ & $8$ & \\

\hline

4 & $\Spin_7$ & $\FF^8$ & $6$ & $12$ & \\

\hline

5 & $\Spin_9$ & $\FF^{16}$ & $14$ & $20$ & \\

\hline

6 & $\SL_3$ & $R(\pi_1 {+} \pi_2)$ & $6$ & $5$ & \\

\hline

7 & $\Sp_4$ & $\mathrm{S}^2 \FF^4$ & $8$ & $6$ & \\

\hline

8 & $\mathsf G_2$ & $R(\pi_2)$ & $12$ & $8$ & \\

\hline

9 & $\SL_2$ & $\mathrm{S}^4 \FF^2$ & $3$ & $2$ & \\

\hline

10 & $\Sp_6$ & $\wedge^2_0 \FF^6$ & $12$ & $12$ & \\

\hline

11 & $\mathsf F_4$ & $R(\pi_4)$ & $24$ & $28$ & \\

\hline

12 & $\Sp_{2n} \times \Sp_4$ & $\FF^{2n} \otimes \FF^4$ & $8n-2$ & $n^2+n+6$ & $n \ge 2$ \\

\hline

13 & $\SL_2 \times \SL_2$ & $\FF^2 \otimes \mathrm{S}^3 \FF^2$ & $6$ & $4$ & \\

\hline
\end{tabular}
\end{table}

For each pair $(H,V)$ listed in the statement of the theorem, $V$ is a spherical $(H \times \FF^\times)$-module, hence $H$ acts spherically on $\PP(V)$, hence on $\SOGr_1(V)$ by Theorem~\ref{thm_subvar_sph}.
As $\dim B_H \ge \dim \SOGr_1(V)$ is a necessary sphericity condition, a case-by-case check of the remaining entries of Table~\ref{table_open_Q1_irr} leaves us with the following three cases, which are treated using Lemma~\ref{lemma_sogr1_sph} and Remark~\ref{rem_sogr1_sph}.
In all the cases, $v$ denotes a highest weight vector of~$V$ (as an $H$-module) and $M$ is a Levi subgroup of the stabilizer in~$H$ of the line~$\langle v \rangle$.

\setcounter{case}{0}

\textit{Case}~\caseno: $H = \Sp_6$, $V = \wedge^2_0 \FF^6$.
The pairs $(M, V / \langle v \rangle)$, $(M, \langle v \rangle)$, and $(M, T_{\langle v \rangle} (H\langle v \rangle))$ are equivalent to
\[
(\SL_2 \times \SL_2 \times \FF^\times, [\mathop{\vphantom|\FF^2} \limits_1{\!} \otimes \mathop{\vphantom|\FF^2} \limits_2{\!}]_\chi \oplus [\mathop{\vphantom|\FF^2} \limits_1{\!} \otimes \mathop{\vphantom|\FF^2} \limits_2{\!}]_{-\chi} \oplus \mathop{\vphantom|\mathrm{S}^2 \FF^2}\limits_1{\!} \oplus \FF^1_{2\chi} \oplus \FF^1),
\]
$(\SL_2 \times \SL_2 \times \FF^\times, \FF^1_{-2\chi})$, and
\[
(\SL_2 \times \SL_2 \times \FF^\times, [\mathop{\vphantom|\FF^2} \limits_1{\!} \otimes \mathop{\vphantom|\FF^2} \limits_2{\!}]_{\chi} \oplus \mathop{[\mathrm{S}^2 \FF^2]}\limits_1{\!}_{2\chi}),
\]
respectively. Then the pair $(M, T_{\langle v \rangle} \PP(V) / T_{\langle v \rangle} (H\langle v \rangle))$ is equivalent to
\[
(\SL_2 \times \SL_2 \times \FF^\times, [\mathop{\vphantom|\FF^2} \limits_1{\!} \otimes \mathop{\vphantom|\FF^2} \limits_2{\!}]_{3\chi} \oplus \FF^1_{4\chi} \oplus \FF^1_{2\chi}).
\]
The latter module contains no spherical submodules of codimension~$1$, hence $\SOGr_1(V)$ is not $H$-spherical.

\textit{Case}~\caseno: $H = \mathsf F_4$, $V = R(\pi_4)$.
The pairs $(M, V / \langle v \rangle)$, $(M, \langle v \rangle)$, and $(M, T_{\langle v \rangle} (H\langle v \rangle))$ are equivalent to
\[
(\Spin_7 \times \FF^\times, \mathop{R(\pi_1)} \limits_1{\!} \oplus [\mathop{R(\pi_3)} \limits_1{\!}]_\chi \oplus [\mathop{R(\pi_3)} \limits_1{\!}]_{-\chi} \oplus \FF^1_{2\chi} \oplus \FF^1),
\]
$(\Spin_7 \times \FF^\times, \FF^1_{-2\chi})$, and
\[
(\Spin_7 \times \FF^\times, [\mathop{R(\pi_1)} \limits_1{\!}]_{2\chi} \oplus [\mathop{R(\pi_3)} \limits_1{\!}]_{\chi}),
\]
respectively (the first pair was computed with LiE using the information in~\cite[\S\,5.12]{LiE2}).
Then the pair $(M, T_{\langle v \rangle} \PP(V) / T_{\langle v \rangle} (H\langle v \rangle))$ is equivalent to
\[
(\Spin_7 \times \FF^\times,
[\mathop{R(\pi_3)} \limits_1{\!}]_{3\chi} \oplus \FF^1_{4\chi} \oplus \FF^1_{2\chi}).
\]
The latter module contains no spherical submodules of codimension~$1$, hence $\SOGr_1(V)$ is not $H$-spherical.

\textit{Case}~\caseno: $H = \Sp_{2n} \times \Sp_4$, $V = \FF^{2n} \otimes \FF^4$, $n \ge 6$.
The pairs $(M, V / \langle v \rangle)$, $(M, \langle v \rangle)$, and $(M, T_{\langle v \rangle} (H\langle v \rangle))$ are equivalent to
\begin{multline*}
(\Sp_{2n-2} \times \SL_2 \times \FF^\times \times \FF^\times,\\
\mathop{\vphantom|\FF^{2n-2}} \limits_1{\!} \otimes \mathop{\vphantom|\FF^2} \limits_2{\!} \oplus \mathop{[\FF^{2n-2}]} \limits_1{\!}_{\chi_2} \oplus \mathop{[\FF^{2n-2}]} \limits_1{\!}_{-\chi_2} \oplus \mathop{[\FF^2]} \limits_2{\!}_{\chi_1} \oplus \mathop{[\FF^2]}\limits_2{\!}_{-\chi_1} \oplus\\
\FF^1_{\chi_1 - \chi_2} \oplus \FF^1_{-\chi_1 + \chi_2} \oplus \FF^1_{\chi_1 + \chi_2}),
\end{multline*}
$(\Sp_{2n-2} \times \SL_2 \times \FF^\times \times \FF^\times, \FF^1_{-\chi_1 - \chi_2})$, and
\[
(\Sp_{2n-2} \times \SL_2 \times \FF^\times \times \FF^\times,
\mathop{[\FF^{2n-2}]} \limits_1{\!}_{\chi_1} \oplus \mathop{[\FF^2]} \limits_2{\!}_{\chi_2} \oplus \FF^1_{2\chi_1} \oplus \FF^1_{2\chi_2}),
\]
respectively. Then the pair $(M, T_{\langle v \rangle} \PP(V) / T_{\langle v \rangle} (H\langle v \rangle))$ is equivalent to
\[
(\Sp_{2n-2} \times \SL_2 \times \FF^\times \times \FF^\times,
[\mathop{\vphantom|\FF^{2n-2}} \limits_1{\!} \otimes \mathop{\vphantom|\FF^2} \limits_2{\!}]_{\chi_1+\chi_2} \oplus \mathop{[\FF^{2n-2}]} \limits_1{\!}_{\chi_1+2\chi_2} \oplus \mathop{[\FF^2]}\limits_2{\!}_{2\chi_1+\chi_2} \oplus
\FF^1_{2\chi_1+2\chi_2}).
\]
The latter module contains no spherical submodules of codimension~$1$, hence $\SOGr_1(V)$ is not $H$-spherical.
\end{proof}

\begin{proposition} \label{prop_so_Q1_V_wr}
Suppose that $V$ is BF-equivalent to $\Omega(W)$ for a simple $H$-module~$W$.
Then $\SOGr_1(V)$ is $H$-spherical if and only if the pair $(H,W)$ is equivalent to one of $(\GL_n, \FF^n)$ $(n \ge 2)$, $(\SL_n, \FF^n)$ $(n \ge 2)$, $(\Sp_{2n} \times \FF^{\times}, \FF^{2n})$ $(n \ge 2)$.
\end{proposition}

\begin{proof}
If $H$ acts spherically on $\SOGr_1(V)$ then $H$ has an open orbit in $\SOGr_1(V)$. According to \cite[Theorem~2.2]{Kim}, all pairs $(H,W)$ (up to equivalence) for which $H$ has an open orbit in $\SOGr_1(\Omega(W))$ are listed in Table~\ref{table_open_Q1_wr}.

\begin{table}[h]
\caption{} \label{table_open_Q1_wr}

\begin{tabular}{|c|c|c|c|c|c|}
\hline
No. & $H$ & $W$ & $\dim \SOGr_1(V)$ & $\dim B_H$ & Note \\

\hline

1 & $\GL_n$ & $\FF^n$ & $2n-2$ & $\frac{n(n+1)}2$ & $n \ge 2$ \\

\hline

2 & $\SL_n$ & $\FF^n$ & $2n-2$ & $\frac{n(n+1)}2 - 1$ & $n \ge 2$ \\

\hline

3 & $\Sp_{2n} \times \FF^\times$ & $[\FF^{2n}]_\chi$ & $4n-2$ & $n^2 + n + 1$ & $n \ge 2$ \\

\hline

4 & $\Sp_{2n}$ & $\FF^{2n}$ & $4n-2$ & $n^2 + n$ & $n \ge 2$ \\

\hline

5 & $\SO_n \times \FF^\times$ & $[\FF^n]_\chi$ & $2n-2$ & $[\frac{n}2][\frac{n+1}2] + 1$ & $n \ge 3$ \\

\hline

6 & $\GL_n \times \SL_2$ & $\FF^n \otimes \FF^2$ & $4n-2$ & $\frac{n(n+1)}2 + 2$ & $n \ge 3$ \\

\hline

7 & $\SL_n \times \SL_2$ & $\FF^n \otimes \FF^2$ & $4n-2$ & $\frac{n(n+1)}2 + 1$ & $n \ge 3$ \\

\hline

8 & $\Spin_7 \times \FF^\times$ & $[\FF^8]_\chi$ & $14$ & $13$ & \\

\hline

9 & $\mathsf G_2 \times \FF^\times$ & $[\FF^7]_\chi$ & $12$ & $9$ & \\

\hline

10 & $\GL_5$ & $\wedge^2 \FF^5$ & $18$ & $15$ & \\

\hline

11 & $\SL_5$ & $\wedge^2 \FF^5$ & $18$ & $14$ & \\

\hline

12 & $\Spin_{10} \times \FF^\times$ & $[\FF^{16}]_\chi$ & $30$ & $26$ & \\

\hline

13 & $\Spin_{10}$ & $\FF^{16}$ & $30$ & $25$ & \\

\hline
\end{tabular}
\end{table}

For each pair $(H,W)$ mentioned in the statement, $V$ is a spherical $(H \times \FF^{\times})$-module (where $\FF^\times$ acts on $V$ by scalar transformations), hence $H$ acts spherically on $\PP(V)$ and hence on $\SOGr_1(V)$ by Theorem~\ref{thm_subvar_sph}. Below we consider all the remaining cases in Table~\ref{table_open_Q1_wr} not satisfying the necessary sphericity condition $\dim B_H \ge \dim \SOGr_1(V)$.

\setcounter{casea}{0}

\textit{Case}~\casenoa: $H = \Sp_{2n}, W = \FF^{2n}$, $n \ge 2$.
The pair $(M, \mathfrak g/(\mathfrak p_I^- + \mathfrak h))$ is equivalent to
\[
(\Sp_{2n-2} \times \FF^{\times}, \mathop{[\FF^{2n-2}]} \limits_1{\!}_\chi \oplus \FF^1),
\]
the latter module being not spherical.

\textit{Case}~\casenoa: $H = \SO_n \times \FF^{\times}, W = \FF^n$, $n \ge 3$.
The pair $(M, \mathfrak g/(\mathfrak p_I^- + \mathfrak h))$ is equivalent to
\[
(\SO_{n-2} \times \FF^{\times} \times \FF^{\times}, \mathop{[\FF^{n-2}]} \limits_1{\!}_{\chi_1 + 2\chi_2} \oplus \FF^1_{2\chi_1} \oplus \FF^1_{2\chi_2}),
\]
the latter module being not spherical.

\textit{Case}~\casenoa: $H = \GL_n \times \SL_2, W = \FF^n \otimes \FF^2$, $n \ge 3$.
Let $v \in V$ be a highest weight vector for the $H$-submodule $W \subset V$ and let $M$ be the Levi subgroup of the stabilizer of the line~$\langle v \rangle$.
Computations using~(\ref{eqn_iso's1}) show that the pair $(M, T_{\langle v \rangle} \PP(V))$ is equivalent to
\begin{multline*}
(\GL_{n-1} \times \FF^{\times} \times \FF^{\times}, \\ [\mathop{\vphantom|\FF^{n-1}} \limits_1{\!}]_{\chi_1} \oplus [\mathop{\vphantom|\FF^{n-1}} \limits_1{\!}]_{\chi_1 + 2\chi_2} \oplus [\mathop{(\FF^{n-1})^*} \limits_1{\!}]_{\chi_1} \oplus [\mathop{(\FF^{n-1})^*} \limits_1{\!}]_{\chi_1 +2\chi_2} \oplus \FF^1_{2\chi_1} \oplus \FF^1_{2\chi_2} \oplus \FF^1_{2\chi_1+2\chi_2}).
\end{multline*}
Using~(\ref{eqn_iso's2}) we find that the pair $(M, T_{\langle v \rangle}(H\langle v \rangle)$ is equivalent to
\[
(\GL_{n-1} \times \FF^{\times} \times \FF^{\times}, [\mathop{\FF^{n-1}} \limits_1{\!}]_{\chi_1} \oplus \FF^1_{2\chi_2}),
\]
hence the pair $(M, T_{\langle v \rangle} \PP(V))/T_{\langle v \rangle}(H\langle v \rangle)$ is equivalent to
\[
(\GL_{n-1} \times \FF^{\times} \times \FF^{\times}, [\mathop{\vphantom|\FF^{n-1}} \limits_1{\!}]_{\chi_1 +2\chi_2} \oplus [\mathop{\vphantom|(\FF^{n-1})^*} \limits_1{\!}]_{\chi_1} \oplus [\mathop{\vphantom|(\FF^{n-1})^*} \limits_1{\!}]_{\chi_1 +2\chi_2} \oplus \FF^1_{2\chi_1} \oplus \FF^1_{2\chi_1+2\chi_2}).
\]
Since the latter module does not contain spherical submodules of codimension~$1$, the variety $\SOGr_1(V)$ is not $H$-spherical by Lemma~\ref{lemma_sogr1_sph}.

\textit{Case}~\casenoa: $H = \SL_n \times \SL_2, W = \FF^n \otimes \FF^2$, $n \ge 3$.
This is a subgroup of the group in the previous case, which does not act spherically on~$\SOGr_1(V)$.
\end{proof}

\begin{proposition} \label{prop_so_Q1_V_sr_aux}
Suppose that $V = \widetilde V_1 \oplus \widetilde V_2$ for two pairwise orthogonal nonzero $H$-submodules $\widetilde V_1, \widetilde V_2 \subset V$.
Then $\SOGr_1(V)$ is $H$-spherical if and only if $V$ is a spherical module with respect to the action of $H \times \FF^\times \times \FF^\times$, where the first \textup(resp. second\textup) factor $\FF^\times$ acts on $\widetilde V_1$ \textup(resp.~$\widetilde V_2$\textup) by scalar transformations.
\end{proposition}

\begin{proof}
Put
\[
U = \lbrace \langle v_1 + v_2 \rangle \in \SOGr_1(V) \mid v_1 \in \widetilde V_1, \ v_2 \in \widetilde V_2, \ \omega(v_1,v_1) \ne 0, \ \omega(v_2,v_2) \ne 0 \rbrace.
\]
Then $U$ is an open subset of $\SOGr_1(V)$.
Define a map $\varphi \colon U \to \PP(\widetilde V_1) \times \PP(\widetilde V_2)$ by $\langle v_1 + v_2 \rangle \mapsto (\langle v_1 \rangle, \langle v_2\rangle)$.
It is easy to see that this map is $H$-equivariant, its image is
\[
(\PP(\widetilde V_1)\setminus \SOGr_1(\widetilde V_1))\times (\PP(\widetilde V_2) \setminus \SOGr_1(\widetilde V_2))
\]
(which is open in $\PP(\widetilde V_1) \times \PP(\widetilde V_2)$), and $\varphi$ is a two-fold covering over the image. It follows that $\SOGr_1(V)$ is $H$-spherical if and only if $\PP(\widetilde V_1) \times \PP(\widetilde V_2)$ is $H$-spherical.
The latter is equivalent to the fact that $V$ is a spherical $(H \times \FF^\times \times \FF^\times)$-module.
\end{proof}

\begin{corollary} \label{crl_no_3_summands}
If $r \ge 3$ then $\SOGr_1(V)$ is not $H$-spherical.
\end{corollary}

\begin{proof}
If $H$ acts spherically on $V$ then so does the group $\SO(V_1) \times \ldots \times \SO(V_r)$.
By Proposition~\ref{prop_so_Q1_V_sr_aux}, in this situation $V_2 \oplus \ldots \oplus V_r$ should be a spherical module with respect to the action of $\SO(V_2) \times \ldots \times \SO(V_r) \times \FF^\times$ with $\FF^\times$ acting by scalar transformations. However the latter module is not spherical.
\end{proof}

\begin{proposition} \label{prop_so_Q1_V_sr}
Suppose that $r \ge 2$.
Then $\SOGr_1(V)$ is $H$-spherical if and only if $r = 2$ and one of the following two conditions holds:
\begin{enumerate}[label=\textup{(\arabic*)},ref=\textup{\arabic*}]
\item \label{prop_so_Q1_V_sr_1}
$H_0 = H_1 \times H_2$ and for $i = 1,2$ the pair $(H_i,V_i)$ is BF-equivalent to a pair in Table~\textup{\ref{table_orth_sph_mod}};

\item \label{prop_so_Q1_V_sr_2}
the pair $(H,V)$ is BF-equivalent to a pair in Table~\textup{\ref{table_so_Q1}}.
\end{enumerate}
\end{proposition}

\begin{proof}
Corollary~\ref{crl_no_3_summands} implies $r=2$. In this case,
according to Proposition~\ref{prop_so_Q1_V_sr_aux}, $\SOGr_1(V)$ is $H$-spherical if and only if $V_1 \oplus V_2$ is a spherical $(H \times \FF^\times \times \FF^\times)$-module.
In particular, for each $i=1,2$ the summand $V_i$ is a spherical $(H_i \times \FF^\times)$-module, hence by Proposition~\ref{prop_spherical+orth} the pair $(H_i,V_i)$ is equivalent to one of those in Table~\ref{table_orth_sph_mod}.
If $H_0 = H_1 \times H_2$ then we get~(\ref{prop_so_Q1_V_sr_1}).
If $H_0$ is a proper subgroup of~$H_1 \times H_2$ then there is a connected normal subgroup $K \subset H_0$ that acts nontrivially on both $V_1$ and~$V_2$.
If $K$ contains a simple factor then the saturation of the $H$-module $V$ is indecomposable, hence it should be contained in the list of~\cite[\S\,5]{Kn}.
The latter implies that both $V_1,V_2$ are simple $H$-modules and an easy case-by-case check of the above-cited list yields the first two cases of Table~\ref{table_so_Q1}.
If $K$ contains no simple factors then, inspecting Table~\ref{table_orth_sph_mod}, we find that $K \simeq \FF^\times$ and for $i=1,2$ the pair $(H_i, V_i)$ is BF-equivalent to either $(\SL_{n_i} \times \FF^\times, \Omega([\FF^{n_i}]_\chi))$ ($n_i \ge 1$) or $(\Sp_{2n_i} \times \FF^\times, \Omega([\FF^{2n_i}]_\chi))$ ($n_i \ge 2$).
Now an application of~\cite[Theorem~5.3]{AvP1} yields the last two cases in Table~\ref{table_so_Q1}, whence~(\ref{prop_so_Q1_V_sr_2}).
\end{proof}

\subsection{Spherical actions on \texorpdfstring{$\SOGr_{\max}^{(\pm)}(V)$}{SOGr\_max\textasciicircum(\textpm)} for \texorpdfstring{$r=1$}{r=1}}
\label{subsec_SOGr_max_r=1}

Throughout \S\S\,\ref{subsec_SOGr_max_r=1}--\ref{subsec_SOGr_max_r>=3} we keep in mind the following identifications (see~\S\,\ref{subsec_FV_via_compositions}):
\begin{itemize}
\item
if $d = 2k+1$ then $\SOGr_{\max}(V) \simeq X_I$ with $I = \lbrace k \rbrace$;

\item
if $d = 2k$ then $\SOGr^+_{\max}(V) \simeq X_I$ with $I = \lbrace k \rbrace$ and $\SOGr^-_{\max}(V) \simeq X_I$ with $I = \lbrace k - 1 \rbrace$.
\end{itemize}

\begin{proposition} \label{prop_so_Qmax_odd_V_irr}
Suppose that $V$ is a simple $H$-module and $d$ is odd.
Then $\SOGr_{\max}(V)$ is $H$-spherical if and only if the pair $(H,V)$ is BF-equivalent to one of $(\SO_n, \FF^n)$ $(n \ge 3)$ or $(\mathsf G_2, \FF^7)$.
\end{proposition}

\begin{proof}
As $d$ is odd, $\SOGr_{\max}(V)$ being $H$-spherical implies that $\SOGr_1(V)$ is $H$-spherical by Proposition~\ref{prop_so_odd_Q1_is_minimal} and Theorem~\ref{thm_sph_descent}.
Then Proposition~\ref{prop_so_Q1_V_irr} leaves us with the following two cases.

\setcounter{caseb}{0}

\textit{Case}~\casenob: $(\SO_n, \FF^n)$, $n \ge 3$.
In this case $\SOGr_{\max}(V)$ is $H$-spherical.

\textit{Case}~\casenob: $(\mathsf G_2, \FF^7)$.
In this case the pair $(M, \mathfrak g/(\mathfrak p_I^- + \mathfrak h))$ is equivalent to $(\FF^\times, \FF^1_{\chi})$, the latter module being spherical.
\end{proof}

\begin{proposition} \label{prop_so_Qmax_even_V_irr}
Suppose that $V$ is a simple $H$-module, $d$ is even, and $* \in \lbrace +,- \rbrace$.
Then $\SOGr_{\max}^{*}(V)$ is $H$-spherical if and only if the pair $(H,V)$ is BF-equivalent to one of $(\SO_{2n}, \FF^{2n})$ $(n \ge 2)$, $(\Sp_4 \times \SL_2, \FF^4 \otimes \FF^2)$, $(\Spin_7, \FF^8)$.
\end{proposition}

\begin{proof}
If $H$ acts spherically on $\SOGr_{\max}^*(V)$ then $H$ has an open orbit in $\SOGr_{\max}^*(V)$. According to \cite[Theorem~2.1]{Kim}, all pairs $(H,V)$ (up to BF-equivalence) for which $H$ has an open orbit in at least one of the two varieties $\SOGr_{\max}^\pm(V)$ are listed in Table~\ref{table_open_Qmax_even_irr}.

\begin{table}[h]
\caption{} \label{table_open_Qmax_even_irr}

\begin{tabular}{|c|c|c|c|c|c|}
\hline
No. & $H$ & $V$ & $\dim \SOGr_{\max}(V)$ & $\dim B_H$ & Note \\

\hline

1 & $\SO_{2n}$ & $\FF^{2n}$ & $\frac{n(n-1)}2$ & $n^2$ & $n \ge 2$ \\

\hline

2 & $\Sp_{2n} \times \SL_2$ & $\FF^{2n} \otimes \FF^2$ & $2n^2 - n$ & $n^2 + n + 2$ & $n \ge 2$ \\

\hline

3 & $\Spin_7$ & $\FF^8$ & $6$ & $12$ & \\

\hline

4 & $\Spin_9$ & $\FF^{16}$ & $28$ & $20$ & \\

\hline

5 & $\SL_3$ & $R(\pi_1 {+} \pi_2)$ & $6$ & $5$ & \\

\hline

6 & $\Sp_4$ & $\mathrm{S}^2\FF^4$ & $10$ & $6$ & \\

\hline

7 & $\Sp_6$ & $\wedge^2_0 \FF^6$ & $21$ & $12$ & \\

\hline
\end{tabular}
\end{table}

Taking into account the necessary sphericity condition $\dim B_H \ge \dim \SOGr_{\max}^\pm(V)$, we are left with the following cases.

\setcounter{casec}{0}

\textit{Case}~\casenoc: $(\SO_{2n}, \FF^{2n})$, $n \ge 2$.
In this case both varieties $\SOGr_{\max}^\pm(V)$ are $H$-spherical.

\textit{Case}~\casenoc: $(\Sp_4 \times \SL_2, \FF^4 \otimes \FF^2)$.
Up to isomorphism, it suffices to assume that $H_0 \subset \SO(V)$ is the default embedding (see~\S\,\ref{subsec_so_prelim_remarks}).
For the action on $\SOGr_{\max}^+(V)$ the pair $(M, \mathfrak g/(\mathfrak p_I^- + \mathfrak h))$ is equivalent to $(\SO_5 \times \FF^\times, [\FF^5]_\chi)$, the latter module being spherical.
For the action on $\SOGr_{\max}^-(V)$ the pair $(M, \mathfrak g/(\mathfrak p_I^- + \mathfrak h))$ is equivalent to $(\GL_2, \FF^2)$, the latter module being spherical.

\textit{Case}~\casenoc: $(\Spin_7, \FF^8)$.
In this case both varieties $\SOGr_{\max}^{\pm}(V)$ are $H$-spherical.
\end{proof}

\begin{proposition} \label{prop_so_Qmax_V_wr}
Suppose that $V$ is BF-equivalent to $\Omega(W)$ for a simple $H$-module~$W$ and $* \in \lbrace +, - \rbrace$.
Then $\SOGr_{\max}^*(V)$ is $H$-spherical if and only if the pair $(H,W)$, considered up to equivalence, and $*$ appear in Table~\textup{\ref{table_so_even_Qmax_wr}}.
\end{proposition}

\begin{table}[h]
\caption{} \label{table_so_even_Qmax_wr}

\begin{tabular}{|c|l|c|c|}
\hline
No. & \multicolumn{1}{|c|}{$(H,W)$} & $*$ &  Note \\

\hline

\newcase &
$(\GL_n, \FF^n)$ & $\pm$ & $n \ge 2$ \\

\hline

\newcase &
$(\SL_{2n+1}, \FF^{2n+1})$ & $+$ & $n \ge 1$ \\

\hline

\newcase &
$(\SL_n, \FF^n)$ & $-$ & $n \ge 2$ \\

\hline

\newcase &
$(\SO_3 \times \FF^\times, [\FF^3]_\chi)$ & $\pm$ & \\

\hline

\newcase &
$(\Sp_4 \times \FF^\times, [\FF^4]_\chi)$ & $-$ & \\

\hline
\end{tabular}
\end{table}

\begin{proof}
If $H$ acts spherically on $\SOGr_{\max}^*(V)$ then $H$ has an open orbit in $\SOGr_{\max}^*(V)$.
According to \cite[Theorem~2.2]{Kim}, $H$ has an open orbit in at least one of the varieties $\SOGr_{\max}^\pm(V)$ if and only if the pair $(H,W)$ is equivalent to one of $(\GL_n, \FF^n)$ ($n \ge 2$), $(\SL_n, \FF^n)$ ($n \ge 2$), $(\SO_3 \times \FF^\times, [\FF^3]_\chi)$, or $(\Sp_4 \times \FF^\times, [\FF^4]_\chi)$.
Now it suffices to compute the pair $(M, \mathfrak g/(\mathfrak p_I^- + \mathfrak h))$ for each of these pairs $(H,W)$ and each~$*$ and conclude whether the resulting module is spherical or not.
The results for each of the cases are summarized in Table~\ref{table_open_Qmax_wr}.
\end{proof}

\begin{table}[h]
\caption{} \label{table_open_Qmax_wr}

\begin{tabular}{|c|c|c|c|c|c|}
\hline
No. & $(H,W)$ & $*$ & $(M, \mathfrak g/(\mathfrak p_I^- + \mathfrak h))$ & Conclusion \\

\hline

\multirow{2}{*}{1} & \multirow{2}{*}{$(\GL_n,\FF^n)$, $n {\ge} 2$} & $+$ & $(\GL_n, \wedge^2 \FF^n)$ & spherical \\ \cline{3-5}

& & $-$ & $(\GL_{n-1} \times \FF^\times, \wedge^2 \FF^{n-1})$ & spherical \\

\hline

\multirow{2}{*}{2} & \multirow{2}{*}{$(\SL_n,\FF^n)$, $n {\ge} 2$} & $+$ & $(\SL_n, \wedge^2 \FF^n)$ & spherical iff $n$ is odd \\ \cline{3-5}

& & $-$ & $(\GL_{n-1}, \wedge^2 \FF^{n-1})$ & spherical\\

\hline

\multirow{2}{*}{3} & \multirow{2}{*}{$(\SO_3 \times \FF^\times,[\FF^3]_\chi)$} & $+$ & $(\SO_3 \times \FF^\times, [\FF^3]_\chi)$ & spherical \\ \cline{3-5}

& & $-$ & $(\FF^\times {\times} \FF^\times, \FF^1_{\chi_1} {\oplus} \FF^1_{\chi_1 +\chi_2})$ & spherical \\

\hline

\multirow{2}{*}{4} & \multirow{2}{*}{$(\Sp_4 \times \FF^\times, [\FF^4]_\chi)$} & $+$ & $(\Sp_4 \times \FF^\times, [\wedge^2 \FF^4]_\chi)$ & not spherical \\ \cline{3-5}

&  & $-$ & $(\SL_2 {\times} \FF^\times {\times} \FF^\times, [\FF^2]_{\chi_1 + \chi_2} {\oplus} \FF^1_{\chi_1})$ & spherical \\

\hline
\end{tabular}

\end{table}

\subsection{Spherical actions on \texorpdfstring{$\SOGr_{\max}^{(\pm)}(V)$}{SOGr\_max\textasciicircum(\textpm)} for \texorpdfstring{$r = 2$}{r=2}}
\label{subsec_SOGr_max_r=2}

We begin with an auxiliary result.

\begin{proposition} \label{prop_so_aux}
Suppose that $V = \widetilde V_1 \oplus \widetilde V_2$ for two pairwise orthogonal nonzero $H$-submodules $\widetilde V_1, \widetilde V_2 \subset V$, $\dim \widetilde V_1 \ge 3$, and
\begin{itemize}
\item
$H$ acts spherically on $\SOGr_{\max}(V)$ if $d$ is odd;

\item
$H$ acts spherically on at least one of $\SOGr_{\max}^{\pm}(V)$ if $d$ is even.
\end{itemize}
Then
\begin{enumerate}[label=\textup{(\alph*)},ref=\textup{\alph*}]
\item \label{prop_so_aux_a}
if $\dim \widetilde V_1$ is odd then $\SOGr_{\max}(\widetilde V_1)$ is $H$-spherical;

\item \label{prop_so_aux_b}
if $\dim \widetilde V_1$ is even then both varieties $\SOGr_{\max}^{\pm}(\widetilde V_1)$ are $H$-spherical.
\end{enumerate}
\end{proposition}

\begin{proof}
(\ref{prop_so_aux_a})
\textit{Case}~1: $\dim \widetilde V_2$ is even.
Choose maximal isotropic subspaces $U_1 \subset \widetilde V_1$, $U_2 \subset \widetilde V_2$ and put $U = U_1 \oplus U_2$.
Then $U \in \SOGr_{\max}(V)$.
Let $Q$ be the stabilizer of $U$ in the group $\SO(\widetilde V_1) \times \SO(\widetilde V_2)$.
It is easy to see that $Q = Q_1 \times Q_2$ where $Q_i$ is the stabilizer of $U_i$ in $\SO(\widetilde V_i)$ for $i=1,2$.
It follows that the $(\SO(\widetilde V_1)\times \SO(\widetilde V_2))$-orbit $O$ of $U$ in $\SOGr_{\max}(V)$ is closed and isomorphic to $\SOGr_{\max}(\widetilde V_1) \times \SOGr_{\max}^*(\widetilde V_2)$ for some choice of $* \in \lbrace +, - \rbrace$.
Theorem~\ref{thm_subvar_sph} implies that $O$ is $H$-spherical, hence $\SOGr_{\max}(\widetilde V_1)$ is also $H$-spherical.

\textit{Case}~2: $\dim \widetilde V_2$ is odd.
Choose $X \in \lbrace \SOGr_{\max}^+(V), \SOGr_{\max}^-(V) \rbrace$ such that $X$ is $H$-spherical.
Choose maximal isotropic subspaces $U_1 \subset \widetilde V_1, U_2 \subset \widetilde V_2$.
For $i=1,2$ let $U_i^\perp$ be the orthogonal complement of $U_i$ in $\widetilde V_i$ and choose $u_i \in U_i^\perp \setminus U_i$ in such a way that the vector $u = u_1 + u_2$ is isotropic.
(Note that for $i=1,2$ the vector $u_i$ is nonisotropic and $U_i^\perp = U_i \oplus \langle u_i \rangle$.)
Then $U = U_1 \oplus \langle u\rangle \oplus U_2$ is a maximal isotropic subspace in~$V$.
Replacing $u_2$ with $-u_2$ if necessary we may assume that $U \in X$.
Let $Q$ be the stabilizer of $U$ in the group $\SO(\widetilde V_1) \times \SO(\widetilde V_2)$.
It is easy to see that for each $i=1,2$ the group $Q$ stabilizes the subspace $U_i^\perp$ and hence the subspace $U_i$, which implies that $Q = Q_1 \times Q_2$ where $Q_i$ is the stabilizer of $U_i$ in $\SO(\widetilde V_i)$ for $i = 1,2$.
It follows that the $(\SO(\widetilde V_1)\times \SO(\widetilde V_2))$-orbit $O$ of $U$ in $\SOGr_{\max}(V)$ is closed and isomorphic to $\SOGr_{\max}(\widetilde V_1) \times \SOGr_{\max}(\widetilde V_2)$.
The rest of the argument is as in Case~1.

(\ref{prop_so_aux_b})
\textit{Case}~1: $\dim \widetilde V_2$ is odd.
Fix $* \in \lbrace +,- \rbrace$ and choose maximal isotropic subspaces $U_1 \subset \widetilde V_1, U_2 \subset \widetilde V_2$ such that $U_1 \in \SOGr_{\max}^*(\widetilde V_1)$.
Then $U \in \SOGr_{\max}(V)$.
Now an argument similar to that in Case~1 of part~(\ref{prop_so_aux_a}) shows that $\SOGr_{\max}^*(\widetilde V_1)$ is $H$-spherical.

\textit{Case}~2: $\dim \widetilde V_2$ is even.
Choose $X \in \lbrace \SOGr_{\max}^+(V), \SOGr_{\max}^-(V) \rbrace$ such that $X$ is $H$-spherical and fix $* \in \lbrace +, - \rbrace$.
Choose maximal isotropic subspaces $U_1 \subset \widetilde V_1, U_2 \subset \widetilde V_2$ and put $U = U_1 \oplus U_2$.
We may assume $U_1 \in \SOGr_{\max}^*(\widetilde V_1)$.
Acting on $U_2$ by an element of $\mathrm{O}(\widetilde V_2) \setminus \SO(\widetilde V_2)$ if necessary we may also assume that $U \in X$.
Now an argument similar to that in Case~1 of part~(\ref{prop_so_aux_a}) shows that $\SOGr_{\max}^*(\widetilde V_1)$ is $H$-spherical.
\end{proof}

\begin{proposition} \label{prop_so_odd_Qmax_sr}
Suppose that $r = 2$ and $d$ is odd.
Then $\SOGr_{\max}(V)$ is $H$-spherical if and only if the pair $(H,V)$ is BF-equivalent to one of $(\SO_{2l} \times \SO_{2m+1}$, $\FF^{2l} \oplus \nobreak \FF^{2m+1})$ $(l \ge 1, m \ge 0)$, $(\Spin_7, \FF^8 \oplus \FF^1)$, $(\GL_n, \Omega(\FF^n) \oplus \FF^1)$ $(n \ge 2)$.
\end{proposition}

\begin{proof}
As $d$ is odd, $\SOGr_{\max}(V)$ being $H$-spherical implies that $\SOGr_1(V)$ is $H$-spherical by Proposition~\ref{prop_so_odd_Q1_is_minimal} and Theorem~\ref{thm_sph_descent}.

Without loss of generality we may assume that $\dim V_1 = 2l$ is even and $\dim V_2 = 2m+1$ is odd.
Then $V_2$ is a simple $H$-module, in which case the pair $(H_2, V_2)$ can be BF-equivalent to one of the two pairs $(\SO_{2m+1}, \FF^{2m+1})$ or $(\mathsf G_2, \FF^7)$ by Proposition~\ref{prop_so_Q1_V_sr}.
We first show that the group $\SO_{2l} \times \mathsf G_2$ does not act spherically on $\SOGr_{\max}(\FF^{2l} \oplus \FF^7)$.
Indeed, in this case the pair $(M, \mathfrak g/(\mathfrak p_I^- + \mathfrak h))$ is equivalent to
\[
(\GL_{l} \times \SL_2 \times \FF^\times, \mathop{\vphantom|\FF^{l}} \limits_1{\!} \oplus
[\mathop{\vphantom|\FF^{l}}\limits_1{\!} \otimes \mathop{\vphantom|\FF^2}\limits_2{\!}]_\chi \oplus \mathop{[\FF^{l}]}\limits_1{\!}_{2\chi} \oplus \FF^1_{2\chi}),
\]
the latter module being not spherical.

We have proved that a necessary $H$-sphericity condition for $\SOGr_{\max}(V)$ is that the pair $(H_2, V_2)$ is BF-equivalent to $(\SO_{2m+1}, \FF^{2m+1})$.
Then according to Proposition~\ref{prop_so_Q1_V_sr} the pair $(H_1,V_1)$ is BF-equivalent to one of those in Table~\ref{table_orth_sph_mod}.
Below we consider all these possibilities for the pair $(H_1,V_1)$ up to BF-equivalence.

\setcounter{cased}{0}

\textit{Case}~\casenod: $(\SO_{2l}, \FF^{2l})$, $l \ge 1$.

If $H_0 = H_1 \times H_2$ then the pair $(M, \mathfrak g/(\mathfrak p_I^- + \mathfrak h))$ is equivalent to
\[
(\GL_{l} \times \GL_{m}, \mathop{\FF^{l}} \limits_1{\!} \oplus \mathop{\FF^{l}} \limits_1{\!} \otimes \mathop{\FF^{m}} \limits_2{\!}),
\]
the latter module being spherical.

If the pair $(H,V)$ is BF-equivalent to $(\SL_2 \times \SL_2, \mathop{\FF^2} \limits_1{\!} {\otimes} \mathop{\FF^2} \limits_2{\!} \oplus \mathop{\mathrm{S}^2 \FF^2} \limits_1{\!})$ then
\[
\dim B_{H} = 4 < 6 = \dim \SOGr_{\max}(V),
\]
hence $\SOGr_{\max}(V)$ is not $H$-spherical.

\textit{Case}~\casenod: $(\Sp_{2n} \times \SL_2, \FF^{2n} \otimes \FF^2)$, $n \ge 2$.
By Proposition~\ref{prop_so_aux}(\ref{prop_so_aux_b}), a necessary $H$-spheric\-i\-ty condition for $\SOGr_{\max}(V)$ is that $\Sp_{2n} \times \SL_2$ acts spherically on both varieties $\SOGr_{\max}^\pm(V_1)$, which implies $n = 2$ by Proposition~\ref{prop_so_Qmax_even_V_irr}.
In this case, by Lemma~\ref{lemma_outer_aut}(\ref{lemma_outer_aut_b}) it suffices to use only the default embedding $H_1 \subset \SO(V_1)$ as described in~\S\,\ref{subsec_so_prelim_remarks}.

If $H_0 = H_1 \times H_2$ then the pair $(M, \mathfrak g/(\mathfrak p_I^- + \mathfrak h))$ is equivalent to
\[
(\Sp_4 \times \GL_{m} \times \FF^\times, [\mathop{\vphantom|\FF^4} \limits_1{\!}]_\chi \oplus [\mathop{\vphantom|\FF^4} \limits_1{\!} \otimes \mathop{\vphantom|\FF^{m}} \limits_2{\!}]_\chi \oplus [\mathop{\vphantom|\wedge^2_0 \FF^4} \limits_1{\!}]_{2\chi}),
\]
the latter module being not spherical.
It also follows from the above that $\SOGr_{\max}(V)$ is not $H$-spherical if $H_0$ is a proper subgroup of $H_1 \times H_2$.

\textit{Case}~\casenod: $(\Spin_7, \FF^8)$.
By Lemma~\ref{lemma_outer_aut}(\ref{lemma_outer_aut_b}), it suffices to do the computations only for $\Spin^+_7$.
If $H_0 = H_1 \times H_2$ then the pair $(M, \mathfrak g/(\mathfrak p_I^- + \mathfrak h))$ is equivalent to
\[
(\SL_3 \times \GL_{m} \times \FF^\times, [\mathop{\vphantom|\FF^3} \limits_1{\!}]_\chi \oplus [\mathop{\vphantom|\FF^3} \limits_1{\!} \otimes \mathop{\vphantom|\FF^{m}} \limits_2{\!}]_\chi \oplus [\mathop{\vphantom| \FF^m} \limits_2{\!}]_{3\chi} \oplus \FF^1_{3\chi}),
\]
the latter module being spherical if and only if $m = 0$.
It also follows from the above that $\SOGr_{\max}(V)$ is not $H$-spherical if $H_0$ is a proper subgroup of $H_1 \times H_2$.

\textit{Case}~\casenod: $(\Spin_9, \FF^{16})$.
By Proposition~\ref{prop_so_aux}(\ref{prop_so_aux_b}), a necessary sphericity condition is that $\Spin_9$ acts spherically on both varieties $\SOGr_{\max}^\pm(\FF^{16})$, which is not the case by Proposition~\ref{prop_so_Qmax_even_V_irr}.

\textit{Case}~\casenod: $(\GL_l, \Omega(\FF^l))$, $l \ge 2$.

If $H_0 = H_1 \times H_2$ then the pair $(M, \mathfrak g/(\mathfrak p_I^- + \mathfrak h))$ is equivalent to
\[
(\GL_l \times \GL_{m}, \mathop{\vphantom|\FF^l} \limits_1{\!} \oplus \mathop{\vphantom|\FF^l} \limits_1{\!} \otimes \mathop{\vphantom|\FF^{m}} \limits_2{\!} \oplus \mathop{\vphantom|\wedge^2 \FF^l} \limits_1{\!}),
\]
the latter module being spherical if and only if $m = 0$.
It also follows from the above that $\SOGr_{\max}(V)$ is not $H$-spherical if $H_0$ is a proper subgroup of $H_1 \times H_2$.

\textit{Case}~\casenod: $(\SL_l, \Omega(\FF^l))$ ($l \ge 3$).
A necessary $H$-sphericity condition is that $\SOGr_{\max}(V)$ is spherical for the group $\GL_l \times \SO_{2m+1}$, which implies $m=0$ by the previous case.
Then the pair $(M, \mathfrak g/(\mathfrak p_I^- + \mathfrak h))$ is equivalent to $(\SL_l, \mathop{\FF^l} \limits_1{\!} \oplus \mathop{\wedge^2 \FF^l} \limits_1{\!})$, the latter module being not spherical.
It also follows from the above that $\SOGr_{\max}(V)$ is not $H$-spherical if $H_0$ is a proper subgroup of $H_1 \times H_2$.

\textit{Case}~\casenod: $(\Sp_{2n} \times \FF^\times, \Omega([\FF^{2n}]_\chi))$, $n \ge 2$.
By Proposition~\ref{prop_so_aux}(\ref{prop_so_aux_b}), a necessary $H$-sphericity condition for $\SOGr_{\max}(V)$ is that $\Sp_{2n} \times \FF^\times$ acts spherically on both varieties $\SOGr_{\max}^\pm(V_1)$, which is not the case by Proposition~\ref{prop_so_Qmax_V_wr}.
\end{proof}

\begin{proposition} \label{prop_so_even_Qmax_V1+V2}
Suppose that $r = 2$, $d$ is even, and $* \in \lbrace +, - \rbrace$.
Then $\SOGr_{\max}^*(V)$ is $H$-spherical if and only if the pair $(H,V)$, considered up to BF-equivalence, and $*$ are listed in Table~\textup{\ref{table_so_even_Qmax_V1+V2}}.
\end{proposition}

\begin{table}[h]
\caption{} \label{table_so_even_Qmax_V1+V2}

\begin{tabular}{|c|l|c|c|}
\hline
No. & \multicolumn{1}{|c|}{$(H,V)$} & $*$ &  Note \\

\hline

\newcase &
$(\SO_p \times \SO_q, \mathop{\FF^p}\limits_1{\!} \oplus \mathop{\FF^q}\limits_2{\!})$ & $\pm$ & $p,q {\ge} 1$, $p{+}q = 2k \ge 6$
\\

\hline

\newcase &
$(\Spin_7 \times \SO_{2l}, \mathop{\FF^8} \limits_1{\!} \oplus \mathop{\FF^{2l}} \limits_2{\!})$ & $\pm$ & $l \ge 1$
\\

\hline

\newcase &
$(\GL_m \times \FF^\times, \Omega(\FF^m) \oplus \Omega(\FF^1_\chi))$ & $\pm$ & $m \ge 2$
\\

\hline

\newcase &
$(\SL_{2m+1} \times \FF^\times, \Omega(\FF^{2m+1}) \oplus \Omega(\FF^1_\chi))$ & $\pm$ & $m \ge 1$\\

\hline

\newcase &
\renewcommand{\tabcolsep}{0pt}%
\begin{tabular}{l}
$(\SL_{2m} \times \FF^\times, \Omega([\FF^{2m}]_{a\chi}) \oplus \Omega(\FF^1_{b\chi}))$,\\[-2pt]
$a,b \in \ZZ \setminus \lbrace 0 \rbrace$
\end{tabular}
& $\pm$ & $m \ge 1$
\\

\hline

\newcase &
\renewcommand{\tabcolsep}{0pt}%
\begin{tabular}{l}
$(\SL_{2m+1} \times \FF^\times, \Omega([\FF^{2m+1}]_{a\chi}) \oplus \Omega(\FF^1_{b\chi}))$,\\[-2pt]
$a,b \in \ZZ \setminus \lbrace 0 \rbrace$, $b \ne - (2m+1)a$
\end{tabular}
& $+$ & $m \ge 1$
\\

\hline

\newcase &
\renewcommand{\tabcolsep}{0pt}%
\begin{tabular}{l}
$(\SL_{2m+1} \times \FF^\times, \Omega([\FF^{2m+1}]_{a\chi}) \oplus \Omega(\FF^1_{b\chi}))$,\\[-2pt]
$a,b \in \ZZ \setminus \lbrace 0 \rbrace$, $b \ne (2m+1)a$
\end{tabular}
& $-$ & $m \ge 1$
\\

\hline

\newcase &
$(\GL_3 \times \SO_{2l}, \Omega(\mathop{\FF^3} \limits_1{\!}) \oplus \mathop{\FF^{2l}} \limits_2{\!})$ & $\pm$ & $l \ge 2$
\\

\hline

\newcase &
$(\GL_2 \times \SO_{2l}, \Omega(\mathop{\FF^2} \limits_1{\!}) \oplus \mathop{\FF^{2l}} \limits_2{\!})$ & $\pm$ & $l \ge 2$
\\

\hline

\newcase &
$(\mathsf G_2, \FF^7 \oplus \FF^1)$ & $\pm$ &
\\

\hline

\newcase &
$(\mathsf G_2 \times \SO_3, \mathop{\FF^7} \limits_1{\!} \oplus \mathop{\FF^3} \limits_2{\!})$ & $\pm$ &
\\

\hline

\newcase &
$(\SL_2 \times \SL_2 \times \SL_2, \mathop{\FF^2} \limits_1{\!} \otimes \mathop{\FF^2} \limits_3{\!} \oplus \mathop{\FF^2} \limits_2{\!} \otimes \mathop{\FF^2} \limits_3{\!})$ & $+$ & $H_0 \subset \Spin^+_7 \subset \SO(V)$
\\

\hline

\newcase &
$(\Spin_7 \times \GL_2, \mathop{\FF^8} \limits_1{\!} \oplus \Omega(\mathop{\FF^2} \limits_2{\!}))$
& $+$ & \\

\hline
\end{tabular}

\end{table}

\begin{proof}
Put $p = \dim V_1$, $q = \dim V_2$. The proof is divided into two parts depending on the parity of~$p$ and~$q$.

Part~1: $p, q$ are odd.

If $\min(p,q) = 1$ then applying Propositions~\ref{prop_so_2n_to_2n-1} and~\ref{prop_so_Qmax_odd_V_irr} yields that for each $* \in \lbrace +, - \rbrace$ the variety $\SOGr_{\max}^*(V)$ is $H$-spherical if and only if the pair $(H,V)$ is BF-equivalent to one of $(\SO_{2m+1}, \FF^{2m+1} \oplus \FF^1)$ ($m \ge 2$) or $(\mathsf G_2, \FF^7 \oplus \FF^1)$.

Now assume $\min(p,q) \ge 3$.
By Proposition~\ref{prop_so_aux}(\ref{prop_so_aux_a}), a necessary $H$-sphericity condition for $\SOGr_{\max}^\pm(V)$ is that $H$ acts spherically on both varieties $\SOGr_{\max}(V_1)$ and $\SOGr_{\max}(V_2)$, hence by Proposition~\ref{prop_so_Qmax_odd_V_irr} each pair $(H_1, V_1)$ and $(H_2,V_2)$ is BF-equivalent to either $(\SO_{2m+1}, \FF^{2m+1})$ for some $m \ge 1$ or $(\mathsf G_2, \FF^7)$.

If the pair $(H,V)$ is BF-equivalent to $(\SO_{2m+1} \times \SO_{2l+1}, \mathop{\FF^{2m+1}} \limits_1{\!} \oplus \mathop{\FF^{2l+1}} \limits_2{\!})$ then the pair $(M, \mathfrak g/(\mathfrak p_I^- + \mathfrak h))$ is equivalent to $(\GL_{m} \times \GL_{l}, \FF^{m} \otimes \FF^{l})$, the latter module being spherical.

If the pair $(H,V)$ is BF-equivalent to $(\SO_{2m+1}, \FF^{2m+1} \oplus \FF^{2m+1})$ then
\[
\dim B_{H} = m^2+m < 2m^2+m = \dim \SOGr_{\max}^\pm(V),
\]
whence $\SOGr_{\max}^\pm(V)$ is not spherical.

If $(H,V)$ is BF-equivalent to $(\mathsf G_2 \times \SO_{2l+1}, \mathop{\FF^7} \limits_1{\!} \oplus \mathop{\FF^{2l+1}} \limits_2{\!})$ then the pair $(M, \mathfrak g/(\mathfrak p_I^- + \mathfrak h))$ is equivalent to
\[
(\SL_2 \times \GL_{l} \times \FF^\times, [\mathop{\vphantom|\FF^{2}} \limits_1{\!} \otimes \mathop{\vphantom|\FF^l} \limits_2{\!}]_{\chi} \oplus [\mathop{\vphantom|\FF^{l}} \limits_2{\!}]_{2\chi} \oplus \FF^1_{2\chi}),
\]
the latter module being spherical if and only if $l \le 1$.
The latter also implies that none of the two varieties $\SOGr_{\max}^\pm(V)$ is $H$-spherical if both pairs $(H_1,V_1)$ and $(H_2,V_2)$ are BF-equivalent to $(\mathsf G_2, \FF^7)$.

Part~2: $p,q$ are even.

By Proposition~\ref{prop_so_aux}(\ref{prop_so_aux_b}), a necessary $H$-sphericity condition for $\SOGr_{\max}^{\pm}(V)$ is that for any $i = 1,2$ with $\dim V_i > 2$ the group $H$ acts spherically on both varieties $\SOGr_{\max}^\pm(V_i)$.
Then Propositions~\ref{prop_so_Qmax_even_V_irr} and~\ref{prop_so_Qmax_V_wr} imply that for $i = 1,2$ the pair $(H_i,V_i)$ can be BF-equivalent to one of $(\SO_{2m}, \FF^{2m})$ ($m \ge 1$), $(\Sp_4 \times \SL_2, \FF^4 \otimes \FF^2)$, $(\Spin_7, \FF^8)$, $(\GL_m, \Omega(\FF^m))$ $(m \ge 2)$, $(\SL_{2m+1}, \Omega(\FF^{2m+1}))$ ($m \ge 1$), $(\SO_3 \times \FF^\times, \Omega([\FF^3]_\chi))$.

In Cases \ref{caseno_1}--\ref{caseno_6} below we assume that the pair $(H_2, V_2)$ is BF-equivalent to $(\SO_{2l}, \FF^{2l})$ with $l \ge 1$ and consider the various possibilities for the pair $(H_1,V_1)$ up to BF-equivalence.
By Lemma~\ref{lemma_outer_aut}(\ref{lemma_outer_aut_a}), for the situation $H_0 = H_1 \times H_2$ in all these cases it suffices to check $H$-sphericity only for $\SOGr_{\max}^+(V)$.

\setcounter{casee}{0}

\textit{Case}~\casenoe: $(\SO_{2m},\FF^{2m})$, $m \ge 1$.
\label{caseno_1}

If $H_0 = H_1 \times H_2$ then the pair $(M, \mathfrak g/(\mathfrak p_I^- + \mathfrak h))$ is equivalent to
$
(\GL_{m} \times \GL_{l}, \FF^{m} \otimes \FF^{l}),
$
the latter module being spherical.

If the pair $(H,V)$ is BF-equivalent to $(\SO_{2l}, \FF^{2l} \oplus \FF^{2l})$, $l \ge 2$, then $\dim B_{H} = l^2 < 2l^2-l = \dim \SOGr_{\max}^\pm(V)$, hence both varieties $\SOGr_{\max}^\pm(V)$ are not $H$-spherical.

If the pair $(H,V)$ is BF-equivalent to $(\SL_2 \times \SL_2 \times \SL_2, \mathop{\FF^2} \limits_1{\!} \otimes \mathop{\FF^2} \limits_3{\!} \oplus \mathop{\FF^2} \limits_2{\!} \otimes \mathop{\FF^2} \limits_3{\!})$ and $H_0 \subset \Spin_7^+$ then:

\begin{itemize}
\item
for $\SOGr_{\max}^+(V)$ the pair $(M, \mathfrak g/(\mathfrak p_I^- + \mathfrak h))$ is equivalent to
\[
(\SL_2 \times \FF^\times \times \FF^\times, [\FF^2]_{\chi_1} \oplus [\FF^2]_{\chi_2}),
\]
the latter module being spherical;

\item
for $\SOGr_{\max}^-(V)$ the pair $(M, \mathfrak g/(\mathfrak p_I^- + \mathfrak h))$ is equivalent to
\[
(\SL_2 \times \FF^\times \times \FF^\times, [\mathrm S^2 \FF^2]_{\chi_1+\chi_2} \oplus \FF^1_{\chi_1+\chi_2}),
\]
the latter module being not spherical.
\end{itemize}

\textit{Case}~\casenoe: $(\Sp_4 \times \SL_2, \FF^4 \otimes \FF^2)$.
By Lemma~\ref{lemma_outer_aut}(\ref{lemma_outer_aut_b}), it suffices to assume that $H_1 \subset \SO(V_1)$ is the default embedding, see \S\,\ref{subsec_so_prelim_remarks}.

If $H_0 = H_1 \times H_2$ then the pair $(M, \mathfrak g/(\mathfrak p_I^- + \mathfrak h))$ is equivalent to
\[
(\GL_{q/2} \times \Sp_4 \times \FF^\times, [\mathop{\vphantom|\FF^{q/2}} \limits_1{\!} \otimes \mathop{\vphantom|\FF^4} \limits_2{\!}]_\chi \oplus [\mathop{\vphantom|\wedge^2_0 \FF^4} \limits_2{\!}]_{2\chi}),
\]
the latter module being not spherical.

It also follows that both varieties $\SOGr_{\max}^\pm(V)$ are not $H$-spherical if $H_0$ is a proper subgroup of $H_1 \times H_2$.

\textit{Case}~\casenoe: $(\Spin_7,\FF^8)$.
By Lemma~\ref{lemma_outer_aut}(\ref{lemma_outer_aut_b}), it suffices to do the computations only for $\Spin_7^+$.

The only possibility is $H_0 = H_1 \times H_2$, in which case the pair $(M, \mathfrak g/(\mathfrak p_I^- + \mathfrak h))$ is equivalent to
\[
(\SL_3 \times \GL_{l} \times \FF^\times, [\mathop{\vphantom|\FF^{3}} \limits_1{\!} \otimes \mathop{\vphantom|\FF^l} \limits_2{\!}]_\chi \oplus [\mathop{\vphantom|\FF^l} \limits_2{\!}]_{3\chi}),
\]
the latter module being spherical.

\textit{Case}~\casenoe: $(\GL_m,\Omega(\FF^m))$, $m \ge 2$.
\label{caseno_4}

If $H_0 = H_1 \times H_2$ then the pair $(M, \mathfrak g/(\mathfrak p_I^- + \mathfrak h))$ is equivalent to
\[
(\GL_m \times \GL_{l}, \mathop{\FF^m} \limits_1{\!} \otimes \mathop{\FF^{l}} \limits_2{\!} \oplus \mathop{\wedge^2 \FF^m} \limits_1{\!}),
\]
the latter module being spherical if and only if $m \le 3$ or $l = 1$.

If the pair $(H,V)$ is BF-equivalent to $(\SL_m \times \FF^\times, \mathop{\Omega([\FF^m]_{a\chi})} \limits_1{\!} \oplus \Omega(\FF^1_{b\chi}))$, $m \ge 2$, $a,b \in \ZZ \setminus \lbrace 0 \rbrace$, then:

\begin{itemize}
\item
for $\SOGr_{\max}^+(V)$ the pair $(M, \mathfrak g/(\mathfrak p_I^- + \mathfrak h))$ is equivalent to
\[
(\SL_m \times \FF^\times, [\mathop{\vphantom|\FF^m} \limits_1{\!}]_{(a+b)\chi} \oplus [\mathop{\vphantom|\wedge^2 \FF^m} \limits_1{\!}]_{2a\chi}),
\]
the latter module being spherical if and only if either $m$ is even or $m$ is odd and $b \ne - ma$;

\item
for $\SOGr_{\max}^-(V)$ the pair $(M, \mathfrak g/(\mathfrak p_I^- + \mathfrak h))$ is equivalent to
\[
(\SL_m \times \FF^\times, [\mathop{\vphantom|\FF^m} \limits_1{\!}]_{(a-b)\chi} \oplus [\mathop{\vphantom|\wedge^2 \FF^m} \limits_1{\!}]_{2a\chi}),
\]
the latter module being spherical if and only if either $m$ is even or $m$ is odd and $b \ne ma$.
\end{itemize}

If the pair $(H,V)$ is BF-equivalent to $(\SL_2 \times \SL_2 \times \FF^\times, \mathop{\Omega([\FF^2]_\chi)} \limits_1{\!} \oplus \mathop{\vphantom|\FF^2} \limits_1{\!} \otimes \mathop{\vphantom|\FF^2} \limits_2{\!})$ then $\dim B_{H} = 5 < 6 = \dim \SOGr_{\max}^\pm(V)$, hence both $\SOGr_{\max}^\pm(V)$ are not $H$-spherical in this case.

If the pair $(H,V)$ is BF-equivalent to $(\SL_4 \times \FF^\times, \mathop{\Omega([\FF^4]_\chi)} \limits_1{\!} \oplus \mathop{\vphantom|\wedge^2 \FF^4} \limits_1{\!})$ then
$\dim B_{H} = 10 < 21 = \dim \SOGr_{\max}^\pm(V)$, hence both $\SOGr_{\max}^\pm(V)$ are not $H$-spherical in this case.

\textit{Case}~\casenoe: $(\SL_{2m+1}, \Omega(\FF^{2m+1}))$, $m \ge 1$.

If $H_0 = H_1 \times H_2$ then the pair $(M, \mathfrak g/(\mathfrak p_I^- + \mathfrak h))$ is equivalent to
\[
(\SL_{2m+1} \times \GL_{l}, \mathop{\FF^{2m+1}} \limits_1{\!} \otimes \mathop{\FF^{l}} \limits_2{\!} \oplus \mathop{\wedge^2 \FF^{2m+1}} \limits_1{\!}),
\]
the latter module being spherical if and only if $l = 1$.
This shows in particular that both $\SOGr_{\max}^\pm(V)$ are not $H$-spherical when $H_0$ is a proper subgroup of $H_1 \times H_2$.

\textit{Case}~\casenoe: $(\SO_3 \times \FF^\times, \Omega([\FF^3]_\chi))$.
\label{caseno_6}

If $H_0 = H_1 \times H_2$ then the pair $(M, \mathfrak g/(\mathfrak p_I^- + \mathfrak h))$ is equivalent to
\[
(\GL_{l} \times \SO_3 \times \FF^\times, [\mathop{\vphantom|\FF^{l}} \limits_1{\!} \otimes \mathop{\vphantom|\FF^3} \limits_2{\!}]_\chi \oplus [\mathop{\vphantom|\FF^3} \limits_2{\!}]_{2\chi}),
\]
the latter module being not spherical.

It also follows that both varieties $\SOGr_{\max}^\pm(V)$ are not $H$-spherical if $H_0$ is a proper subgroup of $H_1 \times H_2$.

The results obtained in Cases~\ref{caseno_1}--\ref{caseno_6} show that it now suffices to consider only the cases where for $i=1,2$ the pair $(H_i,V_i)$ is BF-equivalent to either $(\Spin_7, \FF^8)$ or $(\GL_m, \Omega(\FF^m))$ ($m \ge 2$).
Next we consider the remaining cases for the pairs $(H_1,V_1)$ and $(H_2,V_2)$ up to BF-equivalence.

\textit{Case}~\casenoe: $(\Spin_7, \FF^8)$, $(\Spin_7, \FF^8)$.
\label{caseno_7}

If $H_0 = H_1 \times H_2$ then $\dim B_{H} = 24 < 28 = \dim \SOGr_{\max}^\pm(\FF^8 \oplus \FF^8)$, hence both varieties $\SOGr_{\max}^\pm(V)$ are not $H$-spherical.
Clearly, the latter also holds if $H_0$ is a proper subgroup of~$H_1 \times H_2$.

\textit{Case}~\casenoe: $(\Spin_7, \FF^8)$, $(\GL_m, \Omega(\FF^m))$, $m \ge 2$.

Up to automorphism, it suffices to do the computations only for $\Spin_7^+$.
If $\SOGr_{\max}^*(V)$ is $H$-spherical for some $* \in \lbrace +,- \rbrace$ then $\SOGr_{\max}^*(V)$ would be spherical for $\SO_8 \times \GL_m$, which implies $m \le 3$ by the results in Case~\ref{caseno_4}.

If $m = 3$ and $H_0 = H_1 \times H_2$ then $\dim B_{H} = 18 < 21 = \dim \SOGr_{\max}^\pm(V)$, hence both $\SOGr_{\max}^\pm(V)$ are not $H$-spherical.
Clearly, the latter also holds if $m = 3$ and $H_0$ is a proper subgroup of $H_1 \times H_2$.

If $m = 2$ and $H_0 = H_1 \times H_2$ then:

\begin{itemize}
\item
for $\SOGr_{\max}^+(V)$ the pair $(M, \mathfrak g/(\mathfrak p_I^- + \mathfrak h))$ is equivalent to
\[
(\SL_3 \times \SL_2 \times \FF^\times \times \FF^\times, [\mathop{\vphantom|\FF^3} \limits_1{\!} \otimes \mathop{\vphantom|\FF^2} \limits_2{\!}]_{\chi_1 + \chi_2} \oplus [\mathop{\vphantom|\FF^2} \limits_2{\!}]_{3\chi_1 + \chi_2} \oplus \FF^1_{2\chi_2}),
\]
the latter module being spherical;

\item
for $\SOGr_{\max}^-(V)$ the pair $(M, \mathfrak g/(\mathfrak p_I^- + \mathfrak h))$ is equivalent to
\[
(\Sp_4 \times \SL_2 \times \FF^\times \times \FF^\times, [\mathop{\vphantom|\FF^4} \limits_1{\!} \otimes \mathop{\vphantom|\FF^2} \limits_2{\!}]_{\chi_1 + \chi_2} \oplus \FF^1_{2\chi_1} \oplus \FF^1_{2\chi_2}),
\]
the latter module being not spherical.
\end{itemize}

\textit{Case}~\casenoe: $(\GL_m, \FF^m)$, $(\GL_l, \FF^l)$, $m,l \ge 2$.

If $H_0 = H_1 \times H_2$ then:

\begin{itemize}
\item
for $\SOGr_{\max}^+(V)$ the pair $(M, \mathfrak g/(\mathfrak p_I^- + \mathfrak h))$ is equivalent to
\[
(\GL_m \times \GL_l, \mathop{\FF^m} \limits_1{\!} \otimes \mathop{\FF^l} \limits_2{\!} \oplus \mathop{\wedge^2 \FF^m} \limits_1{\!} \oplus \mathop{\wedge^2 \FF^l} \limits_2{\!}),
\]
the latter module being not spherical;

\item
for $\SOGr_{\max}^-(V)$ the pair $(M, \mathfrak g/(\mathfrak p_I^- + \mathfrak h))$ is equivalent to
\[
(\GL_m \times \GL_{l-1}, \mathop{\FF^m} \limits_1{\!} \oplus \mathop{\FF^m} \limits_1{\!} \otimes \mathop{\FF^{l-1}} \limits_2{\!} \oplus \mathop{\wedge^2 \FF^m} \limits_1{\!} \oplus \mathop{\wedge^2 \FF^{l-1}} \limits_2{\!}),
\]
the latter module being not spherical.
\end{itemize}

It also follows that both varieties $\SOGr_{\max}^\pm(V)$ are not $H$-spherical if $H_0$ is a proper subgroup of $H_1 \times H_2$.
\end{proof}

\subsection{Spherical actions on \texorpdfstring{$\SOGr_{\max}^{(\pm)}(V)$}{SOGr\_max\textasciicircum(\textpm)} for \texorpdfstring{$r \ge 3$}{r>=3}}
\label{subsec_SOGr_max_r>=3}

\begin{proposition} \label{prop_rge3_odd}
Suppose that $r \ge 3$ and $d$ is odd.
Then $\SOGr_{\max}(V)$ is not $H$-spherical.
\end{proposition}

\begin{proof}
As $d$ is odd, $\SOGr_{\max}(V)$ being $H$-spherical implies that $\SOGr_1(V)$ is $H$-spherical by Proposition~\ref{prop_so_odd_Q1_is_minimal} and Theorem~\ref{thm_sph_descent}.
Then Corollary~\ref{crl_no_3_summands} yields $r \le 2$.
\end{proof}

In what follows we assume that $d$ is even.

\begin{proposition} \label{prop_so_even_3summands}
Suppose that $V = \widetilde V_1 \oplus \widetilde V_2 \oplus \widetilde V_3$ for three pairwise orthogonal nonzero $H$-submodules $\widetilde V_1, \widetilde V_2, \widetilde V_3 \subset V$ of dimensions $n_1, n_2, n_3$, respectively.
If $H$ acts spherically on $\SOGr_{\max}^*(V)$ for some $* \in \lbrace +,- \rbrace$ then one of the following possibilities holds:
\begin{enumerate}[label=\textup{(\arabic*)},ref=\textup{\arabic*}]
\item
$\min(n_1,n_2,n_3) = 1$;

\item
at least two of $n_1,n_2,n_3$ equal~$2$.
\end{enumerate}
\end{proposition}

\begin{proof}
It suffices to prove the assertion for the case $H_0 = \SO(\widetilde V_1) \times \SO(\widetilde V_2) \times \SO(\widetilde V_3)$.
Then by Lemma~\ref{lemma_outer_aut}(\ref{lemma_outer_aut_a}) it suffices to consider the case $* = +$.
Without loss of generality we may assume that $n_1$ is even.

\setcounter{casef}{0}

\textit{Case}~\casenof: $n_2,n_3$ are odd. We may assume $n_2 \ge n_3 \ge 1$. The pair $(M, \mathfrak g/(\mathfrak p_I^- + \mathfrak h))$ is equivalent to
\[
(\GL_{n_1/2} \times \GL_{[n_2/2]} \times \GL_{[n_3/2]}, \mathop{\FF^{n_1/2}} \limits_1{\!} \oplus \mathop{\FF^{n_1/2}} \limits_1{\!} \otimes \mathop{\FF^{[n_2/2]}} \limits_2{\!} \oplus \mathop{\FF^{n_1/2}} \limits_1{\!} \otimes \mathop{\FF^{[n_3/2]}} \limits_3{\!} \oplus \mathop{\FF^{[n_2/2]}} \limits_2{\!} \otimes \mathop{\FF^{[n_3/2]}} \limits_3{\!}).
\]
If $n_3 = 1$ then the above module is spherical.
If $n_3 \ge 3$ then the submodule consisting of all summands except the first one is not spherical by Lemma~\ref{lemma_VixVj_spherical}, hence the whole module is not spherical.

\textit{Case}~\casenof: $n_2,n_3$ are even.
The pair $(M, \mathfrak g/(\mathfrak p_I^- + \mathfrak h))$ is equivalent to
\[
(\GL_{n_1/2} \times \GL_{n_2/2} \times \GL_{n_3/2}, \mathop{\FF^{n_1/2}} \limits_1{\!} \otimes \mathop{\FF^{n_2/2}} \limits_2{\!} \oplus \mathop{\FF^{n_1/2}} \limits_1{\!} \otimes \mathop{\FF^{n_3/2}} \limits_3{\!} \oplus \mathop{\FF^{n_2/2}} \limits_2{\!} \otimes \mathop{\FF^{n_3/2}} \limits_3{\!}).
\]
By Lemma~\ref{lemma_VixVj_spherical}, the latter module is spherical if and only if at least two of $n_1,n_2,n_3$ equal~$2$.
\end{proof}

According to Proposition~\ref{prop_so_even_3summands}, the analysis of the case $r = 3$ is completed by Propositions~\ref{prop_so_even_r=3_1} and~\ref{prop_so_even_r=3_2} below.

\begin{proposition} \label{prop_so_even_r=3_1}
Suppose that $r = 3$, $d_1 \ge d_2 \ge d_3 = 1$, and $* \in \lbrace +, - \rbrace$.
Then $\SOGr_{\max}^*(V)$ is $H$-spherical if and only if the pair $(H,V)$ is BF-equivalent to one of $(\SO_{2l} \times \SO_{2m+1}, \FF^{2l} \oplus \FF^{2m+1} \oplus \FF^1)$ $(l \ge 1, m \ge 0)$, $(\Spin_7, \FF^8 \oplus \FF^1 \oplus \FF^1)$, $(\GL_n, \Omega(\FF^n) \oplus\nobreak \FF^1 \oplus\nobreak \FF^1)$ $(n \ge 2)$.

\end{proposition}

\begin{proof}
By Proposition~\ref{prop_so_2n_to_2n-1}, $\SOGr_{\max}^*(V)$ is $H$-spherical if and only if $\SOGr_{\max}(V_1 \oplus V_2)$ is $H$-spherical.
Then the claim follows from Proposition~\ref{prop_so_odd_Qmax_sr}.
\end{proof}

\begin{proposition} \label{prop_so_even_r=3_2}
Suppose that $r = 3$, $d_2 = d_3 = 2$, and $* \in \lbrace +, - \rbrace$.
Then $\SOGr_{\max}^*(V)$ is $H$-spherical if and only if the pair $(H,V)$ is BF-equivalent to $(\SO_{2l} \times \FF^\times \times \FF^\times$, $\FF^{2l} \oplus \Omega(\FF^1_{\chi_1}) \oplus \Omega(\FF^1_{\chi_2}))$.
\end{proposition}

\begin{proof}
It follows from the hypothesis that for $i = 2,3$ the pair $(H_i,V_i)$ is BF-equivalent to $(\FF^\times, \Omega(\FF_\chi))$.
For $i=2,3$ let $V_i = V'_i \oplus V''_i$ be a decomposition of $V_i$ into a direct sum of two $H$-stable isotropic lines.
Let $\widetilde H$ be the subgroup of $\SO(V_1) \times \SO(V_2 \oplus V_3)$ stabilizing both subspaces $V'_2 \oplus V'_3$ and $V''_2 \oplus V''_3$.
Then the pair $(\widetilde H, V)$ is BF-equivalent to $(\SO_{d_1} \times \GL_2, \mathop{\FF^{d_1}} \limits_1{\!} \oplus \Omega(\mathop{\FF^2} \limits_2{\!}))$.

If $\SOGr_{\max}^*(V)$ is $H$-spherical then it is $\widetilde H$-spherical, hence by Proposition~\ref{prop_so_even_Qmax_V1+V2} the pair $(H_1,V_1)$ should be BF-equivalent to one of $(\SO_{2l}, \FF^{2l})$ ($l \ge 1$) or $(\Spin_7, \FF^8)$.
In what follows we treat these two cases separately.

\setcounter{caseg}{0}

\textit{Case}~\casenog: $(\SO_{2l}, \FF^{2l})$, $l \ge 1$.

If $H_0 = H_1 \times H_2 \times H_3$ then the pair $(M, \mathfrak g/(\mathfrak p_I^- + \mathfrak h))$ is equivalent to
\[
(\GL_{l} \times \FF^\times \times \FF^\times, [\mathop{\vphantom|\FF^{l}} \limits_1{\!}]_{\chi_1} \oplus [\mathop{\vphantom|\FF^{l}} \limits_1{\!}]_{\chi_2} \oplus \FF^1_{\chi_1 + \chi_2}),
\]
the latter module being spherical.

If $l = 1$ and $H_0$ is a proper subgroup of $H_1 \times H_2 \times H_3$ then $\dim B_H \le 2 < 3 = \dim \SOGr_{\max}^*(V)$, hence $\SOGr_{\max}^*(V)$ is not $H$-spherical.

If $l \ge 2$ and the pair $(H,V)$ is equivalent to $(\SO_{2l} \times \FF^\times, \FF^{2l} \oplus \Omega(\FF^1_{a\chi}) \oplus \Omega(\FF^1_{b\chi})$ for some $a,b \in \ZZ \setminus \lbrace 0 \rbrace$ then the pair  $(M, \mathfrak g/(\mathfrak p_I^- + \mathfrak h))$ is BF-equivalent to
\[
(\GL_{l} \times \FF^\times \times \FF^\times, [\mathop{\vphantom|\FF^{l}} \limits_1{\!}]_{a\chi} \oplus [\mathop{\vphantom|\FF^{l}} \limits_1{\!}]_{b\chi} \oplus \FF^1_{(a+b)\chi}),
\]
the latter module being not spherical.

\textit{Case}~\casenog: $(\Spin_7, \FF^8)$.
If $H_0 = H_1 \times H_2 \times H_3$ then $\dim B_{H} = 14 < 15 = \dim \SOGr_{\max}^*(V)$, hence $\SOGr_{\max}^*(V)$ is not $H$-spherical.
It also follows that $\SOGr_{\max}^*(V)$ is not $H$-spherical if $H$ is a proper subgroup of $H_1 \times H_2 \times H_3$.
\end{proof}

\begin{proposition}
Suppose that $V = \widetilde V_1 \oplus \widetilde V_2 \oplus \widetilde V_3 \oplus \widetilde V_4$ for four pairwise orthogonal nonzero $H$-submodules $\widetilde V_1, \widetilde V_2, \widetilde V_3, \widetilde V_4 \subset V$.
Then both varieties $\SOGr_{\max}^\pm(V)$ are not $H$-spherical.
\end{proposition}

\begin{proof}
It suffices to prove the assertion for $H = \SO(\widetilde V_1) \times \SO(\widetilde V_2) \times \SO(\widetilde V_3) \times \SO(\widetilde V_4)$, which is assumed in what follows.
Then by Lemma~\ref{lemma_outer_aut}(\ref{lemma_outer_aut_a}) it suffices to consider $\SOGr_{\max}^+(V)$ only.
Assume that $\SOGr_{\max}^+(V)$ is $H$-spherical and put $n_i = \dim \widetilde V_i$ for $i = 1,2,3,4$.

If one of the numbers $n_i$ is odd, say $n_4$, then $H$ acts spherically on $\SOGr_{\max}(\widetilde V_1 \oplus \widetilde V_2 \oplus \widetilde V_3)$ by Proposition~\ref{prop_so_aux}(\ref{prop_so_aux_a}), which is impossible by Proposition~\ref{prop_rge3_odd}.
Thus $n_i$ are even for all $i=1,2,3,4$.

Clearly, the group
$
\SO(\widetilde V_1) \times \SO(\widetilde V_2) \times \SO(\widetilde V_3 \oplus \widetilde V_4)
$
also acts spherically on $\SOGr_{\max}^+(V)$.
As $n_3 + n_4 \ge 4$, Proposition~\ref{prop_so_even_3summands} yields $n_1 = n_2 = 2$.
Likewise, $n_3=n_4 = 2$.
But then $\dim B_H = 4 < 6 = \dim \SOGr_{\max}^+(V)$, hence $\SOGr_{\max}^+(V)$ is not $H$-spherical.
\end{proof}

\begin{corollary}
If $r \ge 4$ then both varieties $\SOGr_{\max}^\pm(V)$ are not $H$-spherical.
\end{corollary}

\subsection{Spherical actions on \texorpdfstring{$\SOGr_2(V)$}{SOGr\_2(V)}}
\label{subsec_SOGr_2}

As $\SOGr_2(V) = \SOGr_{\max}^+(V) \cup \SOGr_{\max}^-(V)$ for $d = 4$ and $\SOGr_2(V) = \SOGr_{\max}(V)$ for $d = 5$, for classifying $H$-spherical actions on $\SOGr_2(V)$ it suffices to assume $d \ge 6$. (Although the result of Proposition~\ref{prop_so_Q2_V_irr} will be needed later for $d \ge 4$).
Recall from \S\,\ref{subsec_FV_via_compositions} that $\SOGr_2(V) \simeq X_I$ with $I = \lbrace 2 \rbrace$ for $d \ge 5$.

\begin{proposition} \label{prop_so_Q2_V_irr}
Suppose that $V$ is a simple $H$-module and $d \ge 4$.
Then
\begin{enumerate}[label=\textup{(\alph*)},ref=\textup{\alph*}]
\item \label{prop_so_Q2_V_irr_a}
if $d \ge 5$ then $\SOGr_2(V)$ is $H$-spherical if and only if the pair $(H,V)$ is BF-equivalent to one of $(\SO_d, \FF^d)$, $(\mathsf G_2, \FF^7)$, $(\Spin_7, \FF^8)$;

\item \label{prop_so_Q2_V_irr_b}
if $d = 4$ and $* \in \lbrace +, - \rbrace$ then $\SOGr_2^*(V)$ is $H$-spherical if and only if the pair $(H,V)$ is BF-equivalent to $(\SO_4, \FF^4)$.
\end{enumerate}
\end{proposition}

\begin{proof}
(\ref{prop_so_Q2_V_irr_a})
If $H$ acts spherically on $\SOGr_2(V)$ then it acts spherically on $\SOGr_1(V)$ by Propositions~\ref{prop_so_odd_Q1_is_minimal} and~\ref{prop_so_even_minimal} and Theorem~\ref{thm_sph_descent}, hence Proposition~\ref{prop_so_Q1_V_irr} leaves us with the following cases (up to BF-equivalence).

\setcounter{caseh}{0}

\textit{Case}~\casenoh: $H = \SO_n$, $V = \FF^n$, $n \ge 6$. In this case $\SOGr_2(V)$ is spherical.

\textit{Case}~\casenoh: $H = \Sp_{2n} \times \SL_2$, $V = \FF^{2n} \otimes \FF^2$, $n \ge 2$.
The pair $(M, \mathfrak g/(\mathfrak p_I^- + \mathfrak h))$ is equivalent to
\[
(\Sp_{2n-4} \times \SL_2 \times \FF^\times \times \FF^\times, [\mathop{\vphantom|\FF^{2n-4}} \limits_1{\!} \otimes \mathop{\vphantom|\FF^2} \limits_2{\!}]_{\chi_1 + 2\chi_2} \oplus [\mathop{\vphantom|\mathrm{S}^2 \FF^2} \limits_2{\!}]_{2\chi_2} \oplus \FF^1_{2\chi_1} \oplus \FF^1_{2\chi_1+2\chi_2}),
\]
the latter module being not spherical.

\textit{Case}~\casenoh: $H= \mathsf G_2$, $V = \FF^7$.
The pair $(M, \mathfrak g/(\mathfrak p_I^- + \mathfrak h))$ is equivalent to $(\GL_2, \FF^2)$, the latter module being spherical.

\textit{Case}~\casenoh: $H = \Spin_7$, $V = \FF^8$.
In this case $\SOGr_2(V)$ is spherical.

\textit{Case}~\casenoh: $H = \Spin_9$, $V = \FF^{16}$.
We have $\dim B_{H} = 20 < 25 = \dim \SOGr_2(V)$, hence $\SOGr_2(V)$ is not $H$-spherical.

(\ref{prop_so_Q2_V_irr_b})
This follows directly from Proposition~\ref{prop_so_Qmax_even_V_irr}.
\end{proof}

\begin{proposition} \label{prop_so_Q2_V_wr}
Suppose that $V$ is BF-equivalent to $\Omega(W)$ for a simple $H$-module~$W$ with $\dim W \ge 3$.
Then $\SOGr_2(V)$ is $H$-spherical if and only if the pair $(H,W)$ is equivalent to one of $(\GL_n, \FF^n)$ $(n \ge 3)$ or $(\SL_n, \FF^n)$ $(n \ge 3, n \ne 4)$.
\end{proposition}

\begin{proof}
If $\SOGr_2(V)$ is $H$-spherical then $\SOGr_1(V)$ is $H$-spherical by Proposition~\ref{prop_so_even_minimal} and Theorem~\ref{thm_sph_descent}.
Then Proposition~\ref{prop_so_Q1_V_wr} leaves us with the following cases (up to equivalence).

\setcounter{casei}{0}

\textit{Case}~\casenoi: $H = \GL_n$, $W = \FF^n$, $n \ge 3$.
The pair $(M, \mathfrak g/(\mathfrak p_I^- + \mathfrak h))$ is equivalent to
\[
(\GL_2 \times \GL_{n-2}, \mathop{\FF^2} \limits_1{\!} \otimes \mathop{\FF^{n-2}} \limits_2{\!} \oplus \mathop{\wedge^2 \FF^2} \limits_1{\!}),
\]
the latter module being spherical.

\textit{Case}~\casenoi: $H = \SL_n$, $W = \FF^n$, $n \ge 3$.
The pair $(M, \mathfrak g/(\mathfrak p_I^- + \mathfrak h))$ is equivalent to
\[
(\SL_2 \times \SL_{n-2} \times \FF^\times, [\mathop{\vphantom|\FF^2} \limits_1{\!} \otimes \mathop{\vphantom|\FF^{n-2}} \limits_2{\!}]_{(n-4)\chi} \oplus \FF^1_{(2n-4)\chi}),
\]
the latter module being spherical if and only if $n \ne 4$.

\textit{Case}~\casenoi: $H = \Sp_{2n} \times \FF^\times$, $W = [\FF^{2n}]_\chi$, $n \ge 2$.
The pair $(M, \mathfrak g/(\mathfrak p_I^- + \mathfrak h))$ is equivalent to
\[
(\SL_2 \times \Sp_{2n-4} \times \FF^\times \times \FF^\times, [\mathop{\vphantom|\FF^2} \limits_1{\!} \otimes \mathop{\vphantom|\FF^{2n-4}} \limits_2{\!}]_{\chi_1+2\chi_2} \oplus [\mathop{\vphantom|\mathrm S^2 \FF^2} \limits_1{\!}]_{2\chi_2} \oplus \FF^1_{2\chi_1} \oplus \FF^1_{2\chi_2} \oplus \FF^1_{2\chi_1+2\chi_2}),
\]
the latter module being not spherical.
\end{proof}

To proceed with the case $r \ge 2$, we shall need the following lemma.

\begin{lemma} \label{lemma_so_Gr2}
Suppose that $V = \widetilde V_1 \oplus \widetilde V_2$ for two pairwise orthogonal nonzero $H$-submodules $\widetilde V_1, \widetilde V_2 \subset V$, $\dim \widetilde V_1 \ge 4$, and $\SOGr_2(V)$ is $H$-spherical.
Then
\begin{enumerate}[label=\textup{(\alph*)},ref=\textup{\alph*}]
\item
if $\dim \widetilde V_1 \ge 5$ then $\SOGr_2(\widetilde V_1)$ is $H$-spherical;

\item
if $\dim \widetilde V_1 = 4$ then both varieties $\SOGr_2^\pm(\widetilde V_1)$ are $H$-spherical.
\end{enumerate}
\end{lemma}

\begin{proof}
Put $Y = \SOGr_2(\widetilde V_1)$ if $\dim \widetilde V_1 \ge 5$ and choose $Y \in \lbrace \SOGr_2^+(\widetilde V_1), \SOGr_2^-(\widetilde V_1) \rbrace$ if $\dim \widetilde V_1 = 4$.
Choose a two-dimensional isotropic subspace $U \subset \widetilde V_1$ such that $U \in \nobreak Y$.
Clearly, the $(\SO(\widetilde V_1)\times \SO(\widetilde V_2))$-orbit of $U$ in $\SOGr_2(V)$ is closed and isomorphic to~$Y$.
Then $Y$ is $H$-spherical by Theorem~\ref{thm_subvar_sph}.
\end{proof}

\begin{proposition} \label{prop_so_Gr2_restr}
Suppose that $d \ge 5$, $V = \widetilde V_1 \oplus \widetilde V_2$ for two pairwise orthogonal nonzero $H$-submodules $\widetilde V_1, \widetilde V_2 \subset V$ of dimensions $n_1,n_2$, respectively, and $\SOGr_2(V)$ is $H$-spherical.
Then $\min(n_1,n_2) \le 2$.
\end{proposition}

\begin{proof}
It suffices to prove that $\SOGr_2(V)$ is not $H$-spherical when $H_0 = \SO(\widetilde V_1) \times \SO(\widetilde V_2)$ and $n_1 \ge n_2 \ge 3$, which is assumed in what follows.

Choose realizations $\widetilde V_1 = \langle e_1, e_d \rangle \oplus W_1$, $\widetilde V_2  = \langle e_2, e_{d-1} \rangle \oplus W_2$ where $W_1,W_2$ are nondegenerate subspaces such that $W_1 \oplus W_2 = \langle e_3, e_4, \ldots, e_{d-2} \rangle$.
With these realizations, the pair $(M, \mathfrak g/(\mathfrak p_I^- + \mathfrak h))$ is equivalent to
\[
(\SO_{n_1-2} \times \SO_{n_2-2} \times \FF^\times \times \FF^\times, [\mathop{\vphantom|\FF^{n_1-2}} \limits_1{\!}]_{\chi_1} \oplus [\mathop{\vphantom|\FF^{n_2-2}} \limits_2{\!}]_{\chi_2} \oplus \FF^1_{\chi_1 + \chi_2}),
\]
the latter module being not spherical.
\end{proof}

\begin{proposition} \label{prop_so_Q2_V_sr}
Suppose that $r \ge 2$ and $d \ge 6$.
Then $\SOGr_2(V)$ is $H$-spherical if and only if $r = 2$ and the pair $(H,V)$ is BF-equivalent to one of $(\SO_n, \FF^n \oplus \FF^1)$ $(n \ge 5)$, $(\SO_n \times \FF^\times, \FF^n \oplus \Omega(\FF^1_\chi))$ $(n \ge 4)$, $(\Spin_7, \FF^8 \oplus \FF^1)$ or $(\Spin_7 \times \FF^\times, \FF^8 \oplus \Omega(\FF^1_\chi))$.
\end{proposition}

\begin{proof}
If $H$ acts spherically on $\SOGr_2(V)$ then it acts spherically on $\SOGr_1(V)$ by Propositions~\ref{prop_so_odd_Q1_is_minimal} and~\ref{prop_so_even_minimal} and Theorem~\ref{thm_sph_descent}.
Then Corollary~\ref{crl_no_3_summands} yields $r = 2$.
By Proposition~\ref{prop_so_Gr2_restr}, in what follows we may assume $d_2 \le 2$.

We now show that the $H$-module $V_1$ cannot be weakly reducible.
To this end, it suffices to prove that $\SOGr_2(V)$ is not $H$-spherical if the pair $(H,V)$ is BF-equivalent to $(\GL_n, \Omega(\FF^n) \oplus \FF^1)$ ($n \ge 3$) or $(\GL_n \times \FF^\times, \Omega(\FF^n) \oplus \Omega(\FF^1_\chi))$ ($n \ge 2$).

Indeed, in the first case the pair $(M, \mathfrak g/(\mathfrak p_I^- + \mathfrak h))$ is equivalent to
\[
(\GL_2 \times \GL_{n-2}, \mathop{\FF^2} \limits_1{\!} \oplus \mathop{\FF^2} \limits_1{\!} \otimes \mathop{\FF^{n-2}} \limits_2{\!} \oplus \mathop{\wedge^2 \FF^2} \limits_1{\!}),
\]
the latter module being not spherical, and in the second case the pair $(M, \mathfrak g/(\mathfrak p_I^- + \mathfrak h))$ is equivalent to
\[
(\GL_2 \times \GL_{n-2} \times \FF^\times, \mathop{\vphantom|\FF^2} \limits_1{\!} \otimes \mathop{\vphantom|\FF^{n-2}} \limits_2{\!} \oplus [\mathop{\vphantom|\FF^2} \limits_1{\!}]_\chi \oplus [\mathop{\vphantom|\FF^2} \limits_1{\!}]_{-\chi} \oplus \mathop{\vphantom|\wedge^2 \FF^2} \limits_1{\!}),
\]
the latter module being not spherical.

Thus it remains to analyze the situation where $V_1$ is a simple $H$-module.
In this situation, Lemma~\ref{lemma_so_Gr2} and Proposition~\ref{prop_so_Q2_V_irr} imply that the pair $(H_1,V_1)$ can be BF-equivalent to one of $(\SO_{d_1}, \FF^{d_1})$, $(\Spin_7,
\FF^8)$, or $(\mathsf G_2, \FF^7)$.
Below we consider all the corresponding cases for the pair $(H,V)$ up to BF-equivalence.

\setcounter{casej}{0}

\textit{Case}~\casenoj: $(\SO_{d_1}, \FF^{d_1} \oplus \FF^1)$.
The pair $(M, \mathfrak g/(\mathfrak p_I^- + \mathfrak h))$ is equivalent to $(\GL_2, \FF^2)$, the latter module being spherical.

\textit{Case}~\casenoj: $(\SO_{d_1} \times \FF^\times, \FF^{d_1} \oplus \Omega(\FF^1_\chi))$.
Choose the realizations $V_2 = \langle e_1, e_d \rangle$ and $V_1 = \langle e_2,\ldots,e_{d-1} \rangle$.
Then the pair $(M, \mathfrak g/(\mathfrak p_I^- + \mathfrak h))$ is equivalent to
\[
(\SO_{d-4}, [\mathop{\vphantom|\FF^{d-4}} \limits_1{\!}]_{\chi_1} \oplus \FF^1_{\chi_1 + \chi_2}),
\]
the latter module being spherical.

\textit{Case}~\casenoj: $(\Spin_7, \FF^8 \oplus \FF^1)$.
By Lemma~\ref{lemma_outer_aut}(\ref{lemma_outer_aut_b}), it suffices to do the calculations for~$\Spin_7^+$.
Then the pair $(M, \mathfrak g/(\mathfrak p_I^- + \mathfrak h))$ is equivalent to
\[
(\SL_2 \times \SL_2 \times \FF^\times, [\mathop{\vphantom|\FF^2} \limits_1{\!}]_\chi \oplus [\mathop{\vphantom|\FF^2}\limits_2{\!}]_\chi),
\]
the latter module being spherical.

\textit{Case}~\casenoj: $(\Spin_7 \times \FF^\times, \FF^8 \oplus \Omega(\FF^1_\chi))$.
By Lemma~\ref{lemma_outer_aut}(\ref{lemma_outer_aut_b}), it suffices to do the calculations for~$\Spin_7^+$.
Then the pair $(M, \mathfrak g/(\mathfrak p_I^- + \mathfrak h))$ is equivalent to
\[
(\SL_2 \times \SL_2 \times \FF^\times \times \FF^\times, [\mathop{\vphantom|\FF^2} \limits_1{\!}]_{\chi_1 + \chi_2} \oplus [\mathop{\vphantom|\FF^2} \limits_1{\!}]_{\chi_1 - \chi_2} \oplus [\mathop{\vphantom|\FF^2} \limits_2{\!}]_{\chi_1}),
\]
the latter module being spherical.

\textit{Case}~\casenoj: $(\mathsf G_2, \FF^7 \oplus \FF^1)$.
The pair $(M, \mathfrak g/(\mathfrak p_I^- + \mathfrak h))$ is equivalent to
\[
(\SL_2 \times \FF^\times, [\mathop{\vphantom|\FF^2} \limits_1{\!}]_\chi \oplus [\mathop{\vphantom|\FF^2} \limits_1{\!}]_\chi),
\]
the latter module being not spherical.

\textit{Case}~\casenoj: $(\mathsf G_2 \times \FF^\times, \FF^7 \oplus \Omega(\FF^1_{\chi}))$.
The pair $(M, \mathfrak g/(\mathfrak p_I^- + \mathfrak h))$ is equivalent to
\[
(\SL_2 \times \FF^\times \times \FF^\times, [\mathop{\vphantom|\FF^2} \limits_1{\!}]_{\chi_1 + \chi_2} \oplus [\mathop{\vphantom|\FF^2} \limits_1{\!}]_{\chi_1 - \chi_2} \oplus \FF^1_{\chi_1} \oplus \FF^1_{\chi_1 + 2\chi_2}),
\]
the latter module being not spherical.
\end{proof}

\subsection{Completion of the classification}
\label{subsec_orth_case_end}

Suppose that $d \ge 5$ and $X$ is a nontrivial flag variety of~$G$ different from $\SOGr_1(V)$, $\SOGr_2(V)$, $\SOGr_{\max}(V)$ (for $d$ odd), and $\SOGr_{\max}^\pm(V)$ (for $d$ even).
By Propositions~\ref{prop_so_odd_SOGr_2},~\ref{prop_so_even_SOGr_2} and Theorem~\ref{thm_sph_descent}, $X$ being $H$-spherical implies that $\SOGr_2(V)$ is $H$-spherical, hence the pair $(H,V)$ is BF-equivalent to one of those listed in Propositions~\ref{prop_so_Q2_V_irr},~\ref{prop_so_Q2_V_wr}, and~\ref{prop_so_Q2_V_sr}.
Excluding the cases where $H$ is either a symmetric subgroup of $G$ or intermediate between a Levi subgroup of $G$ and its derived subgroup (see~\S\,\ref{subsec_Levi&symmetric}), we arrive at the pairs $(\mathsf G_2, \FF^7)$, $(\Spin_7, \FF^8 \oplus \FF^1)$, $(\Spin_7 \times \FF^\times, \FF^8 \oplus \Omega(\FF^1_\chi))$, which remain to be considered.

\begin{proposition}
Suppose that the pair $(H,V)$ is BF-equivalent to $(\mathsf G_2, \FF^7)$.
Then $X_I$ is $H$-spherical if and only if $|I| \le 2$.
\end{proposition}

\begin{proof}
First suppose that $|I| = 3$.
Then the pair $(M, \mathfrak g/(\mathfrak p_I^- + \mathfrak h))$ is equivalent to
\[
(\FF^\times \times \FF^\times, \FF^1_{\chi_1} \oplus \FF^1_{\chi_2} \oplus \FF^1_{\chi_1+\chi_2}),
\]
the latter module being not spherical.

Now suppose that $|I|=2$.

If $I = \lbrace 1,2 \rbrace$ then the pair $(M, \mathfrak g/(\mathfrak p_I^- + \mathfrak h))$ is equivalent to $(\FF^\times \times \FF^\times, \FF^1_{\chi_1} \oplus \FF^1_{\chi_1 + \chi_2})$, the latter module being spherical.

If $I = \lbrace 1,3 \rbrace$ then the pair $(M, \mathfrak g/(\mathfrak p_I^- + \mathfrak h))$ is equivalent to $(\SL_2 \times \FF^\times, [\FF^2]_\chi \oplus \FF^1_{2\chi})$, the latter module being spherical.

If $I = \lbrace 2,3 \rbrace$ then the pair $(M, \mathfrak g/(\mathfrak p_I^- + \mathfrak h))$ is equivalent to $(\FF^\times \times \FF^\times, \FF^1_{\chi_1} \oplus \FF^1_{\chi_1 + \chi_2})$, the latter module being spherical.
\end{proof}

\begin{proposition}
Suppose that the pair $(H,V)$ is BF-equivalent to $(\Spin_7, \FF^8 \oplus \FF^1)$.
Then $X_I$ is $H$-spherical if and only if either $|I| = 1$ or $I = \lbrace 1,2 \rbrace$.
\end{proposition}

\begin{proof}
By Lemma~\ref{lemma_outer_aut}(\ref{lemma_outer_aut_b}), it suffices to do the calculations for~$\Spin_7^+$.

Case $I = \lbrace 3 \rbrace$.
The pair $(M, \mathfrak g/(\mathfrak p_I^- + \mathfrak h))$ is equivalent to
\[
(\SL_2 \times \FF^\times \times \FF^\times, [\FF^2]_{\chi_1} \oplus \FF^1_{2\chi_1} \oplus \FF^1_{2\chi_1 + \chi_2}),
\]
the latter module being spherical.

Case $I = \lbrace 1,2 \rbrace$.
The pair $(M, \mathfrak g/(\mathfrak p_I^- + \mathfrak h))$ is equivalent to
\[
(\SL_2 \times \FF^\times \times \FF^\times, [\FF^2]_{\chi_1 + \chi_2} \oplus \FF^1_{\chi_1} \oplus \FF^1_{\chi_1 + 2\chi_2}),
\]
the latter module being spherical.

Case $I = \lbrace 1,4 \rbrace$ or $\lbrace 3,4 \rbrace$.
As $\dim B_{H} = 12 < 13 = \dim X_I$, the variety $X_I$ is not spherical.

Case $I = \lbrace 1,3 \rbrace$ or $\lbrace 2,3 \rbrace$ or $\lbrace 2,4 \rbrace$.
As $\dim B_{H} = 12 < 14 = \dim X_I$, the variety $X_I$ is not spherical.
\end{proof}

\begin{proposition}
Suppose that $(H,V)$ is BF-equivalent to $(\Spin_7 \times \FF^\times, \FF^8 \oplus \Omega(\FF^1_\chi))$.
Then $X_I$ is $H$-spherical if and only if $I$ equals one of $\lbrace 1 \rbrace$, $\lbrace 2 \rbrace$, $\lbrace 4 \rbrace$, or $\lbrace 5 \rbrace$.
\end{proposition}

\begin{proof}
By Lemma~\ref{lemma_outer_aut}(\ref{lemma_outer_aut_b}), it suffices to do the calculations for~$\Spin_7^+$.

Case $I = \lbrace 3 \rbrace$.
As $\dim B_{H} = 13 < 15 = \dim X_I$, the variety $X_I$ is not $H$-spherical.

Case $I = \lbrace 1,2 \rbrace$ or $\lbrace 1,4 \rbrace$ or $\lbrace 1,5 \rbrace$ or $\lbrace 4,5 \rbrace$.
As $\dim B_{H} = 13 < 14 = \dim X_I$, the variety $X_I$ is not $H$-spherical.

Case $I = \lbrace 2,5 \rbrace$.
As $\dim B_H = 13 < 16 = \dim X_I$, the variety $X_I$ is not $H$-spherical.
\end{proof}

\section{Remarks on computing the restricted branching monoids}
\label{sect_RBM}

In this section we explain how to compute the restricted branching monoids corresponding to all spherical actions on flag varieties appearing in our theorems in \S\S\,\ref{subsec_sp_results},\,\ref{subsec_so_results}.

\subsection{Case~\texorpdfstring{$I = \lbrace 1 \rbrace$}{I = \{1\}}}
\label{subsec_RBM_I={1}}

\begin{proof}[Proof of Theorem~\textup{\ref{thm_RBM_sp_PV}}]
It is well known that for all $k \ge 0$ there is the $G$-module isomorphism $R_G(k\pi_1) \simeq \mathrm{S}^k V$.
Then it follows from definitions that the map
\[
\mathrm E_{H \times \FF^\times}(V^*) \to \Gamma_{\lbrace 1 \rbrace}(G,H), \quad \lambda \mapsto \overline \lambda,
\]
is an isomorphism, which immediately implies the required result.
\end{proof}

\begin{proof}[Proof of Theorem~\textup{\ref{thm_RBM_so_Q1}}]
We shall need the following well-known $G$-module isomorphism that holds for all $k \ge 0$:
\begin{equation} \label{eqn_dec_so}
\mathrm{S}^kV \simeq \bigoplus \limits_{l = 0}^{[k/2]} R_{G}((k-2l)\pi_1).
\end{equation}

(\ref{thm_RBM_so_Q1_a})
Take any $\lambda = \mu + k \delta \in \mathrm E^0_{H \times \FF^\times}(V^*) \setminus \lbrace 2\delta \rbrace$.
As $V^*$ is a spherical $(H \times \FF^\times)$-module, by~(\ref{eqn_dec_so}) there is a unique $l \in \lbrace 0, \ldots, [k/2]\rbrace$ such that $R_H(\mu)$ is a submodule of $\left. R_{G}((k-2l)\pi_1)\right|_H$.
If $l > 0$ then again by~(\ref{eqn_dec_so}) $R_H(\mu)$ is a submodule of $\left. \mathrm{S}^{k-2l} V \right|_H$ and hence $\lambda - 2l\delta \in \mathrm E_{H\times \FF^\times}(V^*)$, a contradiction.
Thus $l = 0$ and $\overline \lambda \in \Gamma_{\lbrace 1 \rbrace}(G,H)$.

Conversely, take any element $(k\pi_1; \mu) \in \Gamma_{\lbrace 1 \rbrace}(G,H)$ and put $\lambda = \mu + k \delta \in \Lambda^+(H \times \FF^\times)$ so that $(k\pi_1;\mu) = \overline \lambda$.
Then by~(\ref{eqn_dec_so}) $R_H(\mu)$ is a submodule of $\left. \mathrm{S}^k V \right|_H$, which yields $\lambda \in \mathrm E_{H \times \FF^\times}(V^*)$.
If $\lambda - 2\delta \in \mathrm E_{H \times \FF^\times}(V^*)$ then (\ref{eqn_dec_so}) implies that $R_H(\mu)$ is a submodule of $\left. \mathrm{S}^{k-2}V \right|_H$ and hence $\left. \mathrm S^k V \right|_H$ contains at least two copies of $R_H(\mu)$, which is impossible as $V^*$ is a spherical ($H \times \FF^\times$)-module.
Thus $\lambda$ is a linear combination of elements in $\mathrm E^0_{H \times \FF^\times}(V^*) \setminus \lbrace 2\delta \rbrace$.

(\ref{thm_RBM_so_Q1_b})
Take any $\lambda = \mu + k_1 \delta_1 + k_2\delta_2 \in \mathrm E^0_{H \times \FF^\times \times \FF^\times}(V^*)$.
If $\lambda = 2\delta_1$ or $\lambda = 2\delta_2$ then $\overline \lambda = (2\pi_1; 0)$.
As $\mathrm S^2 V \simeq R_G(2\pi_1) \oplus \FF^1$ as $G$-modules and each of $\mathrm S^2 V_1$ and $\mathrm S^2 V_2$ contains a trivial $G$-submodule, it follows that $R_G(2\pi_1)$ also contains a trivial submodule, hence $\overline \lambda \in \Gamma_{\lbrace 1 \rbrace}(G,H)$.
Now assume $\lambda \notin \lbrace 2\delta_1, 2\delta_2 \rbrace$.
Then $R_H(\mu)$ is a submodule of $\left.\mathrm S^{k_1+k_2} V \right|_H$.
By~(\ref{eqn_dec_so}), there is the minimal $l \in \lbrace 0, \ldots, [(k_1+k_2)/2]\rbrace$ such that $R_H(\mu)$ is a submodule of $\left. R_{G}((k_1+k_2-2l)\pi_1)\right|_H$.
If $l > 0$ then again by~(\ref{eqn_dec_so}) $R_H(\mu)$ is a submodule of $\left. \mathrm{S}^{k_1+k_2-2l} V \right|_H$ and hence $\lambda - s_1\delta_1 - s_2 \delta_2 \in \mathrm E_{H\times \FF^\times \times \FF^\times}(V^*)$ for some $s_1,s_2 \in \ZZ$ with $s_1+s_2 =2l$, which is impossible.
Thus $l = 0$ and $\overline \lambda \in \Gamma_{\lbrace 1 \rbrace}(G,H)$.

Conversely, take any element $(k\pi_1; \mu) \in \Gamma_{\lbrace 1 \rbrace}(G,H)$.
Then by~(\ref{eqn_dec_so}) $R_H(\mu)$ is a submodule of $\left. \mathrm S^k V \right|_H$, hence there are $k_1,k_2 \ge 0$ with $k_1+k_2 = k$ such that $\lambda = \mu + k_1\delta_1 + k_2 \delta_2 \in \mathrm E_{H \times \FF^\times \times \FF^\times}(V^*)$.
As $(k\pi_1;\mu) = \overline \lambda$, the proof is completed.
\end{proof}

\subsection{Case~\texorpdfstring{$I \ne \lbrace 1 \rbrace$}{I \textbackslash ne \{1\}}}
\label{subsec_RBM_Ine{1}}

For Case~\ref{sp_case_1.2} in Table~\ref{table_result_sympl} as well as for all cases appearing in Theorems~\ref{thm_result_so_odd} and~\ref{thm_result_so_even} the monoids $\Gamma_I(G,H)$ are computed as follows.

The rank of $\Gamma_I(G,H)$ is calculated using formula~(\ref{eqn_rank_of_RBM}).
For each triple $(G,H,I)$, the spherical module $(M,\mathfrak g / (\mathfrak p_I^- + \mathfrak h))$ is computed in the corresponding part of~\S\,\ref{sect_sympl_case} or~\S\,\ref{sect_orth_case} and
the rank of this module is calculated as described in~\S\,\ref{subsec_spherical_modules}.

Once the rank of $\Gamma_I(G,H)$ has been determined, the indecomposable elements of this monoid are found by using the following straightforward observation generalizing \cite[Propositions~4.9 and~4.10]{AvP2}.

\begin{proposition}
Let $I = \lbrace i_1,\ldots, i_k \rbrace \subset S$ and choose a nonzero tuple $(a_1,\ldots,a_k)$ of nonnegative integers.
Let $(\lambda_0; \mu_0) \in \Gamma_I(G,H)$ be an element such that $\lambda_0 = a_1 \pi_{i_1} + \ldots + a_k \pi_{i_k} \in \Lambda^+(G)$.
Let $J$ be the set of indecomposable elements of $\Gamma_I(G,H)$ having the form $(b_1\pi_{i_1} + \ldots + b_k \pi_{i_k}; *)$ for a nonzero tuple $(b_1,\ldots, b_k) \ne (a_1,\ldots, a_k)$ satisfying $b_1 \le a_1, \ldots, b_k \le a_k$.
Suppose that $(\lambda_0; \mu_0) \notin \ZZ^+ J$.
Then $(\lambda_0; \mu_0)$ is an indecomposable element of~$\Gamma_I(G,H)$.
\end{proposition}

In the situation of the above proposition, one successively computes all indecomposable elements of $\Gamma_I(G,H)$ of the form $(a_1\pi_{i_1} + \ldots + a_k \pi_{i_k}; *)$ first with $a_1 + \ldots + a_k = 1$, then with $a_1+ \ldots + a_ k = 2$, and so on until the required number of indecomposable elements has been found.

To implement the above algorithm for computing the indecomposable elements of $\Gamma_I(G,H)$, one should be able to compute explicitly the restriction to $H$ of any given representation $R_G(\lambda)$ with $\lambda \in \Lambda^+_I(G)$.
For each of the cases in Tables~\ref{table_result_sympl},~\ref{table_result_so_odd}, and~\ref{table_result_so_even}, the inclusion $G \supset H$ fits into a chain $G = H_0 \supset H_1 \supset \ldots \supset H_k = H$ where for each $i = 1,\ldots, k$ one of the following possibilities holds:
\begin{itemize}
\item
$H_i$ is a symmetric subgroup of~$H_{i-1}$;

\item
$H_i$ is a Levi subgroup of $H_{i-1}$;

\item
$H_{i-1} = \SO_7 \times K$ for some group $K$ and $H_i = \mathsf G_2 \times K$.
\end{itemize}
In the former two cases, the restrictions are computed using the information in~\cite[\S\S\,5.2,\,5.3]{AvP2}.
In the latter case, the restrictions are computed either via~\cite[Theorem~8, part~3]{AkP} or directly by using the program LiE~\cite{LiE1}.

\appendix

\section{Explicit embeddings \texorpdfstring{$\mathfrak{g}_2 \subset \mathfrak{so}_7$}{g\_2 \textbackslash subset so\_7} and \texorpdfstring{$\mathfrak{spin}_7 \subset \mathfrak{so}_8$}{spin\_7 \textbackslash subset so\_8}}
\label{sect_g2&spin7}

In this appendix, we present explicit realizations of the algebra $\mathfrak g_2$ as a subalgebra of $\mathfrak{so}_7$ and also of the algebra $\mathfrak{spin}_7$ as a subalgebra of~$\mathfrak{so}_8$.
These realizations are widely used in~\S\,\ref{sect_orth_case} for explicit calculations.

The algebra $\mathfrak{g}_2$ is realized as the subalgebra of $\mathfrak{so}_7$ consisting of all matrices of the form
\[
\begin{pmatrix}
t_1+t_2 & x_{10} & x_{11} & \sqrt2 x_{21} & x_{31} & x_{32} & 0 \\
y_{10} & t_1 & x_{01} & -\sqrt2 x_{11} & x_{21} & 0 & -x_{32} \\
y_{11} & y_{01} & t_2 & \sqrt2 x_{10} & 0 & -x_{21} & -x_{31} \\
\sqrt2 y_{21} & -\sqrt2 y_{11} & \sqrt2 y_{10} & 0 & -\sqrt2 x_{10} & \sqrt2 x_{11} & -\sqrt2 x_{21} \\
y_{31} & y_{21} & 0 & -\sqrt2 y_{10} & -t_2 & -x_{01} & -x_{11} \\
y_{32} & 0 & -y_{21} & \sqrt2 y_{11} & -y_{01} & -t_1 & -x_{10} \\
0 & -y_{32} & -y_{31} & -\sqrt2 y_{21} & -y_{11} & -y_{10} & -t_1-t_2
\end{pmatrix}.
\]

The algebra $\mathfrak{spin}_7$ is realized as the subalgebra of $\mathfrak{so}_8$ consisting of all matrices of the form
\[
\begin{pmatrix}
t_1+t_2+t_3 & x_{001} & x_{011} & x_{111} & x_{012} & x_{112} & x_{122} & 0 \\
y_{001} & t_1 & x_{010} & x_{110} & -x_{011} & x_{111} & 0 & -x_{122} \\
y_{011} & y_{010} & t_2 & x_{100} & x_{001} & 0 & -x_{111} & -x_{112} \\
y_{111} & y_{110} & y_{100} & t_3 & 0 & -x_{001} & x_{011} & -x_{012} \\
y_{012} & -y_{011} & y_{001} & 0 & -t_3 & -x_{100} & -x_{110} & -x_{111} \\
y_{112} & y_{111} & 0 & -y_{001} & -y_{100} & -t_2 & -x_{010} & -x_{011} \\
y_{122} & 0 & -y_{111} & y_{011} & -y_{110} & -y_{010} & -t_1 & -x_{001} \\
0 & -y_{122} & -y_{112} & -y_{012} & -y_{111} & -y_{011} & -y_{001} & -t_1-t_2-t_3
\end{pmatrix}.
\]

As a crucial property of these realizations, in both cases $\mathfrak k = \mathfrak g_2 \subset \mathfrak {so}_7$ and $\mathfrak k = \mathfrak{spin}_7 \subset \mathfrak{so}_8$ the set $\mathfrak b^+$ of all upper-triangular (and also the set $\mathfrak b^-$ of all lower-triangular) matrices in $\mathfrak k$ is a Borel subalgebra of~$\mathfrak k$ and the set $\mathfrak t$ of all diagonal matrices in~$\mathfrak k$ is a Cartan subalgebra of~$\mathfrak k$.

In the case $\mathfrak k = \mathfrak g_2 \subset \mathfrak {so}_7$, if $\alpha_1, \alpha_2 \in \mathfrak t^*$ are the two simple roots with respect to $\mathfrak b^+$ then for every positive root $i\alpha_1 + j\alpha_2$ the corresponding root subspace in $\mathfrak k$ is spanned by the matrix for which $x_{ij} = 1$ and all the other coordinates equal~$0$.
Similarly, the root subspace in $\mathfrak k$ corresponding to the negative root $-(i\alpha_1+j\alpha_2)$ is spanned by the matrix for which $y_{ij} = 1$ and all the other coordinates equal~$0$.

In the case $\mathfrak k = \mathfrak{spin}_7 \subset \mathfrak {so}_8$, if $\alpha_1, \alpha_2, \alpha_3 \in \mathfrak t^*$ are the three simple roots with respect to $\mathfrak b^+$ then for every positive root $i\alpha_1 + j\alpha_2 +k\alpha_3$ the corresponding root subspace in $\mathfrak k$ is spanned by the matrix for which $x_{ijk} = 1$ and all the other coordinates equal~$0$.
Similarly, the root subspace in $\mathfrak k$ corresponding to the negative root $-(i\alpha_1+j\alpha_2+k\alpha_3)$ is spanned by the matrix for which $y_{ijk} = 1$ and all the other coordinates equal~$0$.

%%%%%%%%%%%%%%%%%%%%%%%%%%%%%%%%%%%%%%%%%%%%%%%%%%%%%%%%%%%%%%%%%%%%%%%

\end{document}